\documentclass[12pt]{amsart}  

\usepackage{array} 

\usepackage{vmargin}
\usepackage{amssymb}
\usepackage{mathrsfs}
\usepackage[all]{xy}
\usepackage[usenames,dvipsnames]{color}
\usepackage{soul}
\RequirePackage[colorlinks,linkcolor=blue,citecolor=LimeGreen,urlcolor=red]{hyperref} 
\usepackage{amsmath}

\newtheorem{theorem}{Theorem}[section]

\newtheorem{lemma}[theorem]{Lemma}

\newtheorem{prop}[theorem]{Proposition}
\newtheorem{remark}[theorem]{Remark}

\numberwithin{equation}{section}

\newcommand{\R}{\mathbb{R}}
\newcommand{\C}{\mathbb{C}}
\newcommand{\Q}{\mathbb{Q}}
\newcommand{\N}{\mathbb{N}}
\newcommand{\Z}{\mathbb{Z}}

\newcommand{\T}{\mathbb{T}}

\newcommand{\function}[5]{\begin{array}[t]{lrcl}
#1 : & #2 & \longrightarrow & #3 \\
     & #4 & \longmapsto & #5
\end{array}}
\newcommand{\func}[3]{#1 : #2 \longrightarrow #3}

\newcommand{\abs}[1]{\left|#1\right|}
\newcommand{\eps}{\varepsilon}
\newcommand{\norm}[1]{\left\|#1\right\|}

\renewcommand{\leq}{\leqslant}
\renewcommand{\geq}{\geqslant}

\renewcommand{\bar}{\overline}

\renewcommand{\tilde}{\widetilde}

\def\signmb{\bigskip \begin{center} {\sc
Marc Briant\par\vspace{3mm}
CCA, University of Cambridge\par
DPMMS, Centre for Mathematical Sciences\par
Wilberforce Road,
Cambridge CB3 0WA,
UK\par\vspace{3mm}
e-mail:} \tt{m.j.briant@maths.cam.ac.uk} \end{center}}

\begin{document} 

\title[From Boltzmann to incompressible Navier-Stokes on the Torus]{From the Boltzmann equation to the incompressible Navier-Stokes equations on the torus: a quantitative error estimate}
\author{Marc Briant}
\thanks{The author was supported by the UK Engineering and Physical Sciences Research Council (EPSRC) grant EP/H023348/1 for the University of Cambridge Centre for Doctoral Training, the Cambridge Centre for Analysis.}

\begin{abstract}
We investigate the Boltzmann equation, depending on the Knudsen number, in the Navier-Stokes perturbative setting on the torus. Using hypocoercivity, we derive a new proof of existence and exponential decay for solutions close to a global equilibrium, with explicit regularity bounds and rates of convergence. These results are uniform in the Knudsen number and thus allow us to obtain a strong derivation of the incompressible Navier-Stokes equations as the Knudsen number tends to $0$. Moreover, our method is also used to deal with other kinetic models. Finally, we show that the study of the hydrodynamical limit is rather different on the torus than the one already proved in the whole space as it requires averaging in time, unless the initial layer conditions are satisfied.
\end{abstract}

\maketitle

\vspace*{10mm}



\smallskip
\textbf{Acknowledgements:} I would like to thank my supervisor, Cl\'ement Mouhot, for offering me to work on this problem and for the guidance he gave to me.

\tableofcontents

\section{Introduction} \label{sec:intro}

This paper deals with the Boltzmann equation in a perturbative setting as the Knudsen number tends to $0$. The latter equation describes the behaviour of rarefied gas particles moving on $\T^d$ (flat torus of dimension $d\geq 2$) with velocities in $\R^d$ when the only interactions taken into account are binary collisions. More precisely, the Boltzmann equation rules the time evolution of the distribution of particles in position and velocity. A formal derivation of the Boltzmann equation from Newton's laws under the rarefied gas assumption can be found in \cite{Ce}, while \cite{Ce1} present Lanford's Theorem (see \cite{La} and \cite{GST} for detailed proofs) which rigorously proves the derivation in short times.
\par We denote the Knudsen number by $\eps$ and the Boltzmann equation reads
\begin{eqnarray}
\partial_tf + v\cdot\nabla_xf &=&\frac{1}{\eps} Q(f,f) \:,\: \mbox{on} \: \T^d \times \R^d \nonumber
\\ &=&  \int_{\R^d\times \mathbb{S}^{d-1}}\Phi\left(|v - v_*|\right)b\left( \mbox{cos}\:\theta\right)\left[f'f'_* - ff_*\right]dv_*d\sigma, \label{Boltzmann}
\end{eqnarray}
where $f'$, $f_*$, $f'_*$ and $f$ are the values taken by $f$ at $v'$, $v_*$, $v'_*$ and $v$ respectively. Define:
$$\left\{ \begin{array}{rl}&\displaystyle{v' = \frac{v+v_*}{2} +  \frac{|v-v_*|}{2}\sigma} \vspace{2mm} \\ \vspace{2mm} &\displaystyle{v' _*= \frac{v+v_*}{2}  -  \frac{|v-v_*|}{2}\sigma} \end{array}\right., \: \mbox{and} \quad \mbox{cos}\:\theta = \left\langle \frac{v-v_*}{\abs{v-v_*}},\sigma\right\rangle .$$

One can find in \cite{Ce}, \cite{Ce1} or \cite{Go} that the global equilibria for the Boltzmann equation are the \textit{Maxwellians} $\mu(v)$. Without loss of generality we consider only the case of normalized Maxwellians:
$$\mu(v) = \frac{1}{(2\pi)^{\frac{d}{2}}}e^{-\frac{\abs{v}^2}{2}}.$$

The bilinear operator $Q(g,h)$ is given by 
$$Q(g,h) =  \int_{\R^d\times \mathbb{S}^{d-1}}\Phi\left(|v - v_*|\right)b\left( \mbox{cos}\theta\right)\left[h'g'_* - hg_*\right]dv_*d\sigma.$$


\subsection{The problem and its motivations} \label{subsec:problem}
The Knudsen number is the inverse of the average number of collisions for each particle per unit of time. Therefore, as reviewed in \cite{Vi}, one can expect a convergence from the Boltzmann model towards the acoustics and the fluid dynamics as the Knudsen number tends to $0$. This latter convergence will be specified. However, these different models describe physical phenomenon that do not act at the same scales in space or time. As suggested in previous studies, for instance \cite{Vi}\cite{Go}\cite{Sa},a rescaling in time and a perturbation of order $\eps$ around the global equilibrium $\mu(v)$ should approximate, as the Knudsen number tends to $0$, the incompressible Navier-Stokes regime.
\\We therefore study the following equation

\begin{equation}\label{Boltz}
\partial_tf_\eps + \frac{1}{\eps}v \cdot \nabla_xf_\eps =\frac{1}{\eps^2} Q(f_\eps,f_\eps) \:,\: \mbox{on} \: \T^d \times \R^d,
\end{equation}
under the linearization $f_\eps(t,x,v) = \mu(v) + \eps \mu^{1/2}(v) h_\eps(t,x,v)$. This leads to the perturbed Boltzmann equation
\begin{equation}\label{LinBoltz}
\partial_th_\eps + \frac{1}{\eps}v\cdot\nabla_xh_\eps = \frac{1}{\eps^2}L(h_\eps) + \frac{1}{\eps}\Gamma(h_\eps,h_\eps).
\end{equation}
that we will study thoroughly, and where we defined
$$\left\{\begin{array}{ll} L(h) &\displaystyle{= \left[Q(\mu,\mu^\frac{1}{2}h)+Q(\mu^\frac{1}{2}h,\mu)\right]\mu^{-\frac{1}{2}}} \vspace{2mm}
\\ \vspace{2mm} \Gamma(g,h) &\displaystyle{= \frac{1}{2}\left[Q(\mu^\frac{1}{2}g,\mu^\frac{1}{2}h)+Q(\mu^\frac{1}{2}h,\mu^\frac{1}{2}g)\right]\mu^{-\frac{1}{2}}.} \end{array}\right.$$

\bigskip
All along this paper we consider the Boltzmann equation with \textit{hard potential} or \textit{Maxwellian potential} ($\gamma=0$), that is to say there is a constant $C_\Phi >0$ such that
$$\Phi(z) = C_\Phi z^\gamma \:,\:\: \gamma \in [0,1].$$
We also assume a strong form of Grad's \textit{angular cutoff} \cite{Gr1}, expressed here by the fact that we assume $b$ to be $C^1$ with the following controls
$$\forall z \in [-1,1], \: \abs{b(z)}, \abs{b'(z)} \leq C_b,$$
$b$ and $\Phi$ being defined in equation $\eqref{Boltzmann}$.

\bigskip
The aim of the present article is to develop a constructive method to obtain existence and exponential decay for solutions to the perturbed Boltzmann equation $\eqref{LinBoltz}$, uniformly in the Knudsen number.
\par Such a uniform result is then used to derive explicit rates of convergence for $(h_\eps)_{\eps>0}$ towards its limit as $\eps$ tends to $0$. Thus proving and quantifying the convergence from the Boltzmann equation to the incompressible Navier-Stokes equations $\eqref{NS}$.


\subsection{Notations} \label{subsec:notations}

Throughout this paper, we use the following notations. For two multi-indexes $j$ and $l$ in $\N^d$ we define:
\begin{itemize}
\item $\partial_l^j = \partial_{v_j}\partial_{x_l}$,
\item for $i$ in $\{1,\dots,d\}$ we denote by $c_i(j)$ the $i^{th}$ coordinate of $j$,
\item the length of $j$ will be written $|j| = \sum_i c_i(j)$,
\item the multi-index $\delta_{i_0}$ by : $c_i(\delta_{i_0}) = 1$ if $i=i_0$ and $0$ elsewhere.
\end{itemize}
\par We work with the following definitions: $L^p_{x,v} = L^p\left(\T^d\times \R^d\right)$, $L^p_x = L^p\left(\T^d\right)$ and $L^p_v = L^p\left(\R^d\right)$. The Sobolev spaces $H^s_{x,v}$, $H^s_{x}$ and $H^s_{v}$ are defined in the same way and we denote the standard Sobolev norms by $\norm{\cdot}^2_{H^s_{x,v}} = \sum\limits_{|j|+|l|\leq s}\norm{\partial^j_l\cdot}^2_{L^2_{x,v}}$.


\subsection{Our strategy and results} \label{subsec:strategy}
The first step is to investigate the equation $\eqref{LinBoltz}$ in order to obtain existence and exponential decay of solutions close to equilibrium in Sobolev spaces $H^s_{x,v}$, independently of the Knudsen number $\eps$. Moreover, we want all the required smallness assumption on the initial data and rates of convergence to be explicit. Such a result has been proved in \cite{Gu4} by studying independently the behaviour of both microscopic and fluid parts of solutions to $\eqref{LinBoltz}$, we proposed here another method based on hypocoercivity estimates.

\bigskip
Our strategy is to build a norm on Sobolev spaces which is equivalent to the standard norm and which satisfies a Gr\"onwall type inequality. 
\par We first construct a functional on $H^s_{x,v}$ by considering a linear combination of $\norm{\partial^j_l\cdot}^2_{L^2_{x,v}}$, for all $|j|+|l|\leq s$, together with product terms of the form $\langle \partial^{\delta_i}_{l-\delta_i}\cdot,\partial^0_l\cdot\rangle _{L^2_{x,v}}$. The distortion of the standard norm by adding these mixed terms is necessary \cite{MN} in order to exhibit a relaxation, due to the hypocoercivity property of the linear part of the perturbed Boltzmann equation $\eqref{LinBoltz}$.
\par We then study the flow of this functional along time for solutions to the linearized Boltzmann equation $\eqref{LinBoltz}$. This flow is controlled by energy estimates and, finally, a non-trivial choice of coefficients in the functional yields an equivalence between the functional and the standard Sobolev norm, as well as a Gr\"onwall type inequality, both of them being independent of $\eps$.

\bigskip
This strategy is applied to the linear part of the equation to prove that it generates a strongly continuous semigroup with exponential decay (Theorem $\ref{lin}$). We then combine the latter method  and some orthogonal property of the remainder and apply it to the full nonlinear model(Proposition $\ref{apriori}$). This estimate enables us to prove the existence of solutions to the Cauchy problem and their exponential decay as long as the initial data is small enough, with a smallness independent of $\eps$ (Theorem $\ref{perturb}$). We emphasize here that, thanks to the functional we used, the smaller $\eps$ the less control is needed on the $v$-derivatives of the initial data.
\par However, these results seem to tell us that the $v$-derivatives of solutions to equation $\eqref{LinBoltz}$ can blow-up as $\eps$ tends to $0$. The last step is thus to create a new functional, based on the microscopic part of solutions (idea first introduced by Guo \cite{Gu4}), satisfying the same properties but controlling the $v$-derivatives as well. The control on the microscopic part of solutions to equation $\eqref{LinBoltz}$ is due to the deep structure of the linear operator $L$. This leads to the expected exponential decay independently of $\eps$ even for the $v$-derivatives (Theorem $\ref{decaydv}$).

\bigskip
Finally, the chief aim of the present article is to derive explicit rates of convergence from solutions to the perturbed Boltzmann equation to solutions to the incompressible Navier-Stokes equations.
\par Theorem $\ref{perturb}$ tells us that for all $\eps$ we can build a solution $h_\eps$ to the perturbed Boltzmann equation $\eqref{LinBoltz}$, as long as the initial perturbation is sufficiently small, independently of $\eps$. We can then consider the sequence $(h_\eps)_{0<\eps\leq 1}$ and study its limit. It appears that it converges weakly in $L^\infty_tH^s_xL^2_v$, for $s \geq s_0 > d$, towards a function $h$. Furthermore, we have the following form for $h$ (see \cite{BU})
$$h(t,x,v) = \left[\rho(t,x) + v.u(t,x) + \frac{1}{2}(\abs{v}^2-d)\theta(t,x)\right]\mu(v)^{1/2},$$
of which physical observables are weak solutions, in the Leray sense \cite{Le}, of the incompressible Navier-Stokes equations ($p$ being the pressure function, $\nu$ and $\kappa$ being constants determined by $L$, see Theorem $5$ in \cite{Go})

\begin{eqnarray}
\partial_t u - \nu \Delta u + u\cdot \nabla u + \nabla p = 0, \nonumber
\\ \nabla \cdot u = 0, \label{NS}
\\ \partial_t \theta - \kappa \Delta \theta + u\cdot \nabla \theta = 0, \nonumber
\end{eqnarray}

together with the Boussinesq relation

\begin{equation}\label{Boussinesq}
\nabla(\rho + \theta) = 0.
\end{equation}

\bigskip
We conclude by studying the properties of the hydrodynamical convergence thanks to the Fourier transform on the torus of the linear operator $L-v\cdot\nabla_x$. This gives us a strong convergence result on the time average of $h_\eps$ with an explicit rate of convergence in finite time. An interpolation between this finite time convergence and the exponential stability of the global equilibria for Boltzmann equation as well as for Navier-Stokes equations gives a strong convergence for all times (Theorem $\ref{hydrolim}$). We obtain an explicit form for the initial data associated to the limit of $(h_\eps)_{\eps>0}$.


\subsection{Comparison with existing results}\label{subsec:comparisonresults}
For physical purposes, one may assume that $\eps=1$ which is a mere normalization and that is why many articles about the perturbed Boltzmann equation only deal with this case. The associated Cauchy problem has been worked on over the past fifty years, starting with Grad \cite{Gr}, and it has been studied in different spaces, such as weighted  $L^2_v(H^l_x)$ spaces \cite{Uk} or weighted Sobolev spaces \cite{Gu1}\cite{Gu2}\cite{Yu}. Other results have also been proved in $\R^d$ instead of the torus, see for instance \cite{AMUXY}\cite{Ch}\cite{NI}, but it will not be the purpose of this article.
\par Our article explicitly deals with the general case for $\eps$ and we prove results that are uniform in $\eps$. To solve the Cauchy problem we use an iterative scheme, like in the papers mentioned above, but our strategy yields a condition for the existence of solutions in $H^s_{x,v}$ which is uniform in $\eps$ (Theorem $\ref{globalexist}$). In order to obtain such a result, we had to consider more precise estimates on the bilinear operator $\Gamma$. Bardos and Ukai \cite{BU} obtained a similar result in $\R^d$ but in weighted Sobolev spaces and did not prove any decay.

\bigskip
The behaviour of such global in time solutions has also been studied. Guo worked in weighted Sobolev spaces and proved the boundedness of solutions to equation $\eqref{LinBoltz}$ \cite{Gu1}, as well as an exponential decay (uniform in $\eps$) \cite{Gu4}. The norm involved in \cite{Gu1}\cite{Gu4} is quite intricate and requires a lot of technical computations. To avoid specific and technical calculations, the theory of hypocoercivity \cite{Mo} focuses on the properties of the Boltzmann operator and which are quite similar to hypoellipticity. This theory has been used in \cite{MN} to obtain exponential decay in standard Sobolev spaces in the case $\eps =1$.
\par We use the idea of Mouhot and Neumann developed in \cite{MN} consisting of considering a functional on $H^s_{x,v}$ involving mixed scalar products. In this article we thus construct such a quadratic form, but with coefficient depending on $\eps$. Working in the general case for $\eps$ yields new calculations and requires the use of certain orthogonal properties of the bilinear operator $\Gamma$ to overcome these issues. Moreover, we construct a new norm out of this functional, which controls the $v$-derivatives by a factor $\eps$.
\par The fact that the study yields a norm containing some $\eps$ factors prevents us from having a uniform exponential decay for the $v$-derivatives. We use the idea of Guo \cite{Gu4} and look at the microscopic part of the solution $h_\eps$ everytime we look at a differentiation in $v$. This idea catches the interesting structure of $L$ on its orthogonal part. Combining this idea with our previous strategy fills the gap for the $v$-derivatives.

\bigskip
Several studies have been made on the different regimes of hydrodynamical limits for the Boltzmann equation and complete formal derivations have then been obtained by Bardos, Golse and Levermore \cite{BGL}. We refer the reader to \cite{Sa} for an overview on the existing results and standard techniques. The particular case of incompressible Navier-Stokes regime has been first addressed by Sone \cite{Sone} where he dealt with the asymptotic theory for the perturbed equation up to the second order inside a smooth domain. Later, De Masi, Esposito and Lebowitz \cite{DMEL} gave a first rigorous and constructive proof on the torus by considering the stability of Maxwellians whose mean velocity is a solution to incompressible Navier-Stokes equations. Note that their result is of a different nature than the one presented here. Let us also mention the works of Golse and Saint-Raymond \cite{GolSt1}\cite{GolSt2} (in $\R^3$) and Levermore and Masmoudi \cite{LevMas} (on $\T^d$) where the convergence is proved for appropriately scaled Di Perna-Lions renormalized solutions \cite{DL}.
\par Our uniform results enable us to derive a weak convergence in $H^s_xL^2_v$ towards solutions to the incompressible Navier-Stokes equations, together with the Boussinesq relation. We then find a way to obtain strong convergence using the ideas of  the Fourier study of the linear operator $L -v.\nabla_x$, developed in \cite{BU} and \cite{EP}, combined with Duhamel formula. However, the study done in \cite{BU} relies strongly on an argument of stationary phase developed in \cite{Uk1} which is no longer applicable in the torus. Indeed, the Fourier space of $\R^d$ is continuous and so integration by parts can be used in that frequency space. This tool is no longer available in the frequency space of the torus which is discrete.
\par Theorem $\ref{hydrolim}$ shows that the behaviour of the hydrodynamical limit is quite different on the torus, where an averaging in time is necessary for general initial data. However,we obtain the same relation between the limit at $t=0$ and the initial perturbation $h_{in}$ and also the existence of an initial layer. That is to say that we have a convergence in $L^2_{[0,T]}=L^2([0,T])$ if and only if the initial perturbation satisfies some physical properties, which appear to be the same as in $\R^d$ studied in \cite{BU}.
\par This convergence gives a perturbative result for incompressible Navier-Stokes equations in  Sobolev spaces around the steady solution. The regularity of the weak solutions we constructed implies that they are in fact strong solutions (see Lions \cite{Li1}, Section $2.5$, and more precisely Serrin, \cite{Se1} and \cite{Se2}). Moreover, our uniform exponential decay for solutions to the linearized Boltzmann equation yields an exponential decay for the perturbative solutions of the incompressible Navier-Stokes equations in higher Sobolev spaces. Such an exponential convergence to equilibrium has been derived in $H^1_0$ for $d=2$ or $d=3$ in \cite{Te}, or can be deduced from Proposition $3.7$ in \cite{BerMaj} in higher Sobolev spaces for small initial data. The general convergence to equilibrium can be found in \cite{MatNish} (small initial data) and in \cite{NovStra} but they focus on the general compressible case and no rate of decay is built.

\bigskip
Furthermore, results that do not involve hydrodynamical limits (existence and exponential decay results) are applicable to a larger class of operators. In Appendix $\ref{appendix:validation}$ we prove that those theorems also hold for other kinetic collisional models such as the linear relaxation, the semi-classical relaxation, the linear Fokker-Planck equation and the Landau equation with hard and moderately soft potential.


\subsection{Organization of the paper}\label{subsec:organization}

Section $\ref{sec:mainresults}$ is divided in two different subsections.
\par As mentionned above, we shall use the hypocoercivity of the Boltzmann equation $\eqref{Boltzmann}$. This hypocoercivity can be described in terms of technical properties on $L$ and $\Gamma$ and, in order to obtain more general results, we consider them as a basis of our paper. Thus, subsection $\ref{subsec:hypoco}$ describes them in detail and a proof of the fact that $L$ and $\Gamma$ indeed satisfy those properties is given in Appendix $\ref{appendix:validation}$. Most of them have been proved in \cite{MN} but we require more precise ones to deal with the general case.
\par The second subsection $\ref{subsec:statementresults}$ is dedicated to a mathematical formulation of the results described in subsection $\ref{subsec:strategy}$.

\bigskip
As said when we described our strategy (subsection $\ref{subsec:strategy}$), we are going to study the flow of a functional involving $L^2_{x,v}$-norm of $x$ and $v$ derivatives and mixed scalar products. To control this flow in time we compute energy estimates for each of these terms in a toolbox (section $\ref{sec:toolbox}$) which will be used and referred to all along the rest of the paper. Proofs of those energy estimates are given in Appendix $\ref{appendix:toolbox}$.

\bigskip
Finally, sections $\ref{sec:linear}$, $\ref{sec:apriori}$, $\ref{sec:perturbresult}$, $\ref{sec:decaydv}$ and $\ref{sec:hydrolim}$ are the proofs respectively of Theorem $\ref{lin}$ (about the strong semigroup property of the linear part of equation $\eqref{LinBoltz}$), Proposition $\ref{apriori}$ (an a priori estimates on the constructed functional for the full model), Theorem $\ref{perturb}$ (existence and exponential decay of solutions to equation $\eqref{LinBoltz}$), Theorem $\ref{decaydv}$ (showing the uniform boundedness of the $v$-derivatives) and of Theorem $\ref{hydrolim}$ (dealing with the hydrodynamical limit).
\par We notice here that section $\ref{sec:perturbresult}$ is divided in two subsection. Subsection $\ref{subsec:cauchy}$ deals with the existence of solutions for all $\eps > 0$ and subsection $\ref{subsec:expodecay}$ proved the exponential decay of those solutions.

\section{Main results} \label{sec:mainresults}

This section is divided into two parts. The first one translates the hypocoercivity aspects of the Boltzmann linear operator in terms of mathematical properties for $L$ and $\Gamma$. Then, the second one states our results in terms of those assumptions.


\subsection{Hypocoercivity assumptions}\label{subsec:hypoco}

This section is dedicated to describing the framework and assumptions of the hypocoercivity theory. A state of the art of this theory can be found in \cite{Mo}.

 
\subsubsection{Assumptions on the linear operator $L$}


\paragraph{\textbf{Assumptions in $H^1_{x,v}$}}:
\par \textbf{\underline{(H1): Coercivity and general controls}}
\\$\func{L}{L^2_{v}}{L^2_{v}}$ is a closed and self-adjoint operator with $L = K - \Lambda$ such that:
\begin{itemize}
\item $\Lambda$ is coercive: \begin{itemize}
                              \item it exists $\norm{.}_{\Lambda_v}$ norm on $L^2_v$ such that 
                              $$\forall h \in L^2_v \:,\:\nu_0^\Lambda\norm{h}^2_{L^2_v} \leq \nu_1^\Lambda\norm{h}^2_{\Lambda_v} \leq \langle\Lambda(h),h\rangle_{L^2_v} \leq \nu_2^\Lambda\norm{h}^2_{\Lambda_v},$$
                              \item $\Lambda$ has a defect of coercivity regarding its $v$ derivatives:$$\forall h \in H^1_v \:,\:\langle\nabla_v\Lambda(h),\nabla_vh\rangle_{L^2_v} \geq \nu_3^\Lambda\norm{\nabla_vh}^2_{\Lambda_v} - \nu_4^\Lambda\norm{h}^2_{\Lambda_v}.$$
                              \end{itemize}
\item There exists $C^L>0$ such that $$\forall h \in L^2_v \:,\:\forall g \in L^2_v \:,\:\langle L(h),g\rangle_{L^2_v} \leq C^L\norm{h}_{\Lambda_v}\norm{g}_{\Lambda_v},$$
\end{itemize}

\noindent where $(\nu_s^\Lambda)_{1\leq s \leq 4}$ are strictly positive constants depending on the operator and the dimension of the velocities space $d$.
\\As in \cite{MN}, we define a new norm on $L^2_{x,v}$: $$\norm{.}_\Lambda = \norm{\norm{.}_{\Lambda_v}}_{L^2_{x}}.$$

\par\textbf{\underline{(H2): Mixing property in velocity}}
$$\forall \delta>0 \:,\: \exists C(\delta) >0 \:,\: \forall h \in H^1_v \:,\quad \langle\nabla_vK(h),\nabla_vh\rangle_{L^2_v} \leq C(\delta)\norm{h}^2_{L^2_v} + \delta \norm{\nabla_vh}^2_{L^2_v}.$$

\par \textbf{\underline{(H3): Relaxation to equilibrium}}
\\We suppose that the kernel of $L$ is generated by $N$ functions which form an orthonormal basis for $\mbox{Ker}(L)$:
$$\mbox{Ker}(L) = \mbox{Span}\{\phi_1(v), \dots, \phi_N(v)\}.$$
Moreover, we assume that the $\phi_i$ are of the form $P_i(v)e^{-|v|^2/4}$, where $P_i$ is a polynomial.
\par Furthermore, denoting by $\pi_L$ the orthogonal projector in $L^2_v$ on $\mbox{Ker}(L)$ we assume that we have the following local coercivity property:

$$\exists \lambda >0 \:,\:\forall h \in L^2_v \:,\quad \langle L(h),h\rangle_{L^2_v} \leq -\lambda \norm{h^\bot}^2_{\Lambda_v},$$
\noindent where $h^\bot = h- \pi_L(h)$ denotes the microscopic part of $h$ (the orthogonal to $\mbox{Ker}(L)$ in $L^2_v$).

\bigskip
We are using the same hypothesis as in \cite{MN}, except that we require the $\phi_i$ to be of a specific form. This additional requirement allows us to derive properties on the $v$-derivatives of $\pi_L$ that we will state in the toolbox section $\ref{sec:toolbox}$.
\par Then we have two more properties on $L$ in order to deal with higher order Sobolev spaces.

\bigskip
\paragraph{\textbf{Assumptions in $H^s_{x,v}$, $s>1$}}:
\par\textbf{\underline{(H1'): Defect of coercivity for higher derivatives}}
\\ We assume that $L$ satisfies $(H1)$ along with the following property: for all $s \geq 1$, for all $|j|+|l| = s$ such that $\abs{j}\geq 1$,

$$\forall h \in H^s_{x,v}\:,\quad \langle\partial^j_l\Lambda(h),\partial^j_lh\rangle_{L^2_{x,v}}\geq \nu_5^\Lambda\norm{\partial^j_lh}_\Lambda^2 - \nu^\Lambda_6\norm{h}_{H^{s-1}_{x,v}},$$

\noindent where $\nu_5^\Lambda$ and $\nu_6^\Lambda$ are strictly positive constants depending on $L$ and $d$.
\\We also define a new norm on $H^s_{x,v}$: $$\norm{.}_{H^s_\Lambda} = \left(\sum\limits_{|j|+|l|\leq s}\norm{\partial_l^j.}^2_\Lambda\right)^{1/2}.$$

\par\textbf{\underline{(H2'): Mixing properties}}
\\As above, Mouhot and Neumann extended the hypothesis $(H2)$ to higher Sobolev's spaces: for all $s \geq 1$, for all $|j|+|l| = s$ such that $\abs{j}\geq 1$,
$$\forall \delta>0 \:,\:\exists C(\delta)>0\:,\:\forall h \in H^s_{x,v} \:,\quad \langle\partial^j_lK(h),\partial^j_lh\rangle_{L^2_{x,v}}\leq C(\delta)\norm{h}^2_{H^{s-1}_{x,v}} + \delta\norm{\partial^j_lh}^2_{L^2_{x,v}}.$$


\subsubsection{Assumptions on the second order term $\Gamma$}
 To solve our problem uniformly in $\eps$ we had to precise the hypothesis made in \cite{MN} in order to have a deeper understanding of the operator $\Gamma$. This lead us to two different assumptions.
 
\textbf{\underline{(H4): Control on the second order operator}}
\\$\func{\Gamma}{L^2_{v}\times L^2_{v}}{L^2_{v}}$ is a bilinear symmetric operator such that for all multi-indexes $j$ and $l$ such that $|j|+|l| \leq s$, $s\geq 0$,

$$
\left|\langle\partial_l^j\Gamma(g,h),f\rangle_{L^2_{x,v}}\right| \leq \left\{\begin{array}{lr} \mathcal{G}^s_{x,v}(g,h)\norm{f}_{\Lambda}&, \:\: \mbox{if } j \neq 0 \vspace{1mm} 
\\\vspace{1mm} \mathcal{G}^s_{x}(g,h)\norm{f}_{\Lambda}&, \:\: \mbox{if } j = 0 \end{array}\right.,
$$

\noindent$\mathcal{G}^s_{x,v}$ and $\mathcal{G}^s_x$ being such that $\mathcal{G}^s_{x,v}\leq \mathcal{G}^{s+1}_{x,v}$, $\mathcal{G}^s_{x}\leq \mathcal{G}^{s+1}_{x}
$ and satisfying the following property:
\begin{equation*}
\exists s_0 \in \N \:,\: \forall s \geq s_0 \:,\: \exists C_\Gamma>0 \:,\quad \left \{\begin{array}{rl} \mathcal{G}^s_{x,v}(g,h) &\leq C_\Gamma\left(\norm{g}_{H^s_{x,v}}\norm{h}_{H^s_\Lambda} + \norm{h}_{H^s_{x,v}}\norm{g}_{H^s_\Lambda}\right) \vspace{1mm}
\\ \vspace{1mm} \mathcal{G}^s_{x}(g,h) &\leq C_\Gamma\left(\norm{h}_{H^s_{x}L^2_v}\norm{g}_{H^s_\Lambda} + \norm{g}_{H^s_{x}L^2_v}\norm{h}_{H^s_\Lambda}\right).\end{array}\right.
\end{equation*}

\textbf{\underline{(H5): Orthogonality to the Kernel of the linear operator}}
$$\forall h, \:g \in \mbox{Dom}(\Gamma) \cap L^2_{v} \:,\quad \Gamma(g,h) \in \mbox{Ker}(L)^\bot.$$


\subsection{Statement of the Theorems}\label{subsec:statementresults}

\subsubsection{Uniform result for the linear Boltzmann equation}

For $s$ in $\N^*$ and some constants $(b_{j,l}^{(s)})_{j,l}$, $(\alpha_l^{(s)})_l$ and $(a_{i,l}^{(s)})_{i,l}$ strictly positive and $0 < \eps\leq 1$ we define the following functional on $H^s_{x,v}$, where we emphasize that there is a dependance on $\eps$, which is the key point of our study:
$$\norm{\cdot}_{\mathcal{H}^s_\eps} = \left[ \sum\limits_{\overset{|j|+|l|\leq s}{|j|\geq 1}}b_{j,l}^{(s)}\eps^2\norm{\partial^j_l\cdot}^2_{L^2_{x,v}} + \sum\limits_{|l|\leq s}\alpha_l^{(s)}\norm{\partial^0_l\cdot}^2_{L^2_{x,v}} + \sum\limits_{\overset{|l|\leq s}{i,c_i(l)> 0}}a_{i,l}^{(s)}\eps \langle \partial^{\delta_i}_{l-\delta_i}\cdot,\partial^0_l\cdot\rangle _{L^2_{x,v}}\right]^\frac{1}{2}.$$

\bigskip
We first study the perturbed equation $\eqref{LinBoltz}$, without taking into account the bilinear remainder operator. By letting $\pi_w$ be the projector in $L^2_{x,v}$ onto $\mbox{Ker}(w)$ we obtained the following semigroup property for $L$.

\begin{theorem}\label{lin}
If $L$ is a linear operator satisfying the conditions $\emph{(H1')}$, $\emph{(H2')}$ and $\emph{(H3)}$ then it exists $0 < \eps_d\leq 1$ such that for all $s$ in $\N^*$,
\begin{enumerate}
\item for all $ 0 < \eps\leq \eps_d \:,\: G_\eps = \eps^{-2}L - \eps^{-1}v \cdot\nabla_x$ generates a $C^0$-semigroup on $H^s_{x,v}$.
\item there exist $C_G^{(s)},\: (b_{j,l}^{(s)}),\:(\alpha_l^{(s)}),\:(a_{i,l}^{(s)}) >  0 $  such that for all $0<  \eps \leq \eps_d$:
	$$\norm{\cdot}^2_{\mathcal{H}^s_{\eps}} \sim \left(\norm{\cdot}^2_{L^2_{x,v}}+\sum\limits_{|l|\leq s}\norm{\partial_l^0\cdot}^2_{L^2_{x,v}}+\eps^2\sum\limits_{\overset{|l|+|j|\leq s}{|j|\geq 1}}\norm{\partial_l^j\cdot}^2_{L^2_{x,v}}\right),$$
	and for all h in $H^s_{x,v}$, $$\langle G_\eps(h),h\rangle _{\mathcal{H}^s_\eps} \leq -C_G^{(s)}\norm{h-\pi_{G_\eps}(h))}^2_{H^s_\Lambda}.$$
\end{enumerate}
\end{theorem}

This theorem gives us an exponential decay for the semigroup generated by $G_\eps$.


\subsubsection{Uniform perturbative result for the Boltzmann equation}
The next result states that if we add the bilinear remainder operator then it is enough, if $\eps$ is small enough, to slightly change our new norm to have a control on the solution.

\begin{prop}\label{apriori}
If $L$ is a linear operator satisfying the conditions $\emph{(H1')}$, $\emph{(H2')}$ and $\emph{(H3)}$ and $\Gamma$ a bilinear operator satisfying $\emph{(H4)}$ and $\emph{(H5)}$ then it exists $0 < \eps_d\leq 1$ such that for all $s$ in $\N^*$,
\begin{enumerate}
\item there exist $K^{(s)}_0,\:K^{(s)}_1,\: K^{(s)}_2\: (b_{j,l}^{(s)}),\:(\alpha_l^{(s)}),\:(a_{i,l}^{(s)}) >  0 $, independent of $\Gamma$ and $\eps$,  such that for all $0<  \eps \leq \eps_d$:
$$\norm{\cdot}^2_{\mathcal{H}^s_{\eps}} \sim \left(\norm{\cdot}^2_{L^2_{x,v}}+\sum\limits_{|l|\leq s}\norm{\partial_l^0\cdot}^2_{L^2_{x,v}}+\eps^2\sum\limits_{\overset{|l|+|j|\leq s}{|j|\geq 1}}\norm{\partial_l^j\cdot}^2_{L^2_{x,v}}\right),$$
\item and for all $h_{in}$ in $H^s_{x,v}\cap \emph{\mbox{Ker}}(G_\eps)^\bot$ and all $g$ in $\emph{\mbox{Dom}}(\Gamma)\cap H^s_{x,v}$, if we have a solution $h$ in $H^s_{x,v}$ to the following equation 
$$\partial_t h + \frac{1}{\eps}v\cdot\nabla_x h = \frac{1}{\eps^2}L(h) + \frac{1}{\eps}\Gamma(g,h),$$
then
 $$\frac{d}{dt}\norm{h}_{\mathcal{H}_\eps^s}^2 \leq - K^{(s)}_0 \norm{h}_{H^s_\Lambda}^2 + K^{(s)}_1 \left(\mathcal{G}^s_{x}(g,h)\right)^2 + \eps^2K^{(s)}_2\left(\mathcal{G}^s_{x,v}(g,h)\right)^2  .$$
\end{enumerate}
\end{prop}

One can remark that the norm constructed above leaves the $x$-derivatives free while it controls the $v$-derivatives by a factor $\eps$. 
\par We emphasize that this result shows that the derivative of the norm is control by the $x$-derivatives of $\Gamma$ and the Sobolev norm of $\Gamma$, but weakened by a factor $\eps^2$. This is important as our norm $\norm{.}^2_{\mathcal{H}^s_{\eps}}$ controls the $L^2_v(H^s_x)$-norm by a factor of order $1$ whereas it controls the whole $H^s_{x,v}$-norm by a multiplicative factor of order $1/\eps$.

\begin{theorem}\label{perturb}
Let $Q$ be a bilinear operator such that:
\begin{itemize}
\item the equation $\eqref{Boltz}$ admits an equilibrium $0 \leq \mu \in L^1(\T^d\times\R^d)$,
\item the linearized operator $L = L(h)$ around $\mu$ with the scaling $f = \mu + \eps \mu^{1/2} h$ satisfies $\emph{(H1')}$, $\emph{(H2')}$ and $\emph{(H3)}$,
\item the bilinear remaining term $\Gamma = \Gamma(h,h)$ in the linearization satisfies $\emph{(H4)}$ and $\emph{(H5)}$.
\end{itemize}
Then there exists $0<\eps_d\leq 1$ such that for any $s \geq s_0$ (defined in $\emph{(H4)}$ ), 
\begin{enumerate}
\item there exist $(b_{j,l}^{(s)}),\:(\alpha_l^{(s)}),\:(a_{i,l}^{(s)}) >  0 $, independent of $\Gamma$ and $\eps$,  such that for all $0<  \eps \leq \eps_d$:
$$\norm{\cdot}^2_{\mathcal{H}^s_{\eps}} \sim \left(\norm{\cdot}^2_{L^2_{x,v}}+\sum\limits_{|l|\leq s}\norm{\partial_l^0\cdot}^2_{L^2_{x,v}}+\eps^2\sum\limits_{\overset{|l|+|j|\leq s}{|j|\geq 1}}\norm{\partial_l^j\cdot}^2_{L^2_{x,v}}\right),$$
\item  there exist $\delta_s > 0$, $C_s > 0$ and $\tau_s >0$ such that for all $0<\eps\leq \eps_d$:
\end{enumerate}
For any distribution $0\leq f_{in} \in L^1(\T^d\times\R^d)$ with $f_{in} = \mu + \eps \mu^{1/2} h_{in}\geq 0$, $h_{in}$ in $\emph{\mbox{Ker}}(G_\eps)^\bot$ and 
$$\norm{h_{in}}_{\mathcal{H}^s_\eps} \leq \delta_s,$$
there exists a unique global smooth (in $H^s_{x,v}$, continuous in time) solution $f_\eps = f_\eps(t,x,v)$ to $\eqref{Boltz}$ which, moreover, satisfies $f_\eps = \mu + \eps \mu^{1/2} h_\eps \geq 0$ with:
$$\norm{h_\eps}_{\mathcal{H}^s_\eps} \leq \norm{h_{in}}_{\mathcal{H}^s_\eps} e^{-\tau_s t}.$$
\end{theorem}

The fact that we are asking $h_{in}$ to be in $\mbox{Ker}(G_\eps)^\bot$ just states that we want $f_{in}$ to have the same physical quantities as the global equilibrium $\mu$. This is a compulsory requirement as one can easily check that the physical quantities 
$$\int_{\T^d\times\R^d} f_\eps(x,v) dxdv ,\quad \int_{\T^d\times\R^d} vf_\eps(x,v) dxdv, \quad\int_{\T^d\times\R^d} \abs{v}^2 f_\eps(x,v) dxdv$$
are preserved with time (see \cite{Ce1} for instance).
\par Notice that the $\mathcal{H}^s_{\eps}$-norm is this theorem is the same than the one we constructed in Proposition $\ref{apriori}$.

\subsubsection{The boundedness of the $v$-derivatives}

As a corollary we have that the $H^s_x(L^2_v)$-norm decays exponentially independently of $\eps$ but that the only control we have on the $H^s_{x,v}$ is
$$\norm{h_\eps}_{H^s_{x,v}} \leq \frac{\delta_s}{\eps} e^{-\tau_s t}.$$
This seems to tell us that the $v$-derivatives can blow-up at a rate $1/\eps$. However, Guo \cite{Gu4} showed that one can prove that there is no explosion if one controls independently the fluid part and the microscopic part of the solution. This idea, combined with our original one, leads to the construction of a new norm which will only control the microscopic part of the solution whenever we face a derivative in the $v$ variable.
\par We define the following positive quadratic form
$$\norm{\cdot}^2_{\mathcal{H}^s_{\eps\bot}} = \sum\limits_{\overset{|j|+|l|\leq s}{|j|\geq 1}}b_{j,l}^{(s)}\norm{\partial^j_l(\mbox{Id} - \pi_L)}^2_{L^2_{x,v}} + \sum\limits_{|l|\leq s}\alpha_l^{(s)}\norm{\partial^0_l\cdot}^2_{L^2_{x,v}} + \sum\limits_{\overset{|l|\leq s}{i,c_i(l)> 0}}a_{i,l}^{(s)}\eps \langle \partial^{\delta_i}_{l-\delta_i}\cdot,\partial^0_l\cdot\rangle _{L^2_{x,v}}.$$

\begin{theorem}\label{decaydv}
Under the same conditions as in Theorem $\ref{perturb}$, for all $s\geq s_0$, there exist $(b_{j,l}^{(s)}),\:(\alpha_l^{(s)}),\:(a_{i,l}^{(s)}) >  0$ and $0<\eps_d\leq 1$ such that for all $0< \eps \leq \eps_d$:
\begin{enumerate}
\item $\norm{\cdot}_{\mathcal{H}^s_{\eps\bot}}  \sim \norm{\cdot}_{H^s_{x,v}}$, independently of $\eps$,
\item if $h_\eps$ is a solution of $\ref{LinBoltz}$ in $H^s_{x,v}$ with $\norm{h_{in}}_{\mathcal{H}^s_{\eps \bot}} \leq \delta'_s$ then
$$\norm{h_\eps}_{\mathcal{H}^s_{\eps \bot}} \leq \delta'_s e^{-\tau'_s t},$$
where $\delta_s'$ and $\tau_s'$ are strictly positive constants independent of $\eps$.
\end{enumerate}
\end{theorem}
This theorem builds up a functional that is equivalent to the standard Sobolev norm, independently of $\eps$. Thus, it yields an exponential decay for the $v$-derivatives as well as for the $x$-derivatives. However, the distorted norm used in Theorem $\ref{perturb}$ asked less control on the $v$-derivatives for the initial data, suggesting that, in the limit as $\eps$ goes to zero, almost only the $x$-variables have to be controlled.


\subsubsection{The hydrodynamical limit on the torus for Maxwellian particles}

Our theorem states that one can really expect a convergence of solutions of collisional kinetic models near equilibrium towards a solution of fluid dynamics equations. Indeed, the smallness assumption on the initial perturbation does not depend on the parameter $\eps$ as long as $\eps$ is small enough.
\par We then define the following macroscopic quantities
\begin{itemize}
\item the particles density $\rho_\eps(t,x) = \langle \mu(v)^{1/2},h_\eps(t,x,v)\rangle_{L^2_v},$
\item the mean velocity $u_\eps(t,x)  = \langle v\mu(v)^{1/2},h_\eps(t,x,v)\rangle_{L^2_v},$
\item the temperature $\displaystyle{\theta_\eps(t,x) = \frac{1}{d}\langle (\abs{v}^2-d)\mu(v)^{1/2},h_\eps(t,x,v)\rangle_{L^2_v}}.$
\end{itemize}

\bigskip
The theorem $\ref{perturb}$ tells us that, for $s\geq s_0$, the sequence $(h_\eps)_{\eps>0}$ converges (up to an extraction) weakly-* in $L^\infty_t(H^s_lL^2_v)$ towards a function $h$. Such a weak convergence enables us to use the Theorem $1.1$ of \cite{BU}, which is a slight modification of the result in \cite{BGL} to get that
\begin{enumerate}
\item $h$ is in $\mbox{Ker}(L)$, so of the form $$h(t,x,v) = \left[\rho(t,x) + v.u(t,x) + \frac{1}{2}(\abs{v}^2-d)\theta(t,x)\right]\mu(v)^{1/2},$$
\item $(\rho_\eps,u_\eps,\theta_\eps)$ converges weakly* in $L^\infty_t(H^s_x)$ towards $(\rho,u,\theta),$
\item $(\rho,u,\theta)$ satisfies the incompressible Navier-Stokes equations $\eqref{NS}$ as well as the Boussinesq equation $\eqref{Boussinesq}$.
\end{enumerate}
\par If such a result confirms the fact that one can derive the incompressible Navier-Stokes equations from the Boltzmann equation, it does unfortunately neither give us the continuity of $h$ nor the initial condition verified by $(\rho,u,\theta)$, depending on $(\rho_{in},u_{in},\theta_{in})$, macroscopic quantities associated to $h_{in}$. Our next, and final step, is therefore to link the last two triplets and so to understand the convergence $h_\eps \to h$ more deeply. This is the purpose of the following theorem.

\begin{theorem}\label{hydrolim}
Consider $s \geq s_0$ and $h_{in}$ in $H^s_{x,v}$ such that $\norm{h_{in}}_{\mathcal{H}^s_\eps}\leq \delta_s$.
\par Then, $(h_\eps)_{\eps>0}$ exists for all $0<\eps\leq \eps_d$ and converges weakly* in $L^\infty_t(H^s_xL^2_v)$ towards $h$ such that $h \in \mbox{Ker}(L)$, with $\nabla_x\cdot u=0$ and $\rho + \theta =0$.

\bigskip
Furthermore, $\int_0^Thdt$  belongs to $H^s_xL^2_v$ and  it exists $C>0$ such that,
$$\norm{\int_0^{+\infty}hdt - \int_0^{+\infty} h_\eps dt}_{H^s_xL^2_v} \leq C \sqrt{\eps\abs{\mbox{ln}(\eps)}}.$$

\bigskip
One can have a strong convergence in $L^2_{[0,T]}H^s_xL^2_v$ only if $h_{in}$ is in $\emph{\mbox{Ker}}(L)$ with $\nabla_x \cdot u_{in} = 0$ and $\rho_{in} + \theta_{in} =0$ (initial layer conditions).
\par Moreover, in that case we have
$$\norm{h - h_\eps}_{L^2_{[0,+\infty)}H^s_xL^2_v} \leq C \sqrt{\eps\abs{\mbox{ln}(\eps)}},$$
and for all $\delta$ in $[0,1]$, if $h_{in}$ belongs to $H^{s+\delta}_xL^2_v$,
$$\sup\limits_{t\in [0,+\infty)}\norm{h - h_\eps}_{H^s_xL^2_v}(t) \leq C \eps^{\min(\delta,1/2)}.$$
\end{theorem}

This theorem proves the strong convergences for $(\rho_\eps,u_\eps,\theta_\eps)$ towards $(\rho,u,\theta)$ but above all it shows that $(\rho,u,\theta)$ is the solution to the incompressible Navier-Stokes equations together with the Boussinesq equation satisfying the initial conditions:
\begin{itemize}
\item $u(0,x) = Pu_{in}(x)$, where $Pu_{in}(x)$ is the divergence-free part of $u_{in}(x)$,
\item $\rho(0,x) = - \theta(0,x) = \frac{1}{2}(\rho_{in}(x) - \theta_{in}(x)).$
\end{itemize}

\bigskip
Finally, note that in the case of initial data satisfying the initial layer conditions, the strong convergence in time requires a little bit more regularity from the initial data. This fact was already noticed in $\R^d$ (see \cite{BU} Lemma $6.1$) but overcome by considering weighted norms in velocity.

\section{Toolbox: fluid projection and a priori energy estimates} \label{sec:toolbox}
In this section we are going to give some inequalities we are going to use  and to refer to throughout the sequel. First we start with some properties concerning the projection in $L^2_v$ onto $\mbox{Ker}(L)$: $\pi_L$. Then, because we want to estimate all the terms appearing in the $H^s_{x,v}$-norm to estimate the functionals $\mathcal{H}^s_\eps$ and $\mathcal{H}^s_{\eps \bot}$, we will give upper bound on their time derivatives. The proofs are only technical and the interested reader will find them in Appendix $\ref{appendix:toolbox}$.
\par We are assuming there that $L$ is having properties (H1'), (H2') and (H3), that $\Gamma$ satisfies (H4) and (H5) and that $0<\eps\leq 1$.


\subsection{Properties concerning the fluid projection $\pi_L$}
We already know that $L$ is acting on $L^2_v$, with $\mbox{Ker}(L) = \mbox{Span}(\phi_1,\dots,\phi_d)$, with $(\phi_i)_{1\leq i\leq N}$ an orthonormal family, we obtain directly a useful formula for the orthogonal projection on $\mbox{Ker}(L)$ in $L^2_v$, $\pi_L$:

\begin{equation} \label{piL}
\forall h \in L^2_v,\quad \pi_L(h) = \sum_{i=1}^N\left(\int_{\R^d}h\phi_idv\right)\phi_i.
\end{equation}

\bigskip
Plus, (H3) states that $\phi_i = P_i(v)e^{-|v|^2/4}$, where $P_i$ is a polynomial. Therefore, direct computations and Cauchy-Schwarz inequality give that $\pi_L$ is continuous on $H^s_{x,v}$ with

\begin{equation}\label{projineq}
\forall s \in \N, \exists C_{\pi s}>0, \forall h \in H^s_{x,v}, \quad \norm{\pi_L(h)}^2_{H^s_{x,v}}\leq C_{\pi s}\norm{h}^2_{H^s_{x,v}}.
\end{equation}

More precisely one can find that for all $s$ in $\N$

\begin{equation}\label{dvcontrolL}
\forall |j|+|l|=s, \forall h \in H^s_{x,v}, \quad\norm{\partial^j_l\pi_L(h)}^2_{L^2_{x,v}}\leq C_{\pi s}\norm{\partial_l^0\pi_L(h)}^2_{L^2_{x,v}}.
\end{equation}

\bigskip
Finally, building the $\Lambda$-norm one can find that in all the collisional kinetic equations concerned here we have that

\begin{equation}\label{L2L2lambdafluid}
\exists C_\pi >0, \forall h \in L^2_{x,v}, \: \norm{\pi_L(h)}^2_{\Lambda}\leq C_\pi\norm{h}^2_{L^2_{x,v}}.
\end{equation}

\bigskip
Then we can also use the properties of the torus to obtain Poincare type inequalities. This can be very useful thanks to the next proposition, which is proved in Appendix $\ref{appendix:toolbox}$.

\begin{prop} \label{kernel}
Let $a$ and $b$ be in $\R^*$ and consider the operator $G = aL - bv.\nabla_x$ acting on $H^1_{x,v}$.
\par If $L$ satisfies $\emph{(H1)}$ and $\emph{(H3)}$ then $$\emph{\mbox{Ker}}(G) = \emph{\mbox{Ker}}(L).$$
\end{prop}
\begin{remark}
In this proposition, $\mbox{Ker}(G)$ has to be understood as linear combinations with constant coefficients of the functions $\Phi_i$. This subtlety has to be emphasized since in $L^2_{x,v}$, $\mbox{Ker}(L)$ includes all linear combinations of $\Phi_i$ but with coefficients being functions of $x$.
\end{remark}

Therefore, if we define, for $0<\eps\leq 1$:
$$G_\eps = \frac{1}{\eps^2}L - \frac{1}{\eps}v.\nabla_x,$$
then we have a nice desciption of $\pi_{G_\eps}$:
$$\forall h \in L^2_{x,v},\: \pi_{G_\eps}(h) = \sum_{i=1}^N\left(\int_{\T^d}\int_{\R^d}h\phi_i \:dxdv\right)\phi_i.$$
That means that $\pi_{G_\eps}(h)$ is, up to a multiplicative constant, the mean of $\pi_L(h)$ over the torus. We deduce that if $h$ belongs to $\mbox{Ker}(G_\eps)^\bot$, $\pi_L(h)$ has zero mean on the torus and is an operator not depending on the $x$ variable. Thus we can apply Poincar\'e inequality on the torus:

\begin{equation} \label{poincare}
\forall h \in \mbox{Ker}(G_\eps)^\bot, \quad \norm{\pi_L(h)}^2_{L^2_{x,v}} \leq C_p\norm{\nabla_x\pi_L(h)}^2_{L^2_{x,v}} \leq C_p\norm{\nabla_xh}^2_{L^2_{x,v}}.
\end{equation}


\subsection{A priori energy estimates}
Our work in this article is to study the evolution of the norms involved in the definition of the operators $\mathcal{H}^s_\eps$ and $\mathcal{H}^s_{\eps \bot}$ and to combine them to obtain the results stated above. The Appendix $\ref{appendix:toolbox}$ contains the proofs, which are technical computations together with some choices of decomposition, of the following a priori estimates. Note that all the constants $K_1$, $K_{dx}$ and $K_{s-1}$ used in the inequalities below are independent of $\eps$, $\Gamma$ and $g$, and only depend constructively on the constants defined in the hypocoercivity assumptions or in the subsection above. The number $e$ can be any positive real number and will be chosen later.
\par We would like to study both linear and non-linear models but they appeared to be very similar. In order to avoid long and similar inequalities we will write in parenthesis terms we need to add for the full model.

\bigskip
Let $g$ be a function in $H^s_{x,v}$. We now consider a function $h$ in $\mbox{Ker}(G_\eps)^\bot \cap H^s_{x,v}$, for some $s$ in $\N^*$, which is solution of the linear (linearized) Boltzmann equation:
$$\partial_th + \frac{1}{\eps}v.\nabla_xh = \frac{1}{\eps^2}L(h) \:\left(+ \frac{1}{\eps}\Gamma(g,h)\right).$$
We remind the reader that the following notation is used: $h^\bot = h-\pi_L(h)$.

\subsubsection{Time evolutions for quantities in $H^1_{x,v}$}
We write the $L^2_{x,v}$-norm estimate

\begin{equation}\label{h}
\frac{d}{dt}\norm{h}_{L^2_{x,v}}^2  \leq -\frac{\lambda}{\eps^2}\norm{h^\bot}^2_\Lambda \:\left(+\frac{1}{\lambda}\left(\mathcal{G}^0_{x}(g,h)\right)^2\right).
\end{equation}

Then the time evolution of the $x$-derivatives

\begin{equation}\label{dx}
\frac{d}{dt}\norm{\nabla_xh}_{L^2_{x,v}}^2  \leq -\frac{\lambda}{\eps^2}\norm{\nabla_xh^\bot}^2_\Lambda \:\left(+\frac{1}{ \lambda}\left(\mathcal{G}^1_{x}(g,h)\right)^2\right),
\end{equation}

and of the $v$-derivatives

\begin{eqnarray} 
\frac{d}{dt}\norm{\nabla_vh}_{L^2_{x,v}}^2 &\leq&  \frac{K_1}{\eps^2}\norm{h^\bot}^2_{\Lambda} + \frac{K_{dx}}{\eps^2}\norm{\nabla_xh}^2_{L^2_{x,v}} - \frac{\nu_3^\Lambda}{\eps^2}\norm{\nabla_vh}^2_{\Lambda} \label{dv}
\\&&\left(+ \frac{3}{\nu_3^\Lambda}\left(\mathcal{G}^1_{x,v}(g,h)\right)^2\right).\nonumber
\end{eqnarray}

\bigskip
Finally, we will need a control on the scalar product as well, as explained in the strategy subsection $\ref{subsec:strategy}$. Notice that we have some freedom as $e$ can be any positive number.

\begin{eqnarray}
\frac{d}{dt}\langle \nabla_xh,\nabla_vh\rangle _{L^2_{x,v}} &\leq& \frac{C^Le}{\eps^3} \norm{\nabla_xh^\bot}_{\Lambda}^2  -\frac{1}{\eps}\norm{\nabla_xh}^2_{L^2_{x,v}} + \frac{2C^L}{e\eps}\norm{\nabla_vh}_{\Lambda}^2 \nonumber
\\ && \left(+ \frac{e}{C^L\eps}\left(\mathcal{G}^1_{x}(g,h)\right)^2\right). \label{dx,dv}
\end{eqnarray}

\subsubsection{Time evolutions for quantities in $H^s_{x,v}$}
We consider multi-indexes $j$ and $l$ such that $|j|+|l|=s$.
\\As in the previous case, we have a control on the time evolution of the pure $x$-derivatives,

\begin{equation}\label{d0l}
\frac{d}{dt}\norm{\partial^0_lh}_{L^2_{x,v}}^2  \leq -\frac{\lambda}{\eps^2}\norm{\partial^0_lh^\bot}^2_\Lambda \:\left(+\frac{1}{\lambda}\left(\mathcal{G}^s_{x}(g,h)\right)^2\right).
\end{equation}

In the case where $|j|\geq 1$, that is to say when we have at least one derivative in $v$, we obtained the following upper bound

\begin{eqnarray}
\frac{d}{dt}\norm{\partial^j_lh}^2_{L^2_{x,v}} &\leq& -\frac{\nu_5^\Lambda}{\eps^2}\norm{\partial^j_lh}^2_\Lambda + \frac{3(\nu_1^\Lambda)^2 d}{\nu_5^\Lambda(\nu_0^\Lambda)^2}\sum\limits_{i,c_i(j)> 0}\norm{\partial^{j-\delta_i}_{l+\delta_i}h}^2_\Lambda + \frac{K_{s-1}}{\eps^2}\norm{h}^2_{H^{s-1}_{x,v}} \nonumber
\\                                             && \left(+ \frac{3}{\nu_5^\Lambda}\left(\mathcal{G}^s_{x,v}(g,h)\right)^2\right).\label{djl}
\end{eqnarray}

We may find useful to consider the particular case where $|j|=1$,

\begin{eqnarray}
\frac{d}{dt}\norm{\partial^{\delta_i}_{l-\delta_i}h}^2_{L^2_{x,v}} &\leq& -\frac{\nu_5^\Lambda}{\eps^2}\norm{\partial^{\delta_i}_{l-\delta_i}h}^2_\Lambda + \frac{3\nu_1^\Lambda}{\nu_5^\Lambda\nu_0^\Lambda}\norm{\partial^0_lh}^2_{L^2_{x,v}} + \frac{K_{s-1}}{\eps^2}\norm{h}^2_{H^{s-1}_{x,v}} \nonumber
\\                                             && \left(+ \frac{3}{\nu_5^\Lambda}\left(\mathcal{G}^s_{x,v}(g,h)\right)^2\right).\label{ddeltai}
\end{eqnarray}

\bigskip
Finally we will need the time evolution of the following scalar product:

\begin{eqnarray}
\frac{d}{dt}\langle  \partial^{\delta_i}_{l-\delta_i}h, \partial^0_lh\rangle _{L^2_{x,v}} &\leq& \frac{C^Le}{\eps^3} \norm{ \partial^0_lh^\bot}_{\Lambda}^2 -\frac{1}{\eps}\norm{ \partial^0_lh}^2_{L^2_{x,v}} + \frac{2C^L}{e\eps}\norm{ \partial^{\delta_i}_{l-\delta_i}h}_{\Lambda}^2  \nonumber
\\ && \left(+ \frac{e}{ C^L\eps}\left(\mathcal{G}^s_{x}(g,h)\right)^2\right), \label{deltai,d0l}
\end{eqnarray}

where we still have some freedom as $e$ is any positive number.

\bigskip
We just emphasize here that one can see that we were careful about which derivatives are involved in the terms that contain $\Gamma$. This is because our operator $\norm{.}_{\mathcal{H}^s_\eps}$ controls the $H^s_x(L^2_v)$-norm by a mere constant whereas it controls the entire $H^s_{x,v}$-norm by a factor $1/\eps$.

\subsubsection{Time evolutions for orthogonal quantities in $H^s_{x,v}$}
For the theorem $\ref{decaydv}$ we are going to need four others inequalities which are a little bit more intricate as they need to know the shape of $\pi_L$ as described in the subsection above. The proofs are written in Appendix $\ref{appendix:toolbox}$ and we are just looking at the whole equation in the setting $g=h$.

\bigskip
We want the time evolution of the $v$-derivatives of the orthogonal (microscopic) part of $h$, as suggested in \cite{Gu4} this allows us to really take advantage of the structure of the linear operator $L$ on its orthogonal:

\begin{eqnarray}
\frac{d}{dt}\norm{\nabla_vh^\bot}_{L^2_{x,v}}^2 &\leq& \frac{K_1^\bot}{\eps^2}\norm{h^\bot}^2_{\Lambda} + K_{dx}^\bot\norm{\nabla_xh}^2_{L^2_{x,v}} -\frac{\nu_3^\Lambda}{2\eps^2}\norm{\nabla_vh^\bot}^2_{\Lambda} \nonumber
\\ && + \frac{3}{\nu_3^\Lambda}\left(\mathcal{G}^1_{x,v}(h,h)\right)^2. \label{dvbot}
\end{eqnarray}

\bigskip
Then we can have a new bound for the scalar product used before

\begin{eqnarray}
\frac{d}{dt}\langle \nabla_xh,\nabla_vh\rangle _{L^2_{x,v}} &\leq& \frac{K^\bot e}{\eps^3} \norm{\nabla_xh^\bot}_{\Lambda}^2 + \frac{1}{4C_{\pi 1}C_\pi C_pe\eps}\norm{\nabla_vh^\bot}_{\Lambda}^2 \nonumber
\\ &&-\frac{1}{2\eps}\norm{\nabla_xh}^2_{L^2_{x,v}}  + \frac{4C_\pi}{\eps}\left(\mathcal{G}^1_{x,v}(h,h)\right)^2, \label{dx,dvbot}
\end{eqnarray}
where $e$ is any number greater than $1$.
\par As usual, we may need the same kind of bounds in higher degree Sobolev spaces. The reader may notice that the bounds we are about to write are more intricate than the ones in the previous section because they involve more terms with less derivatives. We consider multi-indexes $j$ and $l$ such that $|j|+|l|=s$. This time we really have to divide in two different cases.
\par Firstly when $|j|\geq 2$,

\begin{eqnarray}
\frac{d}{dt}\norm{\partial_l^jh^\bot}^2_{L^2_{x,v}} &\leq& -\frac{\nu_5^\Lambda}{\eps^2}\norm{\partial^j_lh^\bot}^2_\Lambda + \frac{9(\nu_1^\Lambda)^2d}{2(\nu_0^\Lambda)^2\nu_5^\Lambda}\sum\limits_{i,c_i(j)> 0}\norm{\partial^{j-\delta_i}_{l+\delta_i}h^\bot}^2_\Lambda \nonumber 
\\ && + K_{dl}^\bot\sum\limits_{|l'| \leq s-1} \norm{\partial^0_{l'}h}^2_{L^2_{x,v}}  + \frac{K_{s-1}^\bot}{\eps^2} \norm{h^\bot}^2_{H^{s-1}_{x,v}}  + \frac{3}{\nu_5^\Lambda}\left(\mathcal{G}^s_{x,v}(h,h)\right)^2.\label{djlbot}
\end{eqnarray}

\par Then the case when $|j| = 1$

\begin{eqnarray}
\frac{d}{dt}\norm{\partial_{l-\delta_i}^{\delta_i}h^\bot}^2_{L^2_{x,v}} &\leq& - \frac{\nu_5^\Lambda}{\eps^2}\norm{\partial^{\delta_i}_{l-\delta_i}h^\bot}^2_\Lambda+ K_{dl}^\bot\sum\limits_{|l'|=s}\norm{\partial^0_{l'}h}^2_{L^2_{x,v}}+ \frac{K_{s-1}^\bot}{\eps^2}\norm{h^\bot}^2_{H^{s-1}_{x,v}}\nonumber
\\ && + \frac{3}{\nu_5^\Lambda}\left(\mathcal{G}^s_{x,v}(h,h)\right)^2.\label{ddeltaibot}
\end{eqnarray}

\bigskip
Finally we give a new version of the control over the scalar product in higher Sobolev's spaces.

\begin{eqnarray}
\frac{d}{dt}\langle  \partial^{\delta_i}_{l-\delta_i}h, \partial^0_lh\rangle _{L^2_{x,v}} &\leq& \frac{\tilde{K}^\bot}{\eps^3}e \norm{ \partial^0_lh^\bot}_{\Lambda}^2 + \frac{1}{4C_{\pi s}C_\pi d e\eps}\norm{ \partial^{\delta_i}_{l-\delta_i}h^\bot}_{\Lambda}^2 -\frac{1}{2\eps}\norm{ \partial^0_lh}^2_{L^2_{x,v}}  \nonumber
\\&&  + \frac{1}{4d\eps}\sum\limits_{|l'|\leq s-1}\norm{ \partial^0_{l'}h}_{L^2_{x,v}}^2 + \frac{2C_\pi}{\eps}\left(\mathcal{G}^s_{x,v}(h,h)\right)^2,  \label{deltai,d0lbot} 
\end{eqnarray}
for any $e\geq 1$.

\section{Linear case: proof of Theorem $\ref{lin}$}\label{sec:linear}
In this section we are looking at the linear equation
$$\partial_th = G_\eps(h), \:\: \mbox{on}\:\: \T^d\times\R^d. $$

Theorem $\ref{lin}$ will be proved by induction on $s$. We remind here the operator we will work with on $H^s_{x,v}$
\begin{itemize}
\item in the case $s=1$:
$$\norm{h}^2_{\mathcal{H}^1_\eps} = A\norm{h}^2_{L^2_{x,v}} + \alpha\norm{\nabla_xh}^2_{L^2_{x,v}} + b\eps^2\norm{\nabla_vh}^2_{L^2_{x,v}} + a\eps \langle\nabla_xh,\nabla_vh\rangle_{L^2_{x,v}},$$
\item in the case $s>1$:
$$\norm{h}^2_{\mathcal{H}^s_\eps} = \sum\limits_{\overset{|j|+|l|\leq s}{|j|\geq 1}}b_{j,l}^{(s)}\eps^2\norm{\partial^j_lh}^2_{L^2_{x,v}} + \sum\limits_{|l|\leq s}\alpha_l^{(s)}\norm{\partial^0_lh}^2_{L^2_{x,v}} + \sum\limits_{\overset{|l|\leq s}{i,c_i(l)> 0}}a_{i,l}^{(s)}\eps \langle \partial^{\delta_i}_{l-\delta_i}h,\partial^0_lh\rangle _{L^2_{x,v}}.$$
\end{itemize}
The Theorem $\ref{lin}$ only requires us to choose suitable coefficients that gives us the expected inequality and equivalence.

\bigskip
Consider $h_{in}$ in $H^s_{x,v}\cap \mbox{Dom}(G_\eps)$. Let $h$ be a solution of $\partial_th = G_\eps(h)$ on $\T^d\times\R^d$ such that $h(0,\cdot,\cdot) = h_{in}(\cdot,\cdot)$. 
\par Notice that if $h_{in}$ is in $H^s_{x,v}\cap \mbox{Dom}(G_\eps) \cap \mbox{Ker}(G_\eps)$ then we have that the associated solution remains the same in time: $\partial_t h = 0$. Therefore the fluid part of a solution does not evolve in time and so the semigroup is identity on $\mbox{Ker}(G_\eps)$. Besides, we can see directly from the definition and the adjointness property of $L$ that $h\in \mbox{Ker}(G_\eps)^\bot$ for all $t$ if $h_{in}$ belongs in $\mbox{Ker}(G_\eps)^\bot$.
\par Therefore, to prove the theorem it is enough to consider $h_{in}$ in $H^s_{x,v}\cap \mbox{Dom}(G_\eps) \cap \mbox{Ker}(G_\eps)^\bot$.


\subsection{The case $s=1$}

For now on we assume that our operator $L$ satisfies the conditions (H1), (H2) and (H3) and that $0<\eps\leq 1$.
\par If (H3) holds for $L$ then we have that $\eps^{-2}L$ is a non-positive self-adjoint operator on $L^2_{x,v}$. Moreover, $\eps^{-1}v\cdot\nabla_x$ is skew-symmetric on $L^2_{x,v}$. Therefore the $L^2_{x,v}$-norm decreases along the flow and it can be deduced that $G_\eps$ yields a $C_0$-semigroup on $L^2_{x,v}$ for all positive $\eps$ (see \cite{Ka} for general theory and \cite{Uk} for its use in our case).

\bigskip
Using the toolbox, which is possible since $h$ is in $\mbox{Ker}(G_\eps)^\bot$ for all $t$, we just have to consider the linear combination $A\eqref{h} + \alpha\eqref{dx} + b\eps^2\eqref{dv} + a\eps\eqref{dx,dv}$ to obtain

\begin{eqnarray}
\frac{d}{dt}\norm{h}^2_{\mathcal{H}^1_\eps} &\leq& \frac{1}{\eps^2}\left[bK_1 - \lambda A\right] \norm{h^\bot}^2_\Lambda + \frac{1}{\eps^2}\left[C^Lea - \lambda\alpha\right] \norm{\nabla_xh^\bot}^2_\Lambda \nonumber
\\               && + \left[\frac{2C^La}{e} - b\nu_3^\Lambda\right] \norm{\nabla_vh}^2_\Lambda + \left[bK_{dx}-a\right] \norm{\nabla_xh}^2_{L^2_{x,v}}. \label{boundlin}
\end{eqnarray}

Then we make the following choices:
\begin{enumerate}
\item We fix $b$ such that $-\nu_3^\Lambda b < -1$.
\item We fix $A$ big enough such that $\left[bK_1 - \lambda A\right] \leq -1$.
\item We fix $a$ big enough such that $\left[bK_{dx}-a\right] \leq -1$.
\item We fix $e$ big enough such that $\left[\frac{2C^La}{e} - b\nu_3^\Lambda\right]  \leq -1$.
\item We fix $\alpha$ big enough such that $\left[C^Lea - \lambda\alpha\right] \leq -1$ and such that $\left\{\begin{array}{rl} a^2 &\leq \alpha b \\ b &\leq \alpha \end{array}\right.$.
\end{enumerate}
This leads to, because $0 < \eps \leq 1$:
$$\frac{d}{dt}\norm{h}^2_{\mathcal{H}^1_\eps} \leq - \left(\norm{h^\bot}^2_\Lambda + \norm{\nabla_xh^\bot}^2_\Lambda + \norm{\nabla_vh}^2_\Lambda + \norm{\nabla_xh}^2_{L^2_{x,v}}\right).$$

\bigskip
Finally we can apply the Poincar\'e inequality $\eqref{poincare}$ together with the equivalence of the $L^2_{x,v}$-norm and the $\Lambda$-norm on the fluid part $\pi_L$, equation $\eqref{L2L2lambdafluid}$, to get
$$\exists C, C' > 0,\quad \left\{\begin{array}{ll}\norm{h}^2_\Lambda &\leq C\left(\norm{h^\bot}^2_\Lambda + \frac{1}{2}\norm{\nabla_xh}^2_{L^2_{x,v}}\right),
                                   \\\norm{\nabla_xh}^2_\Lambda &\leq C'\left(\norm{\nabla_xh^\bot}^2_\Lambda + \frac{1}{2}\norm{\nabla_xh}^2_{L^2_{x,v}}\right).\end{array}\right.$$

\bigskip
Therefore we proved the following result:
$$\exists K > 0, \forall \: 0< \eps\leq1 \:,\quad \frac{d}{dt}\norm{h}^2_{\mathcal{H}^1_\eps}\leq -C^{(1)}_G\left(\norm{h}^2_\Lambda + \norm{\nabla_{x,v}h}^2_\Lambda\right).$$
With these constants, $\norm{.}_{\mathcal{H}^1_\eps}$ is equivalent to $$\left(\norm{h}^2_{L^2_{x,v}}+\norm{\nabla_xh}^2_{L^2_{x,v}}+\eps^2\norm{\nabla_vh}^2_{L^2_{x,v}}\right)^{1/2}$$ since $a^2\leq \alpha b$ and $b\leq \alpha$ and hence:

$$A\norm{h}^2_{L^2_{x,v}} +\frac{b}{2}\left(\norm{\nabla_xh}^2_{L^2_{x,v}}+\eps^2\norm{\nabla_vh}^2_{L^2_{x,v}}\right) \leq \norm{h}^2_{\mathcal{H}^1_\eps} $$
and 
$$\norm{h}^2_{\mathcal{H}^1_\eps}\leq A\norm{h}^2_{L^2_{x,v}} + \frac{3\alpha}{2}\left(\norm{\nabla_xh}^2_{L^2_{x,v}}+\eps^2\norm{\nabla_vh}^2_{L^2_{x,v}}\right).$$
The results above gives us the expected theorem for $s=1$.


\subsection{The induction in higher order Sobolev spaces}
Then we assume that the theorem is true up to the integer $s-1$, $s > 1$. Then we suppose that $L$ satisfies (H1'), (H2') and (H3) and we consider $\eps$ in $(0,1]$.
\par Let $h_{in}$ be in $H^s_{x,v}\cap \mbox{Dom}(G_\eps) \cap \mbox{Ker}(G_\eps)^\bot$ and $h$ be the solution of $\partial_th=G_\eps(h)$ such that $h(0,\cdot,\cdot) = h_{in}(\cdot,\cdot)$.
\par As before, $h$ belongs to $\mbox{Ker}(G_\eps)^\bot$ for all $t$ and thus we can use the results given by the toolbox.

\bigskip
Thanks to the proof in the case $s=1$ we know that we are able to handle the case where there is only a difference of one derivative between the number of derivatives in $x$ and in $v$. Therefore, instead of working with the entire norm of $H^s_{x,v}$, we will look at an equivalent of the Sobolev semi-norm. We define:
\begin{eqnarray*}
F_s(t) &=& \sum\limits_{\overset{|j|+|l|=s}{|j|\geq 2}} \eps^2B\norm{\partial^j_lh}^2_{L^2_{x,v}} + B'\sum\limits_{\overset{|l|=s}{i,c_i(l)> 0}} Q_{l,i}(t),
\\Q_{l,i}(t) &=& \alpha\norm{\partial^0_lh}^2_{L^2_{x,v}} + b\eps^2\norm{\partial^{\delta_i}_{l-\delta_i}h}^2_{L^2_{x,v}} + a\eps\langle \partial^{\delta_i}_{l-\delta_i}h,\partial^0_lh\rangle _{L^2_{x,v}},
\end{eqnarray*}
where the constants, strictly positive, will be chosen later.
\par Like in the section above, we shall study the time evolution of every term involved in $F_s$ in order to bound above $dF_s/dt (t)$ with negative coefficients.

\subsubsection{The time evolution of $Q_{l,i}$}
We will first study the time evolution of $Q_{l,i}$ for given $|j|+|l| = s$. The toolbox already gave us all the bounds we need and we just have to gather them in the following way: $\alpha\eqref{d0l} + b\eps^2\eqref{ddeltai} + a\eps\eqref{deltai,d0l}$. This leads to, because $0<\eps\leq 1$,

\begin{eqnarray*}
\frac{d}{dt}Q_{l,i}(t) &\leq& \frac{1}{\eps^2}\left[C^Lea - \lambda\alpha\right] \norm{\partial^0_lh^\bot}^2_\Lambda + \left[\frac{2C^La}{e} -\nu_5^\Lambda b\right] \norm{\partial^{\delta_i}_{l-\delta_i}h}^2_\Lambda
\\ &&  + \left[\frac{3\nu_1^\Lambda}{\nu_5^\Lambda\nu_0^\Lambda} b-a\right] \norm{\partial^0_lh}^2_{L^2_{x,v}} +K_{s-1}b\norm{h}_{H^{s-1}_{x,v}}.
\end{eqnarray*}

One can notice that, except for the last term, we have exactly the same kind of bound as in $\eqref{boundlin}$, in the proof of the case $s=1$.  Therefore we can choose $\alpha$, $b$, $a$, $e$, independently of $\eps$ such that it exists $K_Q > 0$ and $C_{s-1} > 0$ such that for all $0 < \eps\leq 1$:
\begin{itemize}
\item $Q_{l,i}(t) \sim \norm{\partial^0_lh}^2_{L^2_{x,v}} + \eps^2\norm{\partial^{\delta_i}_{l-{\delta_i}}h}^2_{L^2_{x,v}},$
\\
\item $\frac{d}{dt}Q_{l,i}(t) \leq -K_Q\left(\norm{\partial^0_lh}^2_\Lambda + \norm{\partial^{\delta_i}_{l-{\delta_i}}h}^2_\Lambda\right) + C_{s-1}\norm{h}_{H^{s-1}_{x,v}},$
\end{itemize}
where we used $\eqref{L2L2lambdafluid}$ (equivalence of norms $L^2_{x,v}$ and $\Lambda$ on the fluid part) to get
$$\norm{\partial^0_lh}^2_\Lambda \leq C'\left(\norm{\partial^0_lh^\bot}^2_\Lambda + \norm{\partial^0_lh}^2_{L^2_{x,v}}\right).$$

\subsubsection{The time evolution of $F_s$ and conclusion}
The last result about $Q_{l,i}$ gives us that
$$F_s(t) \sim \sum\limits_{|l|= s}\norm{\partial_l^0h}^2_{L^2_{x,v}}+\eps^2\sum\limits_{\overset{|l|+|j|= s}{|j|\geq 1}}\norm{\partial_l^jh}^2_{L^2_{x,v}}.$$
To study the time evolution of $F_s$ we just need to combine the evolution of $Q_{l,i}$ and the one of $\norm{\partial^j_lh}^2_{L^2_{x,v}}$ which is given in the toolbox by $\eqref{djl}$.

\begin{eqnarray}
\frac{d}{dt}F_s(t) &\leq& \sum\limits_{\overset{|j|+|l|=s}{|j|\geq 2}} -\nu_5^\Lambda B\norm{\partial^j_lh}^2_\Lambda  + \sum\limits_{\overset{|j|+|l|=s}{|j|\geq 2}}\frac{3(\nu_1^\Lambda)^2 d}{\nu_5^\Lambda (\nu_0^\Lambda)^2}B\eps^2\sum\limits_{i,c_i(j)> 0}\norm{\partial^{j-\delta_i}_{l+\delta_i}h}^2_\Lambda \nonumber
\\                 && -K_QB' \sum\limits_{\overset{|l|=s}{i,c_i(l)> 0}}\left(\norm{\partial^0_lh}^2_\Lambda + \norm{\partial^{\delta_i}_{l-{\delta_i}}h}^2_\Lambda\right) \label{Fklin}
\\                && + \left[\sum\limits_{\overset{|j|+|l|=s}{|j|\geq 2}}K_{s-1}B+ \sum\limits_{\overset{|l|=s}{i,c_i(l)> 0}}B'C_{s-1}\right]\norm{h}^2_{H^{s-1}_{x,v}}. \nonumber
\end{eqnarray}

\bigskip
Then we choose the following coefficients $B = 2/\nu_5^\Lambda$ and we can rearrange the sums to obtain
\begin{eqnarray*}
\frac{d}{dt}F_s(t) &\leq& \sum\limits_{\overset{|j|+|l|=s}{|j|\geq 2}} \left(\frac{6d(\nu_1^\Lambda)^2}{(\nu_5^\Lambda\nu_0^\Lambda)^2}\eps^2 - 2\right)\norm{\partial^j_lh}^2_\Lambda + \sum\limits_{\overset{|j|+|l|=s}{|j|=1}}\left(\frac{6d(\nu_1^\Lambda)^2}{(\nu_5^\Lambda\nu_0^\Lambda)^2}\eps^2 - K_QB'\right) \norm{\partial^j_lh}^2_\Lambda
\\                  && + \sum\limits_{\overset{|j|+|l|=s}{|j|=0}}(-K_QB') \norm{\partial^j_lh}^2_\Lambda + C_+^{(s-1)}(B') \norm{h}_{H^{s-1}_{x,v}}.
\end{eqnarray*}

\bigskip
Therefore we can choose the remaining coefficients:
\begin{enumerate}
\item $\eps_d = \min\left\{1,\sqrt{\frac{(\nu_5^\Lambda\nu_0^\Lambda)^2}{6d(\nu_1^\Lambda)^2}} \right\}$,
\item we fix $B'$ big enough such that $K_QB'\geq 1$ and $\left(\frac{6d(\nu_1^\Lambda)^2}{(\nu_5^\Lambda\nu_0^\Lambda)^2}\eps_d^2 - K_QB'\right)\leq -1$.
\end{enumerate}

\bigskip
Everything is now fixed in $C_+^{(s-1)}(B')$ and therefore it is just a constant $C_+^{(s-1)}$ that does not depend on $\eps$. Therefore we then have the final result.
$$\forall \: 0 < \eps\leq\eps_d \:,\: \frac{d}{dt}F_s(t) \leq C_+^{(s-1)} \norm{h}_{H^{s-1}_{x,v}}^2 - \left(\sum\limits_{|j|+|l|=s} \norm{\partial^j_lh}^2_\Lambda\right).$$
Then, we know that $\norm{.}_\Lambda$ controls the $L^2$-norm. And therefore: 
$$\forall \: 0 < \eps\leq\eps_d \:,\: \frac{d}{dt}F_s(t) \leq C_+^{(s)} \left(\sum\limits_{|j|+|l|\leq s-1} \norm{\partial^j_lh}^2_\Lambda\right) - \left(\sum\limits_{|j|+|l|=s} \norm{\partial^j_lh}^2_\Lambda\right).$$
This inequality is true for all $s$ and therefore we can take a linear combination of the $F_s$ to obtain the following, where $C_s$ is a constant that does not depend on $\eps$ since $C_+^{(s)}$ does not depend on it.
$$\forall \: 0 < \eps\leq\eps_d \:,\: \frac{d}{dt}\left(\sum\limits_{p = 1}^n C_pF_p(t)\right) \leq - C^{(s)}_G \left(\sum\limits_{|j|+|l||\leq s} \norm{\partial^j_lh}^2_\Lambda\right).$$

\bigskip
We can use the induction assumption from rank $1$ up to rank $s-1$ to find that this linear combination is equivalent to $$\norm{.}^2_{L^2_{x,v}}+\sum\limits_{|l|\leq s}\norm{\partial_l^0.}^2_{L^2_{x,v}}+\eps^2\sum\limits_{\overset{|l|+|j|\leq s}{|j|\geq 1}}\norm{\partial_l^j.}^2_{L^2_{x,v}}$$
 and so fits the expected requirements.

\section{Estimate for the full equation: proof of Proposition $\ref{apriori}$} \label{sec:apriori}
We will prove that proposition by induction on $s$. For now on we assume that $L$ satisfies hypothesis (H1'), (H2') and (H3), that $\Gamma$ satisfies properties (H4) and (H5) and we take $g$ in $H^s_{x,v}$.

\bigskip
So we take $h_{in}$ in $H^s_{x,v}\cap \mbox{Ker}(G_\eps)^\bot$ and we consider the associated solution, denoted by $h$, of
$$\partial_t h + \frac{1}{\eps}v\cdot\nabla_x h = \frac{1}{\eps^2}L(h) + \frac{1}{\eps}\Gamma(g,h).$$

\bigskip
One can notice that thanks to (H5) and the self-adjointness of $L$, $h$ remains in $\mbox{Ker}(G_\eps)^\bot$ for all times.
\par Besides, while considering the time evolution we find a term due to $G_\eps$ and another due to $\Gamma$. Therefore, we will use the results found in the toobox but including the terms in parenthesis.


\subsection{The case $s=1$}

We want to study the following operator on $H^s_{x,v}$
$$\norm{h}^2_{\mathcal{H}^1_\eps} = A\norm{h}^2_{L^2_{x,v}} + \alpha\norm{\nabla_xh}^2_{L^2_{x,v}} + b\eps^2\norm{\nabla_vh}^2_{L^2_{x,v}} + a\eps \langle\nabla_xh,\nabla_vh\rangle_{L^2_{x,v}}.$$

\bigskip
Therefore, using the toolbox we just have to consider the linear combination $A\eqref{h} + \alpha\eqref{dx} + b\eps^2\eqref{dv} + a\eps\eqref{dx,dv}$ to yield

\begin{eqnarray}
\frac{d}{dt}\norm{h}^2_{\mathcal{H}^1_\eps} &\leq& \frac{1}{\eps^2}\left[bK_1 - \lambda A\right] \norm{h^\bot}^2_\Lambda + \frac{1}{\eps^2}\left[C^Lea - \lambda\alpha\right] \norm{\nabla_xh^\bot}^2_\Lambda \nonumber
\\               && + \left[\frac{2C^La}{e} - b\nu_3^\Lambda\right] \norm{\nabla_vh}^2_\Lambda + \left[bK_{dx}-a\right] \norm{\nabla_xh}^2_{L^2_{x,v}}  \label{boundpriori}
\\               && + \frac{A\nu_1^\Lambda}{\nu_0^\Lambda \lambda}\left(\mathcal{G}^0_{x}(g,h)\right)^2 + \left[\frac{\alpha\nu_1^\Lambda}{\nu_0^\Lambda\lambda} + \frac{\nu_1^\Lambda ea}{C^L\nu_0^\Lambda}\right]\left(\mathcal{G}^1_{x}(g,h)\right)^2 \nonumber
\\               && + \frac{3\nu_1^\Lambda b}{\nu_0^\Lambda\nu_3^\Lambda}\eps^2\left(\mathcal{G}^1_{x,v}(g,h)\right)^2. \nonumber
\end{eqnarray}

\bigskip
One can see that we obtained exactly the same upper bound as in the proof of the previous theorem, equation $\eqref{boundlin}$, adding the terms involving $\Gamma$ (remember that $\mathcal{G}^s_{x}$ is increasing in $s$). Therefore we can make the same choices for $A$, $\alpha$, $b$, $a$ and $e$, independently of $\Gamma$ and $g$, to get that
$$\norm{h}^2_{\mathcal{H}^1_\eps} \sim \norm{h}^2_{L^2_{x,v}} + \norm{\nabla_xh}^2_{L^2_{x,v}}+\eps^2\norm{\nabla_vh}^2_{L^2_{x,v}},$$
and that, once those parameters are fixed, there exist $ K^{(1)}_0, \: K^{(1)}_1,\: K^{(1)}_2 > 0$ such that for all $ 0< \eps\leq1$,

$$ \frac{d}{dt}\norm{h}^2_{\mathcal{H}^1_\eps} \leq -K^{(1)}_0\left(\norm{h}^2_\Lambda + \norm{\nabla_{x,v}h}^2_\Lambda\right)+K^{(1)}_1 \left(\mathcal{G}^1_{x}(g,h)\right)^2 +\eps^2K^{(1)}_2\left(\mathcal{G}^1_{x,v}(g,h)\right)^2,
$$

which is the expected result in the case $s=1$.


\subsection{The induction in higher order Sobolev spaces}
Then we assume that the theorem is true up to the integer $s-1$, $s > 1$. Then we suppose that $L$ satisfies (H1'), (H2') and (H3) and we consider $\eps$ in $(0,1]$.
\par Since $h_{in}$ is in $\mbox{Ker}(G_\eps)^\bot$, $h$ belongs to $\mbox{Ker}(G_\eps)^\bot$ for all $t$ and so we can use the results given in the toolbox.

\bigskip
As in the proof in the linear case we define:
\begin{eqnarray*}
F_s(t) &=& \sum\limits_{\overset{|j|+|l|=s}{|j|\geq 2}} \eps^2B\norm{\partial^j_lh}^2_{L^2_{x,v}} + B'\sum\limits_{\overset{|l|=s}{i,c_i(l)> 0}} Q_{l,i}(t),
\\Q_{l,i}(t) &=& \alpha\norm{\partial^0_lh}^2_{L^2_{x,v}} + b\eps^2\norm{\partial^{\delta_i}_{l-\delta_i}h}^2_{L^2_{x,v}} + a\eps\langle \partial^{\delta_i}_{l-\delta_i}h,\partial^0_lh\rangle _{L^2_{x,v}},
\end{eqnarray*}
where the constants, strictly positive, will be chosen later.
\par Like in the section above, we shall study the time evolution of every term involved in $F_s$ in order to bound above $dF_s/dt(t)$ with expected coefficients.

\subsubsection{The time evolution of $Q_{l,i}$}
We will first study the time evolution of $Q_{l,i}$ for given $|j|+|l| = s$. The toolbox already gave us all the bounds we need and we just have to gather them in the following way: $\alpha\eqref{d0l} + b\eps^2\eqref{ddeltai} + a\eps\eqref{deltai,d0l}$. This leads to, because $0<\eps\leq 1$,

\begin{eqnarray*}
\frac{d}{dt}Q_{l,i}(t) &\leq& \frac{1}{\eps^2}\left[C^Lea - \lambda\alpha\right] \norm{\partial^0_lh^\bot}^2_\Lambda + \left[\frac{2C^La}{e} -\nu_5^\Lambda b\right] \norm{\partial^{\delta_i}_{l-\delta_i}h}^2_\Lambda
\\ &&  + \left[\frac{3\nu_1^\Lambda}{\nu_5^\Lambda\nu_0^\Lambda} b-a\right] \norm{\partial^0_lh}^2_{L^2_{x,v}} +K_{s-1}b\norm{h}_{H^{s-1}_{x,v}}
\\ &&  +  \left[\frac{\alpha\nu_1^\Lambda}{\nu_0^\Lambda\lambda} + \frac{\nu_1^\Lambda ea}{C^L\nu_0^\Lambda}\right]\left(\mathcal{G}^s_{x}(g,h)\right)^2 + \frac{3\nu_1^\Lambda b}{\nu_0^\Lambda\nu_5^\Lambda}\eps^2\left(\mathcal{G}^s_{x,v}(g,h)\right)^2.
\end{eqnarray*}

\bigskip
One can notice that, except for the term in $\norm{h}_{H^{s-1}_{x,v}}$, we have exactly the same kind of bound as in the case $s=1$, given by $\eqref{boundpriori}$.  Therefore we can choose $\alpha$, $b$, $a$, $e$, independently of $\eps$, $\Gamma$ and $g$ such that it exists $K_Q,\: K_{\Gamma 1}, \: K_{\Gamma 2} > 0$ and $C_{s-1} > 0$ such that for all $0 < \eps\leq 1$:
\begin{itemize}
\item $Q_{l,i}(t) \sim \norm{\partial^0_lh}^2_{L^2_{x,v}} + \eps^2\norm{\partial^{\delta_i}_{l-{\delta_i}}h}^2_{L^2_{x,v}},$
\\
\item \begin{eqnarray*}\frac{d}{dt}Q_{l,i}(t) &\leq& -K_Q\left(\norm{\partial^0_lh}^2_\Lambda + \norm{\partial^{\delta_i}_{l-{\delta_i}}h}^2_\Lambda\right) + K_{\Gamma 1}\left(\mathcal{G}^s_{x}(g,h)\right)^2
\\ && + \eps^2 K_{\Gamma 2}\left(\mathcal{G}^s_{x,v}(g,h)\right)^2 +  C_{s-1}\norm{h}_{H^{s-1}_{x,v}},
\end{eqnarray*}
\end{itemize}
where we used $\eqref{L2L2lambdafluid}$ (equivalence of norms $L^2_{x,v}$ and $\Lambda$ on the fluid part) to get
$$\norm{\partial^0_lh}^2_\Lambda \leq C'\left(\norm{\partial^0_lh^\bot}^2_\Lambda + \norm{\partial^0_lh}^2_{L^2_{x,v}}\right).$$

\subsubsection{The time evolution of $F_s$ and conclusion}
The last result about $Q_{l,i}$ gives us that
$$F_s(t) \sim \sum\limits_{|l|= s}\norm{\partial_l^0h}^2_{L^2_{x,v}}+\eps^2\sum\limits_{\overset{|l|+|j|= s}{|j|\geq 1}}\norm{\partial_l^jh}^2_{L^2_{x,v}},$$
so it remains to show that $F_s$ satisfies the property describe by the theorem for some $B$ and $B'$.

\bigskip
To study the time evolution of $F_s$ we just need to combine the evolution of $Q_{l,i}$ and the one of $\norm{\partial^j_lh}^2_{L^2_{x,v}}$ which is given in the toolbox by $\eqref{djl}$.

\begin{eqnarray}
\frac{d}{dt}F_s(t) &\leq& \sum\limits_{\overset{|j|+|l|=s}{|j|\geq 2}} -\nu_5^\Lambda B\norm{\partial^j_lh}^2_\Lambda  + \sum\limits_{\overset{|j|+|l|=s}{|j|\geq 2}}\frac{3(\nu_1^\Lambda)^2 d}{\nu_5^\Lambda (\nu_0^\Lambda)^2}B\eps^2\sum\limits_{i,c_i(j)> 0}\norm{\partial^{j-\delta_i}_{l+\delta_i}h}^2_\Lambda \nonumber
\\                 && -K_QB' \sum\limits_{\overset{|l|=s}{i,c_i(l)> 0}}\left(\norm{\partial^0_lh}^2_\Lambda + \norm{\partial^{\delta_i}_{l-{\delta_i}}h}^2_\Lambda\right) \nonumber
\\                 && + \left[\sum\limits_{\overset{|j|+|l|=s}{|j|\geq 2}}K_{s-1}B+ \sum\limits_{\overset{|l|=s}{i,c_i(l)> 0}}B'C_{s-1}\right]\norm{h}^2_{H^{s-1}_{x,v}} \label{Fkpriori}
\\                 && + \sum\limits_{\overset{|l|=s}{i,c_i(l)> 0}}B'K_{\Gamma 1}\left(\mathcal{G}^s_{x}(g,h)\right)^2 \nonumber
\\ && + \eps^2 \left[\sum\limits_{\overset{|l|=s}{i,c_i(l)> 0}}B'K_{\Gamma 2} + \sum\limits_{\overset{|j|+|l|=s}{|j|\geq 2}}\frac{3\nu_1^\Lambda}{\nu_0^\Lambda\nu_5^\Lambda}B\right]\left(\mathcal{G}^s_{x,v}(g,h)\right)^2. \nonumber
\end{eqnarray}

\bigskip
One can easily see that, apart from the terms including $\Gamma$, we have exactly the same bound as in the proof in the linear case, equation $\eqref{Fklin}$. Therefore we can choose $B$, $B'$ and $\eps_d$ like we did, thus independent of $\Gamma$ and $g$, to have for all $0<\eps\leq \eps_d$
\begin{eqnarray*}
\frac{d}{dt}F_s(t) \leq && C_+^{(s-1)} \norm{h}_{H^{s-1}_{x,v}}^2 - \left(\sum\limits_{|j|+|l|=s} \norm{\partial^j_lh}^2_\Lambda\right)
\\ &&+ \tilde{K}_{\Gamma 1}\left(\mathcal{G}^s_{x}(g,h)\right)^2+ \eps^2\tilde{K}_{\Gamma 2}\left(\mathcal{G}^s_{x,v}(g,h)\right)^2,
\end{eqnarray*}
with $C_+^{(s-1)}$, $\tilde{K}_{\Gamma 1}$ and $\tilde{K}_{\Gamma 2}$ positive constants independent of $\eps$, $\Gamma$ and $g$.

\bigskip
To conclude we just have to, as in the linear case, take a linear combination of the $\left(F_p\right)_{p\leq s}$ and use the induction hypothesis (remember that both $\mathcal{G}^p_{x,v}$ and $\mathcal{G}^p_{x}$ are increasing functions of $p$) to obtain the expected result: $\forall \: 0 < \eps\leq\eps_d \:,$
\begin{eqnarray*}
\frac{d}{dt}\left(\sum\limits_{p = 1}^n C_pF_p(t)\right) \leq &-& K^{(s)}_0 \left(\sum\limits_{|j|+|l||\leq s} \norm{\partial^j_lh}^2_\Lambda\right) +K^{(s)}_1\left(\mathcal{G}^s_{x}(g,h)\right)^2
\\ &+& \eps^2K^{(s)}_1\left(\mathcal{G}^s_{x,v}(g,h)\right)^2,
\end{eqnarray*}
with this linear combination being equivalent to $$\norm{\cdot}^2_{L^2_{x,v}}+\sum\limits_{|l|\leq s}\norm{\partial_l^0\cdot}^2_{L^2_{x,v}}+\eps^2\sum\limits_{\overset{|l|+|j|\leq s}{|j|\geq 1}}\norm{\partial_l^j\cdot}^2_{L^2_{x,v}}$$
 and so fits the expected requirements.

\section{Existence and exponential decay: proof of Theorem $\ref{perturb}$} \label{sec:perturbresult}
One can clearly see that solving the kinetic equation $\eqref{Boltz}$ in the setting $f = \mu + \eps \mu^{1/2}h$ is equivalent to solving the linearized kinetic equation $\eqref{LinBoltz}$ directly. Therefore we are going to focus only on this linearized equation.

\bigskip
The proof relies on the \textit{a priori} estimate derived in the previous section. We shall use this inequality as a bootstrap to obtain first the existence of solutions thanks to an iteration scheme and then the exponential decay of those solutions, as long as the initial data is small enough.


\subsection{Proof of the existence of global solutions} \label{subsec:cauchy}

\subsubsection{Construction of solutions to a linearized problem}
Here we will follow the classical method that is approximating our solution by a sequence of solutions of a linearization of our initial problem. Then we have to construct a functional on Sobolev spaces for which this sequence can be uniformly bounded in order to be able to extract a convergent subsequence.
\par Starting from $h_0$ in $H^s_{x,v}\cap \mbox{Ker}(G_\eps)^\bot$, to be define later,  we define the function $h_{n+1}$ in $H^s_{x,v}$ by induction on $n \geq 0$ :

\begin{equation}\label{LBEn}
\left\{\begin{array}{rl} &\displaystyle{\partial_t h_{n+1} + \frac{1}{\eps}v.\nabla_x h_{n+1} = \frac{1}{\eps^2}L(h_{n+1}) + \frac{1}{\eps}\Gamma(h_{n},h_{n+1})}\vspace{2mm} 
                         \\ \vspace{2mm}& \displaystyle{h_{n+1}(0,x,v) = h_{in}(x,v)}, \end{array}\right.
\end{equation}

\bigskip
First we need to check that our sequence is well-defined.

\begin{lemma}\label{lem1}
Let $L$ be satisfying assumptions $\emph{(H1')}$, $\emph{(H2')}$ and $\emph{(H3)}$, and let $\Gamma$ be satisfying assumptions $\emph{(H4)}$ and $\emph{(H5)}$.
\par Then, it exists $0<\eps_d\leq 1$ such that for all $s \geq s_0$ (defined in (H4)), it exists $\delta_s >0$ such that for all $0< \eps \leq \eps_d$, if $\norm{h_{in}}_{\mathcal{H}^s_\eps}\leq \delta_s$ then the sequence $(h_n)_{n \in \N}$ is well-defined, continuous in time, in $H^s_{x,v}$ and belongs to $\emph{\mbox{Ker}}(G_\eps)^\bot$.
\end{lemma}

\begin{proof}[Proof of Lemma $\ref{lem1}$]
By induction, let us suppose that for a fixed $n \geq 0$ we have constructed $h_n$ in $H^s_{x,v}$, which is true for $h_{in}$.
\\\par Using the previous notation one can see that we are in fact trying to solve the linear equation on the torus:
$$\partial_t h_{n+1} = G_\eps(h_{n+1}) + \frac{1}{\eps}\Gamma(h_n,h_{n+1})$$
with $h_{in}$ as an initial data.
\par The existence of a solution $h_{n+1}$ has already been shown for each equation covered by the hypocoercivity theory in the case $\eps =1$ (see papers described in the introduction). It was proved by fixed point arguments applied to the Duhamel's formula. In order not to write several times the same estimates one may use our next lemma $\ref{lem3}$ together with the Duhamel's formula (instead of considering directly the time derivative of $h_{n+1}$) to get a fixed point argument as long as $h_{in}$ is small enough, the smallness not depending on $\eps$.

\bigskip
As shown in the study of the linear part of the linearized model, under assumptions (H1'), (H2') and (H3) $G_\eps$ generates a $C^0$-semigroup on $H^s_{x,v}$, for all $0<\eps\leq \eps_d$. Moreover, hypothesis (H4) shows us that $\Gamma(h_n,\cdot)$ is a bounded linear operator from $(H^s_{x,v},E(\cdot))$ to $(H^s_{x,v},\norm{\cdot}_{H^s_{x,v}}$). Thus $h_{n+1}$ is in $H^s_{x,v}$.

\bigskip
The belonging to $\mbox{Ker}(G_\eps)^\bot$  is direct since $\Gamma(h_n,\cdot)$ is in $\mbox{Ker}(G_\eps)^\bot$ (hypothesis (H5)).
\end{proof}

\bigskip
Then we have to strongly bound the sequence, at least in short time, to have a chance to obtain a convergent subsequence, up to an extraction.


\subsubsection{Boundedness of the sequence}

We are about to prove the global existence in time of solutions in $C(\R^+,\norm{.}_{\mathcal{H}^s_\eps})$.That will give us existence of solutions in standard Sobolev's spaces as long as the initial data is small enough in the sense of the $\mathcal{H}^s_\eps$-norm,which is smaller than the standard $H^s_{x,v}$-norm. To achieve that we define a new functional on $H^s_{x,v}$

\begin{equation}\label{defE}
E(h) = \sup\limits_{t \in \R^+} \left(\norm{h(t)}_{\mathcal{H}^s_\eps}^2 + \int_0^t\norm{h(s)}_{H^s_\Lambda}^2ds\right).
\end{equation}

\begin{lemma}\label{lem3}
Let $L$ be satisfying assumptions $\emph{(H1')}$, $\emph{(H2')}$ and $\emph{(H3)}$, and let $\Gamma$ be satisfying assumptions $\emph{(H4)}$ and $\emph{(H5)}$.
\par Then it exists $0<\eps_d \leq 1$ such that for all $s\geq s_0$ (defined in (H4)) it exists $\delta_s > 0$  independent of $\eps$, such that for all $0<\eps \leq \eps_d$, if $\norm{h_{in}}_{\mathcal{H}^s_\eps}\leq \delta_s$ then 
$$\left(E(h_n)\leq \delta_s \right) \Rightarrow \left( E(h_{n+1}) \leq \delta_s \right).$$
\end{lemma}

\begin{proof}[Proof of Lemma $\ref{lem3}$]
We let $t>0$.
\\We know that $h_{in}$ belongs to $H^s_{x,v} \cap \mbox{Ker}(G_\eps)^\bot$. Moreover we have, thanks to Lemma $\ref{lem1}$, that $(h_n)$ is well-defined, in $\mbox{Ker}(G_\eps)^\bot$ and in $H^s_{x,v}$, since $s\geq s_0$. Moreover, $\Gamma$ satisfies (H5). Therefore we can use the Proposition $\ref{apriori}$ to write, for $\eps \leq \eps_d$ ($\eps_d$ being the minimum between the one in Lemma $\ref{lem1}$ and the one in Proposition $\ref{apriori}$),

\begin{eqnarray*}
\frac{d}{dt}\norm{h_{n+1}}^2_{\mathcal{H}^s_\eps} &\leq& - K^{(s)}_0 \norm{h_{n+1}}^2_{H^s_\Lambda} + K^{(s)}_1 \left(\mathcal{G}^s_x(h_n,h_{n+1})\right)^2+ \eps^2K^{(s)}_2\left(\mathcal{G}^s_{x,v}(h_n,h_{n+1})\right)^2 .
\end{eqnarray*}

\bigskip
We can use the hypothesis (H4) and the fact that
\begin{equation} \label{eq:normeqv}
C_m\left(\norm{.}^2_{L^2_{x,v}}+\sum\limits_{|l|\leq s}\norm{\partial_l^0.}^2_{L^2_{x,v}}+\eps^2\sum\limits_{\overset{|l|+|j|\leq s}{|j|\geq 1}}\norm{\partial_l^j.}^2_{L^2_{x,v}}\right) \leq \norm{.}^2_{\mathcal{H}^s_{\eps}} \leq C_M\norm{.}_{H^s_{x,v}},
\end{equation}
to get the following upper bounds:

\begin{eqnarray*}
 \left(\mathcal{G}^s_x(h_n,h_{n+1})\right)^2 &\leq& \frac{C_\Gamma^2}{C_m}\left(\norm{h_n}^2_{\mathcal{H}^s_\eps}\norm{h_{n+1}}^2_{H^s_\Lambda}+\norm{h_{n+1}}^2_{\mathcal{H}^s_\eps}\norm{h_n}^2_{H^s_\Lambda}\right)
\\ \left(\mathcal{G}^s_{x,v}(h_n,h_{n+1})\right)^2 &\leq& \frac{C_\Gamma^2}{C_m\eps^2}\left(\norm{h_n}^2_{\mathcal{H}^s_\eps}\norm{h_{n+1}}^2_{H^s_\Lambda}+\norm{h_{n+1}}^2_{\mathcal{H}^s_\eps}\norm{h_n}^2_{H^s_\Lambda}\right).
\end{eqnarray*}

\bigskip
Therefore we have the following upper bound, where $K_1$ and $K_2$ are constants independent of $\eps$:

\begin{eqnarray*}
\frac{d}{dt}\norm{h_{n+1}}^2_{\mathcal{H}^s_\eps} &\leq& - K^{(s)}_0 \norm{h_{n+1}}^2_{H^s_\Lambda} + K_1\norm{h_n}^2_{\mathcal{H}^s_\eps}\norm{h_{n+1}}^2_{H^s_\Lambda} + K_2\norm{h_{n+1}}^2_{\mathcal{H}^s_\eps}\norm{h_{n}}^2_{H^s_\Lambda}
\\ &\leq& \left[K_1E(h_n)- K^{(s)}_0\right] \norm{h_{n+1}}^2_{H^s_\Lambda} + K_2E(h_{n+1})\norm{h_{n}}^2_{H^s_\Lambda}.
\end{eqnarray*}

\bigskip
We consider now that $E(h_n) \leq K^{(s)}_0/2K_1$.
\\We can integrate the equation above between $0$ and $t$ and one obtains

$$\norm{h_{n+1}}^2_{\mathcal{H}^s_\eps} + \frac{K^{(s)}_0}{2}\int_0^t\norm{h_{n+1}}^2_{H^s_\Lambda}ds \leq \norm{h_0}^2_{\mathcal{H}^s_\eps} + KE(h_{n+1})E(h_{n}).$$

This is true for all $t>0$, then we define $C = \min\{1,K^{(s)}_0/2\}$,  if $E(h_{n}) \leq C/2K$ we have

$$E(h_{n+1}) \leq \frac{2}{C}\norm{h_0}^2_{\mathcal{H}^s_\eps}.$$

Therefore choosing $M^{(s)} = \min\{C/2K,K^{(s)}_0/2K_1\}$ and $\delta_s \leq \min\{M^{(s)}C/2,M^{(s)}\}$ gives us the expected result.

\end{proof}

\subsubsection{The global existence of solutions}
Now we are able to prove the global existence result:

\begin{theorem}\label{globalexist}
Let $L$ be satisfying assumptions $\emph{(H1')}$, $\emph{(H2')}$ and $\emph{(H3)}$, and let $\Gamma$ be satisfying assumptions $\emph{(H4)}$ and $\emph{(H5)}$.
\\Then it exists $0 < \eps_d \leq 1$ such that for all $s \geq s_0$ (defined in $\emph{(H4)}$), it exists $\delta_s >0$ and for all $0<\eps\leq \eps_d$:
\\\par If $\norm{h_{in}}_{\mathcal{H}^s_\eps} \leq \delta_s$ then there exist a solution of $\eqref{LinBoltz}$ in $C(\R^+,E(\cdot))$ and it satisfies, for some constant $C >0$,
$$E(h) \leq C \norm{h_{in}}^2_{\mathcal{H}_\eps^s}.$$
\end{theorem}

\begin{proof}[Proof of Theorem $\ref{globalexist}$]
Regarding Lemma $\ref{lem3}$, by induction we can strongly bound the sequence $(h_n)_{n\in\N}$, as long as $E(h_{0}) \leq \delta_s$, the constant being defined in Lemma $\ref{lem3}$ . Therefore, defining $h_0$ to be $h_{in}$ at $t=0$ and $0$ elsewhere gives us $E(h_0) = \norm{h_{in}}_{\mathcal{H}^s_\eps} \leq \delta_s.$
\\ Thus, we have the boundedness of the sequence $(h_n)_{n\in\N}$ in $L^{\infty}_tH^s_{x,v} \cap L^1_tH^s_\Lambda$. By compact embeddings into smaller Sobolev's spaces (Rellich theorem) we can take the limit in $\eqref{LBEn}$ as $n$ tends to $+\infty$, since $G_\eps$ and $\Gamma$ are continuous. We obtain $h$ a solution, in $C(\R^+,E(\cdot))$, to
$$\left\{\begin{array}{rl} &\displaystyle{\partial_t h + \frac{1}{\eps}v.\nabla_x h = \frac{1}{\eps^2}L(h) + \frac{1}{\eps}\Gamma(h,h)}\vspace{2mm} 
                         \\ \vspace{2mm}& \displaystyle{h(0,x,v) = h_{in}(x,v).} \end{array}\right.$$
\end{proof}


\subsection{Proof of the exponential decay} \label{subsec:expodecay}
The function constructed above, $h$, is in $\mbox{Ker}(G_\eps)^\bot$ for all $0<\eps \leq 1$. Moreover, this function is clearly a solution of the following equation:
$$\partial_th = G_\eps(h) + \frac{1}{\eps}\Gamma(h,h),$$
with $\Gamma$ satisfying (H5). Therefore, we can use the a priori estimate on solutions of the full perturbative model concerning the time evolution of the $\mathcal{H}^s_\eps$-norm (where we will omit to write the dependence on $s$ for clearness purpose), Proposition $\ref{apriori}$.
$$\frac{d}{dt}\norm{h}_{\mathcal{H}_\eps^s}^2 \leq- K_0\norm{h}^2_{H^s_\Lambda} + K_1 \left(\mathcal{G}^s_{x}(h,h)\right)^2 + \eps^2K_2\left(\mathcal{G}^s_{x,v}(h,h)\right)^2.$$

\bigskip
Moreover, using $\eqref{eq:normeqv}$ and hypothesis (H4) to find:

\begin{eqnarray*}
 \left(\mathcal{G}^s_x(h,h)\right)^2 &\leq& \frac{2C_\Gamma^2}{C_m}\norm{h}^2_{\mathcal{H}^s_\eps}\norm{h}^2_{H^s_\Lambda}
\\ \left(\mathcal{G}^s_{x,v}(h,h)\right)^2 &\leq& \frac{2C_\Gamma^2}{C_m\eps^2}\norm{h}^2_{\mathcal{H}^s_\eps}\norm{h}^2_{H^s_\Lambda}.
\end{eqnarray*}

\bigskip
Hence, $K$ being a constant independent of $\eps$:
$$\frac{d}{dt}\norm{h}_{\mathcal{H}_\eps^s}^2 \leq  \left(K\norm{h}_{\mathcal{H}^s_\eps}^2 - K_0\right)\norm{h}^2_{H^s_\Lambda}.$$
Therefore, one can notice that if $\norm{h_{in}}_{\mathcal{H}^s_\eps}^2 \leq K_0 /2 K$ then we have that $\norm{h}_{\mathcal{H}^s_\eps}^2$ is decreasing in time. Hence, because the $\Lambda$-norm controls the $L^2$-norm which controls the $\mathcal{H}$-norm:

\begin{eqnarray*}
\frac{d}{dt}\norm{h}_{\mathcal{H}_\eps^s}^2 &\leq& -\frac{K_0}{2}\norm{h}^2_{H^s_\Lambda} 
\\ &\leq& -\frac{K_0}{2}\frac{\nu_0^\Lambda}{\nu_1^\Lambda C_{M}}\norm{h}_{\mathcal{H}^s_\eps}^2.
\end{eqnarray*}

Then we have directly, by Gronwall's lemma and setting $\tau_s = K_0\nu_0^\Lambda/4\nu_1^\Lambda C_{M}$,
$$\norm{h}_{\mathcal{H}_\eps^s}^2 \leq \norm{h_{in}}_{\mathcal{H}_\eps^s}^2  e^{-2\tau_s t}$$
as long as $\norm{h_{in}}_{\mathcal{H}^s_\eps}^2 \leq K_0/2K$, which is the expected result with $\delta_s \leq \sqrt{K_0/2K}$.

\section{Exponential decay of $v$-derivatives: proof of Theorem $\ref{decaydv}$} \label{sec:decaydv}
In order to prove this theorem we are going to state a proposition giving an a priori estimate on a solution to the equation $\eqref{LinBoltz}$
$$\partial_th + \frac{1}{\eps}v.\nabla_xh = \frac{1}{\eps^2}L(h) + \frac{1}{\eps}\Gamma(h,h).$$

We remind the reader that we work in $H^s_{x,v}$ with the following positive functional
$$\norm{\cdot}^2_{\mathcal{H}^s_{\eps\bot}} = \sum\limits_{\overset{|j|+|l|\leq s}{|j|\geq 1}}b_{j,l}^{(s)}\norm{\partial^j_l(\mbox{Id} - \pi_L)\cdot}^2_{L^2_{x,v}} + \sum\limits_{|l|\leq s}\alpha_l^{(s)}\norm{\partial^0_l\cdot}^2_{L^2_{x,v}} + \sum\limits_{\overset{|l|\leq s}{i,c_i(l)> 0}}a_{i,l}^{(s)}\eps \langle \partial^{\delta_i}_{l-\delta_i}\cdot,\partial^0_l\cdot\rangle _{L^2_{x,v}}.$$

\bigskip
One can notice that if we choose coefficients $(b_{j,l}^{(s)}),\:(\alpha_l^{(s)}),\:(a_{i,l}^{(s)}) >  0$ such that $\norm{\cdot}_{\mathcal{H}^s_{1\bot}}^2$ is equivalent to 
$$\sum\limits_{\overset{|j|+|l|\leq s}{|j|\geq 1}}\norm{\partial^j_l(\mbox{Id} - \pi_L)\cdot}^2_{L^2_{x,v}} + \sum\limits_{|l|\leq s}\norm{\partial^0_l\cdot}^2_{L^2_{x,v}} $$
 then for all $\eps$ less than some $\eps_0$, $\norm{\cdot}_{\mathcal{H}^s_{\eps\bot}}^2$ is also equivalent to the latter norm with equivalence coefficients not depending on $\eps$.

\bigskip
Moreover, using equation $\eqref{dvcontrolL}$, we have that
$$\norm{\partial_l^jh}^2_{L^2_{x,v}} \leq C_{\pi s} \norm{\partial^0_lh}^2_{L^2_{x,v}} +\norm{\partial_l^jh^\bot}^2_{L^2_{x,v}} \leq  2C_{\pi s} \norm{\partial^0_lh}^2_{L^2_{x,v}} + \norm{\partial_l^jh}^2_{L^2_{x,v}},$$
and therefore
 $$\sum\limits_{\overset{|j|+|l|\leq s}{|j|\geq 1}}\norm{\partial^j_l(Id - \pi_L)}^2_{L^2_{x,v}} + \sum\limits_{|l|\leq s}\norm{\partial^0_l.}^2_{L^2_{x,v}} $$
 is equivalent to the standard Sobolev norm. Thus, we will just construct coefficients $(b_{j,l}^{(s)})$, $(\alpha_l^{(s)})$ and $(a_{i,l}^{(s)})$ so that $\norm{.}_{\mathcal{H}^s_{1\bot}}^2$ is equivalent to the latter norm and then for $\eps$ small enough we will have the equivalence, not depending on $\eps$, between $\norm{\cdot}_{\mathcal{H}^s_{\eps \bot}}^2$ and the $H^s_{x,v}$-norm.


\subsection{An a priori estimate}
In this subsection we will prove the following proposition:

\begin{prop}\label{aprioribot}
If $L$ is a linear operator satisfying the conditions $\emph{(H1')}$, $\emph{(H2')}$ and $\emph{(H3)}$ and $\Gamma$ a bilinear operator satisfying $\emph{(H5)}$ then it exists $0 < \eps_d\leq 1$ such that for all $s$ in $\N^*$,
\begin{enumerate}
\item for $h_{in}$ in $\mbox{Ker}(G_\eps)^\bot$ if we have $h$ an associated solution of $$\partial_t h + \frac{1}{\eps}v\cdot\nabla_x h = \frac{1}{\eps^2}L(h) + \frac{1}{\eps}\Gamma(h,h),$$
\item there exist $K^{(s)}_0,\:K^{(s)}_1, \: (b_{j,l}^{(s)}),\:(\alpha_l^{(s)}),\:(a_{i,l}^{(s)}) >  0 $  such that for all $0<  \eps \leq \eps_d$:
	\begin{itemize}
	\item $\norm{\cdot}_{\mathcal{H}^s_{\eps \bot}} \sim \norm{\cdot}_{H^s_{x,v}}$,
	\item $\forall h_{in} \in H^s_{x,v}\cap \emph{\mbox{Ker}}(G_\eps)^\bot \:,\:$ $$\frac{d}{dt}\norm{h}_{\mathcal{H}_{\eps\bot}^s}^2 \leq - K^{(s)}_0 \left(\frac{1}{\eps^2}\norm{h^\bot}^2_{H^s_\Lambda}+ \sum\limits_{1\leq|l|\leq s}\norm{\partial_l^0h}^2_{L^2_{x,v}}\right)+ K^{(s)}_1\left(\mathcal{G}^s_{x,v}(h,h)\right)^2.$$
	\end{itemize}
\end{enumerate}
\end{prop}

\begin{remark}
We notice here that in front of the microscopic part of $h$ is a negative constant order $-1/\eps^2$ which is the same order than the control derived by Guo in \cite{Gu4} for his dissipation rate.
\end{remark}
We will prove that proposition by induction on $s$.
\\So we take $h_{in}$ in $H^s_{x,v}\cap \mbox{Ker}(G_\eps)^\bot$ and we consider the associated solution of $\eqref{LinBoltz}$, denoted by $h$. One can notice that thanks to (H5), $h$ remains in $\mbox{Ker}(G_\eps)^\bot$ for all times and thus we are allowed to use the inequalities given in the toolbox

\subsubsection{The case $s=1$}
In that case we have 
$$\norm{h}_{\mathcal{H}^1_{\eps\bot}}^2= A\norm{h}_{L^2_{x,v}}^2 + \alpha \norm{\nabla_xh}_{L^2_{x,v}}^2+b \norm{\nabla_vh^\bot}^2_{L^2_{x,v}}+a\eps \langle \nabla_xh,\nabla_vh\rangle _{L^2_{x,v}},$$
with $A$, $\alpha$, $b$ and $a$ strictly positive.

\bigskip
Therefore we can study the time evolution of that operator acting on $h$ by gathering results given in the toolbox. We simply take $A\eqref{h} + \alpha\eqref{dx} + b\eqref{dvbot} + a\eps\eqref{dx,dvbot}$

\begin{eqnarray}
\frac{d}{dt}\norm{h}_{\mathcal{H}^1_{\eps\bot}}^2 &\leq& \frac{1}{\eps^2}\left[K_1^\bot b - \lambda A\right] \norm{h^\bot}^2_\Lambda + \frac{1}{\eps^2}\left[K^\bot ea - \lambda\alpha\right] \norm{\nabla_xh^\bot}^2_\Lambda \nonumber
\\               && + \frac{1}{\eps^2}\left[\frac{1}{4C_{\pi 1}C_\pi C_p}\frac{a}{e} - b\frac{\nu^\Lambda_3}{2}\right] \norm{\nabla_vh^\bot}^2_\Lambda + \left[K_{dx}^\bot b -  \frac{a}{2}\right] \norm{\nabla_xh}^2_{L^2_{x,v}} \nonumber 
\\ && + K(A,\alpha,b,a)\left(\mathcal{G}^1_{x,v}(h,h)\right)^2, \label{boundbot}
\end{eqnarray}
with $K$ a fonction only depending on the coefficients appearing in hypocoercivity hypothesis and independent of $\eps$.

\bigskip
We directly see that we have exactly the same kind of bound as the one we obtain while working on the a priori estimates for the operator $\norm{h}_{\mathcal{H}^1_{\eps}}$, equation $\eqref{boundpriori}$. Therefore we can choose of coefficients $A$, $\alpha$, $b$, $e$ and $a$ in the same way (in the right order) and use the same inequalities to finally obtain the expected result: $\exists K_0, \: K_1 > 0,\: \forall \: 0<\eps \leq 1,$

\begin{eqnarray*}
\frac{d}{dt}\norm{h}_{\mathcal{H}^1_{\eps\bot}}^2 &\leq& -K_0^{(1)}\left(\frac{1}{\eps^2}\norm{h^\bot}^2_\Lambda + \frac{1}{\eps^2}\norm{\nabla_{x}h^\bot}^2_\Lambda+\frac{1}{\eps^2}\norm{\nabla_{v}h^\bot}^2_\Lambda+ \norm{\nabla_xh}^2_{L^2_{x,v}} \right) 
\\ && +  K_1^{(1)}\left(\mathcal{G}^1_{x,v}(h,h)\right)^2,
\end{eqnarray*}
with the constants $K_0^{(1)}$ and $K_1^{(1)}$ independent of $\eps$, and $\norm{h}_{\mathcal{H}^1_{1\bot}}^2$ equivalent to  $\norm{h}_{L^2_{x,v}}^2 + \norm{\nabla_xh}_{L^2_{x,v}}^2+ \norm{\nabla_vh^\bot}^2_{L^2_{x,v}}$. Therefore, for all $\eps$ small enough we have the expected result in the case $s=1$.


\subsubsection{The induction in higher order Sobolev spaces}

Then we assume that the theorem is true up to the integer $s-1$, $s > 1$. Then we suppose that $L$ satisfies (H1'), (H2') and (H3) and we consider $\eps$ in $(0,1]$.
\par Since $h_{in}$ is in $\mbox{Ker}(G_\eps)^\bot$, $h$ belongs to $\mbox{Ker}(G_\eps)^\bot$ for all $t$ and so we can use the results given in the toolbox.
\par As in the proofs of previous sections, we define on $H^s_{x,v}$:
\begin{eqnarray*}
F_s(t) &=& \sum\limits_{\overset{|j|+|l|=s}{|j|\geq 2}} B\norm{\partial^j_lh^\bot}^2_{L^2_{x,v}} + B'\sum\limits_{\overset{|l|=s}{i,c_i(l)> 0}} Q_{l,i}(t),
\\Q_{l,i}(t) &=& \alpha\norm{\partial^0_lh}^2_{L^2_{x,v}} + b\norm{\partial^{\delta_i}_{l-\delta_i}h^\bot}^2_{L^2_{x,v}} + a\eps\langle \partial^{\delta_i}_{l-\delta_i}h,\partial^0_lh\rangle _{L^2_{x,v}},
\end{eqnarray*}
where the constants, strictly positive, will be chosen later.

\bigskip
Like in the section above, we shall study the time evolution of every term involved in $F_s$ in order to bound above $\frac{dF_s}{dt}(t)$ with expected coefficients. However, in this subsection we will need to control all the $Q_{l,i}$'s in the same time rather than treating them separately as we did in the proof of Proposition $\eqref{apriori}$, because the toolbox tells us that each $Q_{l,i}$ is controlled by quantities appearing  in the others.


\subsubsection{The time evolution of $\sum Q_{l,i}$}
Gathering the toolbox inequalities in the following way: $\alpha\eqref{d0l} + b\eqref{ddeltaibot} + a\eps\eqref{deltai,d0lbot}$. This yields, because $0<\eps\leq 1$ and $\mbox{Card}\{i,c_i(l)> 0\}\leq d$,
 
 \begin{eqnarray*}
\frac{d}{dt}\left(\sum\limits_{\overset{|l|=s}{i,c_i(l)> 0}}Q_{l,i}(t)\right) &\leq& \frac{1}{\eps^2}\left[\tilde{K}^\bot e a   - \lambda\alpha\right] \sum\limits_{|l|=s}\norm{\partial^0_lh^\bot}^2_\Lambda
\\               && + \frac{1}{\eps^2}\left[\frac{1}{4C_{\pi s}C_\pi d}\frac{a}{e}
-\nu_5^\Lambda b\right]\sum\limits_{\overset{|l|=s}{i,c_i(l)> 0}}\norm{\partial^{\delta_i}_{l-\delta_i}h^\bot}^2_\Lambda
\\               && + \left[K_{dl}^\bot d b-\frac{a}{2}\right]\sum\limits_{|l|=s}\norm{\partial^0_lh}^2_{L^2_{x,v}} + \frac{a}{4}\sum\limits_{|l|\leq s-1}\norm{\partial^0_lh}^2_{L^2_{x,v}}
\\               && + \frac{bK_{s-1}^\bot}{\eps^2}\left(\sum\limits_{\overset{|l|+ |j|=s}{i,c_i(l)> 0}}1\right)\norm{h^\bot}^2_{H^{s-1}_{x,v}} + K(\alpha,b,a,e)\left(\mathcal{G}^s_{x,v}(h,h)\right)^2,
\end{eqnarray*}
with $K$ a fonction only depending on the coefficients appearing in hypocoercivity hypothesis and independent of $\eps$.
\\One can notice that except for the terms in $\norm{h}_{H^{s-1}_{x,v}}$ and $\sum\limits_{|l|\leq s-1}\norm{\partial^0_lh}^2_{L^2_{x,v}}$, we have exactly the same bound as in the case $s=1$, equation $\eqref{boundbot}$.  Therefore we can choose $\alpha$, $b$, $a$, $e$, independently of $\eps$ and $\Gamma$  such that it exists $K'_0>0$, $K'_1>0$ and $C_0, C_1 >0$ such that for all $0<\eps\leq 1$:

\begin{itemize}
\item $\sum\limits_{\overset{|l|=s}{i,c_i(l)> 0}}Q_{l,i}(t) \sim \sum\limits_{\overset{|l|=s}{i,c_i(l)> 0}}\left(\norm{\partial^0_lh}^2_{L^2_{x,v}} + \norm{\partial^{\delta_i}_{l-{\delta_i}} h^\bot}^2_{L^2_{x,v}}\right)$,
\item 
\begin{eqnarray*}
\frac{d}{dt}\sum\limits_{\overset{|l|=s}{i,c_i(l)> 0}}Q_{l,i}(t) \leq &-&K'_0\left(\frac{1}{\eps^2}\sum\limits_{|l|=s}\norm{\partial^0_lh^\bot}^2_\Lambda +\frac{1}{\eps^2} \sum\limits_{\overset{|l|=s}{i,c_i(l)> 0}}\norm{\partial^{\delta_i}_{l-{\delta_i}}h^\bot}^2_\Lambda\right. 
\\ && \:\:\:\:\:\:\:\:\:\:\: \:\:\:\left.+ \sum\limits_{|l|=s}\norm{\partial^0_lh}^2_{L^2_{x,v}}\right)
\\ &+& \frac{C_0}{\eps^2}\norm{h^\bot}_{H^{s-1}_{x,v}}^2+ C_1\sum\limits_{|l|\leq s-1}\norm{\partial^0_lh}^2_{L^2_{x,v}} + K'_1\left(\mathcal{G}^s_{x,v}(h,h)\right)^2.
\end{eqnarray*}
\end{itemize}


\subsubsection{The time evolution of $F_s$ and conclusion}
We can finally obtain the time evolution of $F_s$, using $\frac{d}{dt}\norm{\partial^j_lh^\bot}^2_{L^2_{x,v}}$, equation $\eqref{djlbot}$, so that there is no more $\eps$ in front of the $\Gamma$ term:

\begin{eqnarray*}
\frac{d}{dt}F_s(t) &\leq& -B\frac{\nu_5^\lambda}{\eps^2} \sum\limits_{\overset{|j|+|l|=s}{|j|\geq 2}}\norm{\partial^j_lh^\bot}^2_\Lambda + B\frac{9(\nu_1^\Lambda)^2d}{2(\nu_0^\Lambda)^2\nu_5^\Lambda}\sum\limits_{\overset{|j|+|l|=s}{|j|\geq 2}}\sum\limits_{i,c_i(j)> 0}\norm{\partial^{j-\delta_i}_{l+\delta_i}h^\bot}^2_\Lambda
\\                 && -K'_0B'\left(\frac{1}{\eps^2}\sum\limits_{|l|=s}\norm{\partial^0_lh^\bot}^2_\Lambda +\frac{1}{\eps^2} \sum\limits_{\overset{|l|=s}{i,c_i(l)> 0}}\norm{\partial^{\delta_i}_{l-{\delta_i}}h^\bot}^2_\Lambda + \sum\limits_{|l|=s}\norm{\partial^0_lh}^2_{L^2_{x,v}}\right)
\\                 && + \left(\sum\limits_{\overset{|j|+|l|=s}{|j|\geq 2}}BK_{dl}^\bot  + B'C_1\right) \sum\limits_{|l|\leq s-1}\norm{\partial^0_lh}^2_{L^2_{x,v}} 
\\                 &&  + \frac{1}{\eps^2}\left[\sum\limits_{\overset{|j|+|l|=s}{|j|\geq 2}}BK_{s-1}^\bot+ B'C_0\right]\norm{h^\bot}^2_{H^{s-1}_{x,v}}
\\                 && +\left[\sum\limits_{\overset{|j|+|l|=s}{|j|\geq 2}}\frac{3B\nu_1^\Lambda}{\nu_0^\Lambda\nu_5^\Lambda} + B'K'_1 \right]\left(\mathcal{G}^s_{x,v}(h,h)\right)^2,
\end{eqnarray*}

\bigskip
Therefore we obtain the same bound (except $\sum\limits_{|l|\leq s-1}\norm{\partial^0_lh}^2_{L^2_{x,v}}$)  as in the proof of Proposition $\ref{apriori}$, equation $\eqref{Fkpriori}$, and so by choosing coefficients in the same way we have that it exists $C_+^{(s)}>0$, $0<\eps_d\leq 1$ and $K^{(s*)}_1 >0$, none of them depending on $\eps$, such that for all $0<\eps\leq\eps_d$:

\begin{eqnarray*}
\frac{d}{dt}F_s(t) \leq &&C_+^{(s)} \left(\frac{1}{\eps^2}\sum\limits_{|j|+|l| \leq s-1} \norm{\partial^j_lh^\bot}^2_\Lambda + \sum\limits_{|l|\leq s-1}\norm{\partial^0_lh}^2_{L^2_{x,v}}\right)
\\                      && - \left(\frac{1}{\eps^2}\sum\limits_{|j|+|l|=s} \norm{\partial^j_lh^\bot}^2_\Lambda + \sum\limits_{|l|=s}\norm{\partial^0_lh}^2_{L^2_{x,v}}\right)
\\                              &&+ K^{(s*)}_1 \left(\mathcal{G}^s_{x,v}(h,h)\right)^2.
\end{eqnarray*}
\bigskip

This inequality is true for all $s$ and therefore we can take a linear combination of the $F_s$ to obtain the required result. Using the induction hypothesis on $F_1$ up to $F_{s-1}$ we also have the equivalence of norms.


\subsection{The exponential decay: proof of Theorem $\ref{decaydv}$}

Thanks to Theorem $\ref{perturb}$, we know that we have a solution to the equation $\eqref{LinBoltz}$ for any given $h_{in}$ small enough in the standard Sobolev norm. Call $h$ the associated solution of $h_{in} \in H^s_{x,v}$ to $\eqref{LinBoltz}$. Since the existence has been proved we can use the a priori estimate above and the Proposition $\ref{aprioribot}$.
\\ Thus we have
$$\frac{d}{dt}\norm{h}_{\mathcal{H}_{\eps\bot}^s}^2 \leq - K^{(s)}_0 \left(\frac{1}{\eps^2}\norm{h^\bot}^2_{H^s_\Lambda}+ \sum\limits_{1\leq|l|\leq s}\norm{\partial_l^0h}^2_{L^2_{x,v}}\right)+ K^{(s)}_1\left(\mathcal{G}^s_{x,v}(h,h)\right)^2.$$

\bigskip
As before we can use $\eqref{L2L2lambdafluid}$ (equivalence of norms $L^2_{x,v}$ and $\Lambda$ on the fluid part) to get, for $|l|>1$,
$$\norm{\partial^0_lh}^2_\Lambda \leq C'\left(\norm{\partial^0_lh^\bot}^2_\Lambda + \norm{\partial^0_lh}^2_{L^2_{x,v}}\right),$$
and for the case $|l|\leq1$ we can apply the Poincare inequality $\eqref{poincare}$ together with the equivalence of the $L^2_{x,v}$-norm and the $\Lambda$-norm on the fluid part $\pi_L$, $\eqref{L2L2lambdafluid}$ to get
$$\exists C, C' > 0, \left\{\begin{array}{ll}\norm{h}^2_\Lambda &\leq C\left(\norm{h^\bot}^2_\Lambda + \frac{1}{2}\norm{\nabla_xh}^2_{L^2_{x,v}}\right),
                                   \\\norm{\nabla_xh}^2_\Lambda &\leq C'\left(\norm{\nabla_xh^\bot}^2_\Lambda + \frac{1}{2}\norm{\nabla_xh}^2_{L^2_{x,v}}\right).\end{array}\right.$$

\bigskip
Then we get that
\begin{eqnarray*}
\frac{d}{dt}\norm{h}_{\mathcal{H}_{\eps\bot}^s}^2 &\leq& - K^{(s)}_0 \left( \sum\limits_{\overset{|j|+|l|\leq s}{|j|\geq 1}}\norm{\partial^j_lh^\bot}^2_\Lambda + \sum\limits_{|l|\leq s}\norm{\partial^0_lh}^2_\Lambda\right) + K^{(s)}_1\left(\mathcal{G}^s_{x,v}(h,h)\right)^2
\\&\leq& - K^{(s*)}_0 \norm{h}^2_{H^s_\Lambda} + K^{(s)}_1\left(\mathcal{G}^s_{x,v}(h,h)\right)^2.
\end{eqnarray*}

\bigskip
Then for $s\geq s_0$, defined in (H4), and because $\Gamma$ satisfies (H4) we can write
$$\frac{d}{dt}\norm{h}_{\mathcal{H}_{\eps\bot}^s}^2 \leq \left(K^{(s)}_1C_\Gamma^2\norm{h}^2_{H^s_{x,v}}-K^{(s*)}_0\right)\norm{h}^2_{H^s_\Lambda}.$$

Because $\norm{h}_{\mathcal{H}_{\eps\bot}^s}$ and $\norm{h}^2_{H^s_{x,v}}$ are equivalent, independently of $\eps$, we finally have
$$\frac{d}{dt}\norm{h}_{\mathcal{H}_{\eps\bot}^s}^2 \leq \left(K^{(s)}_1C_\Gamma^2C\norm{h}_{\mathcal{H}_{\eps\bot}^s}^2-K^{(s*)}_0\right)\norm{h}^2_{H^s_\Lambda}.$$

Therefore if 
$$\norm{h_{in}}_{\mathcal{H}_{\eps\bot}^s}^2 \leq \frac{K^{(s*)}_0}{2K^{(s)}_1C_\Gamma^2C}$$
we have that $\norm{h}_{\mathcal{H}_{\eps\bot}^s}^2$ is always decreasing on $\R^+$ and so for all $t>0$
$$\frac{d}{dt}\norm{h}_{\mathcal{H}_{\eps\bot}^s}^2 \leq -\frac{K^{(s*)}_0}{2K^{(s)}_1C_\Gamma^2C}\norm{h}^2_{H^s_\Lambda}.$$
And the $H^s_\Lambda$-norm controls the $H^s_{x,v}$-norm which is equivalent to the $\mathcal{H}^s_{\eps\bot}$-norm. Thus applying Gronwall's lemma gives us the expected exponential decay.

\section{Incompressible Navier-Stokes Limit: proof of Theorem $\ref{hydrolim}$} \label{sec:hydrolim}

In this section we consider $s \geq s_0$, $0<\eps\leq\eps_d$ and we take $h_{in}$ in $H^s_{x,v}$ such that $\norm{h_{in}}_{\mathcal{H}^s_\eps}\leq \delta_s$.
\\ Therefore we know, thanks to theorem $\ref{perturb}$, that we have a solution $h_\eps$ to the linearized Boltzmann equation
$$\partial_th_\eps + \frac{1}{\eps}v.\nabla_xh_\eps = \frac{1}{\eps^2}L(h_\eps) + \frac{1}{\eps}\Gamma(h_\eps,h_\eps),$$
with $h_\eps(0,x,v) = h_{in}(x,v)$. Moreover, we also know that $(h_\eps)$ tends weakly-* to $h$ in $L^\infty_t(H^s_xL^2_v)$.

\bigskip
The first step towards the proof of Theorem $\ref{hydrolim}$ is to derived a convergence rate in finite time. Then, as described in Section $\ref{subsec:strategy}$, we shall interpolate this result with the exponential decay behaviour of our solutions in order to obtain a global in time convergence.


\subsection{A convergence in finite time}

In Remark $\ref{Linfinityx}$, we define $V_T(\eps)$ and prove the following result
$$\forall T>0,\quad V_T(\eps) = \sup\limits_{t\in [0,T]}\norm{h_\eps - h}_{L^\infty_xL^2_v} \to 0, \: \mbox{as}\:\: \eps \to 0.$$
Thanks to this remark we can give an explicit convergence in finite time.

\begin{theorem}\label{hydrolimT}
Consider $s \geq s_0$ and $h_{in}$ in $H^s_{x,v}$ such that $\norm{h_{in}}_{\mathcal{H}^s_\eps}\leq \delta_s$.
\par Then, $(h_\eps)_{\eps>0}$ exists for all $0<\eps\leq \eps_d$ and converges weakly* in $L^\infty_t(H^s_xL^2_v)$ towards $h$ such that $h \in \mbox{Ker}(L)$, with $\nabla_x\cdot u=0$ and $\rho + \theta =0$.

\bigskip
Furthermore, $\int_0^Thdt$  belongs to $H^s_xL^2_v$ and  it exists $C>0$ such that for all $T>0$,
$$\norm{\int_0^Thdt - \int_0^T h_\eps dt}_{H^s_xL^2_v} \leq C \max\{\sqrt{\eps},\sqrt{T\eps},TV_T(\eps)\}.$$
One can have a strong convergence in $L^2_{[0,T]}H^s_xL^2_v$ only if $h_{in}$ is in $\emph{\mbox{Ker}}(L)$ with $\nabla_x \cdot u_{in} = 0$ and $\rho_{in} + \theta_{in} =0$ (initial layer conditions). Moreover, in that case we have, for all $T>0$,
$$\norm{h - h_\eps}_{L^2_{[0,T]}H^s_xL^2_v} \leq C \max\{\sqrt{\eps}, \sqrt{T}V_T(\eps)\},$$
and for all $\delta$ in $[0,1]$, if $h_{in}$ belongs to $H^{s+\delta}_xL^2_v$,
$$\sup\limits_{t\in [0,T]}\norm{h - h_\eps}_{H^s_xL^2_v}(t) \leq C \max\{\eps^{\min(\delta,1/2)}, V_T(\eps)\}.$$
\end{theorem}

\bigskip
\begin{remark}
We mention here that the obligation of an integration in time for non special initial condition is only due to the linear part $\eps^{-2}L - \eps^{-1}v \cdot\nabla_x$, whereas the case $T = +\infty$ is prevented by the second order term $\Gamma$.
\end{remark}

\bigskip
We proved in the linear case, theorem $\ref{lin}$, that the linear operator $G_\eps = \eps^{-2}L-\eps^{-1}v\cdot\nabla_x$ generates a semigroup $e^{tG_\eps}$ on $H^s_{x,v}$. Therefore we can use Duhamel's principle to rewrite our equation under the following form, defining $u_\eps =\Gamma(h_\eps,h_\eps)$,

\begin{eqnarray}
h_\eps &=& e^{tG_\eps}h_{in} + \int_0^t\frac{1}{\eps}e^{(t-s)G_\eps}u_\eps(s) ds \nonumber
\\ &:=& U^\eps h_{in} + \Psi^\eps(u_\eps) . \label{duhamel}
\end{eqnarray}

\bigskip
The article by Ellis and Pinsky \cite{EP} gives us a Fourier theory in $x$ of the semigroup $e^{tG_\eps}$ and therefore we are going to use it to study the strong limit of $U^\eps h_{in} $ and $\Psi^\eps(u_\eps)$ as $\eps$ tends to $0$. We will denote by $\mathcal{F}_x$ the Fourier transform in $x$ on the torus (which is discrete) and $n$ the discrete variable associated in $\Z^d$.
\\ From \cite{EP}, we are using Theorem $3.1$, rewriten thanks with the Proposition $2.6$ and the Appendix $II$ with $\delta = \lambda/4$ in Proposition $2.3$, to get the following theorem

\begin{theorem} \label{fourier}
There exists $n_0 \in \R^{*+}$, there exists functions
\begin{itemize}
\item $\func{\lambda_j}{[-n_0,n_0]}{\C}, \: -1\leq j \leq 2,$ $C^\infty$
\item $\function{e_j}{[-n_0,n_0]\times \mathbb{S}^{d-1}}{L^2_v}{(\zeta,\omega)}{e_j(\zeta,\omega)}, \: -1\leq j \leq d,$ $C^\infty$ in $\zeta$ and $C^0$ in $\omega$,
\end{itemize}
such that
\begin{enumerate}
\item for all $-1\leq j\leq 2$, $\lambda_j(\zeta) = i\alpha_j\zeta - \beta_j \zeta^2 + \gamma_j(\zeta),$ where $\alpha_j \in \R$, with $\alpha_{0} =\alpha_2 = 0$, $\beta_j <0$ and $\abs{\gamma_j(\zeta)} \leq C_\gamma \abs{\zeta}^3$ with $\abs{\gamma_j(\zeta)} \leq \frac{\beta_j}{2}\abs{\zeta}^2,$
\item for all $-1\leq j \leq d$
		\begin{itemize}
		\item $e_j(\zeta,\omega) = e_{0j}(\omega) + \zeta e_{1j}(\omega) + \zeta^2e_{2j}(\zeta,\omega),$
		\item $e_{0-1}(\omega)(v) = e_{01}(-\omega)(v) = A\left(1 - \omega.v + \frac{\abs{v}^2-d}{2}\right)\mu(v)^{1/2},$
		\end{itemize}
\item we have $e^{tG_\eps} = \mathcal{F}^{-1}_x \hat{U}(t/\eps^2, \eps n,v)\mathcal{F}_x$ where $$\hat{U}(t,n,v) = \sum\limits_{j=-1}^{2} \hat{U}_j(t,n,v) + \hat{U}_R(t,n,v)$$ with the following properties
		\begin{itemize}
		\item for $-1\leq j \leq 2, \hat{U}_j(t,n,v) = \chi_{\abs{n}\leq n_0}e^{t\lambda_j(\abs{n})} P_j\left(\abs{n},\frac{n}{\abs{n}}\right)(v),$
		\item for $-1 \leq j \leq 1, P_j\left(\abs{n},\frac{n}{\abs{n}}\right) = e_j\left(\abs{n},\frac{n}{\abs{n}}\right) \otimes e_j\left(\abs{n},\frac{-n}{\abs{n}}\right),$
		\item $P_2\left(\abs{n},\frac{n}{\abs{n}}\right) = \sum\limits_{j=2}^{d}e_j\left(\abs{n},\frac{n}{\abs{n}}\right) \otimes e_j\left(\abs{n},\frac{-n}{\abs{n}}\right),$
		\item for $-1\leq j \leq 2, P_j\left(\abs{n},\frac{n}{\abs{n}}\right) = P_{0j}\left(\frac{n}{\abs{n}}\right) + \abs{n} P_{1j}\left(\frac{n}{\abs{n}}\right) + \abs{n}^2P_{2j}\left(\abs{n},\frac{n}{\abs{n}}\right),$
		\item $\sum\limits_{j=-1}^{2} P_{0j} = \pi_L,$
		\item it exists $C_{R},\sigma >0$ such that for all $t \in \R^+$ and all $n \in \Z^d$, $$|||\hat{U}_R(t,n,v)|||_{L^2_v}\leq C_R e^{-\sigma t}.$$
		\end{itemize}
\end{enumerate}
\end{theorem}

\begin{remark}
This decomposition of the spectrum of the linear operator is based on a low and high frequencies decomposition. It shows that the spectrum of the whole operator can be viewed as a perturbation of the spectrum of the homogeneous linear operator. It can be divided into large eigenvalues, which are negative and therefore create a strong semigroup property for the remainder term, and small eigenvalues around the origin that are smooth perturbations of the homogeneous ones.
\end{remark}

This theorem gives us all the tools we need to study the convergence as $\eps$ tends to $0$ since we have an explicit form for the Fourier transform of the semigroup. We also know that this semigroup commutes with the pure $x$-derivatives. Therefore, studying the convergence in the $L^2_xL^2_v$-norm will be enough to obtain the desired result in the $H^s_xL^2_v$-norm.
\\ We are going to prove the following convergences in the different settings stated by Theorem $\ref{hydrolim}$
\begin{enumerate}
\item $U^\eps h_{in}$ tends to $V(t,x,v)h_{in}$ with $V(0,x,v)h_{in}=V(0)(h_{in})(x,v)$ where $V(0)$ the projection on the subset of $\mbox{Ker}(L)$ consisting in functions $g$ such that $\nabla_x\cdot u_g =0$ and $\rho_g + \theta_g = 0$,
\item $\Psi^\eps(u_\eps)$ converges to $\Psi(h,h)$ with $\Psi(h,h)(t=0) = 0$.
\end{enumerate}


\subsubsection{Study of the linear part}

We remind here that we have$$U^\eps h_{in} = \mathcal{F}^{-1}_x \hat{U}^\eps(t,n,v) \hat{h}_{in}(n,v)$$ with
\begin{eqnarray*}
\hat{U}^\eps(t,n,v) &=& \sum\limits_{j=-1}^{2} \hat{U}_j^\eps(t,n,v) + \hat{U}^\eps_R(t,n,v),
\\ \hat{U}_j^\eps(t,n,v) &=& \chi_{\abs{\eps n}\leq n_0}e^{\frac{i \alpha_j t \abs{n}}{\eps} - \beta_j t \abs{n}^2 + \frac{t}{\eps^2}\gamma_j(\abs{\eps n})} \left[P_{0j}\left(\frac{n}{\abs{n}}\right) + \eps\abs{n} \tilde{P}_{1j}\left(\abs{\eps n},\frac{n}{\abs{n}}\right)\right].
\end{eqnarray*}

\bigskip
We can decompose $\hat{U}^\eps_j$ into four different terms

\begin{eqnarray}
\hat{U}^\eps_j(t,n,v) &=& e^{\frac{i \alpha_j t \abs{n}}{\eps} - \beta_j t \abs{n}^2}P_{0j}\left(\frac{n}{\abs{n}}\right)  \nonumber
\\    &&  + \chi_{\abs{\eps n}\leq n_0}e^{\frac{i \alpha_j t \abs{n}}{\eps} - \beta_j t \abs{n}^2}\left(e^{\frac{t}{\eps^2}\gamma_j(\abs{\eps n})} -1\right)P_{0j}\left(\frac{n}{\abs{n}}\right) \label{decomposition}
\\ && +\chi_{\abs{\eps n}\leq n_0}e^{\frac{i \alpha_j t \abs{n}}{\eps} - \beta_j t \abs{n}^2 +\frac{t}{\eps^2}\gamma_j(\abs{\eps n})} \eps\abs{n} \tilde{P}_{1j}\left(\abs{\eps n},\frac{n}{\abs{n}}\right) \nonumber
\\ &&  + \left(\chi_{\abs{\eps n}\leq n_0} - 1\right)e^{\frac{i \alpha_j t \abs{n}}{\eps} - \beta_j t \abs{n}^2}P_{0j}\left(\frac{n}{\abs{n}}\right). \nonumber
\\&=& U^\eps_{0j} + U^\eps_{1j}+U^\eps_{2j}+U^\eps_{3j} .\nonumber
\end{eqnarray}

\begin{remark} \label{noeps}
One can notice that $U^\eps_{00}$ and $U^\eps_{02}$ do not depend on $\eps$, since $\alpha_0 = \alpha_2 =0$.
\end{remark}

We are going to study each of these four terms in two different lemmas and then add a last lemma to deal with the remainder term $U_R h_{in}$. The lemmas will be proven in Appendix $\ref{appendix:hydro}$.

\begin{lemma}\label{Ueps0j}
For $\alpha_j \neq 0$ ($j = \pm 1$) we have that it exists $C_0 >0$ such that for all $T \in [0,+\infty]$
$$\norm{\int_0^T U^\eps_{0j}h_{in} dt}^2_{L^2_xL^2_v} \leq C_0 \eps^2 \norm{h_{in}}^2_{L^2_xL^2_v}.$$
Moreover we have a strong convergence in the $L^2_{[0,+\infty)}L^2_xL^2_v$-norm if and only if $h_{in}$ satisfies $\nabla_x\cdot u_{in} = 0$ and $\rho_{in} + \theta_{in} =0$. In that case we have $U^\eps_{0j}h_{in} = 0$.
\end{lemma}

\begin{lemma}\label{Uepslj}
For $-1\leq j \leq 2$ and for $1\leq l \leq 3$ we have that the three following inequalities hold for $U^\eps_{lj}$
\begin{itemize}
\item $\exists C_l>0, \forall T>0, \quad \displaystyle{\norm{\int_0^T U^\eps_{lj}h_{in} dt}^2_{L^2_xL^2_v} \leq C_l \eps^2 \norm{h_{in}}^2_{L^2_xL^2_v}},$
\item $\exists C_l'> 0, \quad \norm{U^\eps_{lj}h_{in}}^2_{L^2_{[0,+\infty)}L^2_xL^2_v}\leq C_l' \eps^2\norm{h_{in}}^2_{L^2_xL^2_v},$
\item $\forall \delta \in [0,1], \exists C^{(l)}_\delta >0, \forall t>0,  \quad \norm{U^\eps_{lj}h_{in}(t)}^2_{L^2_xL^2_v}\leq C^{(l)}_\delta \eps^{2\delta}\norm{h_{in}}^2_{H^\delta_xL^2_v}.$
\end{itemize}
\end{lemma}

\begin{lemma}\label{UR}
For the remainder term we have the two following inequalities
\begin{itemize}
 \item $\exists C_4>0, \forall T>0, \quad \displaystyle{\norm{\int_0^T U^\eps_R h_{in} dt}^2_{L^2_xL^2_v} \leq  C_4 T \eps^2 \norm{h_{in}}^2_{L^2_xL^2_v}},$
 \item $\exists C_4' >0, \quad \norm{U^\eps_Rh_{in}}^2_{L^2_{[0,+\infty)}L^2_xL^2_v}\leq C_4' \eps^2\norm{h_{in}}^2_{L^2_xL^2_v},$
 \item $\forall t_0 >0,\exists C_r>0, \forall t>t_0, \quad \displaystyle{\norm{U_Rh_{in}(t)}^2_{L^2_xL^2_v} \leq \frac{C_r}{\sqrt{t_0}}\eps \norm{h_{in}}^2_{L^2_xL^2_v}}.$
 \end{itemize}
Moreover, the strong convergence up to $t_0=0$ is possible if and only if $h_{in}$ is in $\mbox{Ker}(L)$. In that case we have
$$\forall \delta \in [0,1], \exists C^{(R)}_\delta >0, \forall t>0,  \: \norm{U^\eps_{R}h_{in}}^2_{L^2_xL^2_v}\leq C^{(R)}_\delta \eps^{2\delta}\norm{h_{in}}^2_{H^\delta_xL^2_v}.$$
\end{lemma}

\bigskip
Therefore, gathering lemmas $\ref{Ueps0j}$, $\ref{Uepslj}$ and $\ref{UR}$ and reminding Remark $\ref{noeps}$, we proved that, as $\eps$ tends to $0$, $\left(e^{tG_\eps}h_{in}\right)$ converges to 

\begin{equation}\label{V(0)}
V(t,x,v)h_{in}(x,v) = \mathcal{F}^{-1}_x \left[e^{ - \beta_0 t \abs{n}^2}P_{00}\left(\frac{n}{\abs{n}}\right) + e^{- \beta_2 t \abs{n}^2}P_{02}\left(\frac{n}{\abs{n}}\right)\right]\mathcal{F}_x h_{in}.
\end{equation}

The convergence is strong when we consider the average in time and is strong in $L^2_tH^s_xL^2_v$ ( and in $C([0,+\infty),H^s_xL^2_v)$ if $h_{in}$ is in $H^{s+0}_xL^2_v$ ) if an only if  both conditions found in Lemma $\ref{Ueps0j}$ and Lemma $\ref{UR}$ are satisfied. That is to say $h_{in}$ belongs to $\mbox{Ker}(L)$ with  $\nabla_x \cdot u_{in} = 0$ and $\rho_{in} + \theta_{in} =0$.
\par Moreover this also allows us to see that $V(0,x,v)h_{in}=V(0)(h_{in})(x,v)$ where $V(0)$ is the projection on the subset of $\mbox{Ker}(L)$ consisting in functions $g$ such that $\nabla_x \cdot u_g =0$ and $\rho_g + \theta_g = 0$.


\subsubsection{Study of the bilinear part}
We recall here that $u_\eps = \Gamma(h_\eps,h_\eps)$. Therefore, by hypothesis (H5), $u_\eps$ belongs to $\mbox{Ker}(L)^\bot$. Then we know that for all $-1\leq j \leq 2$, $P_{0j}\left(\frac{n}{\abs{n}}\right)$ is a projection onto a subspace of $\mbox{Ker}(L)$. Therefore we have that, in the Fourier space,
$$P_j\left(\abs{\eps n},\frac{n}{\abs{n}}\right)  \hat{u}_{\eps} =  \abs{\eps n} P_{1j}\left(\frac{n}{\abs{n}}\right)\hat{u}_{\eps} + \abs{\eps n}^2P_{2j}\left(\abs{\eps n},\frac{n}{\abs{n}}\right)\hat{u}_{\eps}.$$

\bigskip
Thus, recalling that $$\Psi^\eps( u_\eps) = \int_0^t\frac{1}{\eps}e^{(t-s)G_\eps}u_\eps(s) ds,$$
we can decompose it
$$ \Psi^\eps( u_\eps) = \sum\limits_{j=-1}^{2}\psi^\eps_j( u_\eps) + \psi^\eps_R( u_\eps), $$
with

\begin{eqnarray*}
\psi^\eps_{j} (u_\eps) &=& \mathcal{F}^{-1}_x\chi_{\abs{\eps n}\leq n_0}\int_0^t e^{\frac{i \alpha_j (t-s) \abs{n}}{\eps} - \beta_j (t-s) \abs{n}^2 + \frac{t-s}{\eps^2}\gamma_j(\abs{\eps n})} \abs{n}\left(P_{1j} + \eps\abs{n} P_{2j}\right)\hat{u}_\eps(s)ds.
\\ &:=& \psi^\eps_{0j} (u_\eps) + \psi^\eps_{1j}  (u_\eps)+ \psi^\eps_{2j} (u_\eps) + \psi^\eps_{3j} ( u_\eps) ,
\end{eqnarray*}
where we have used the same decomposition as in the linear case, equation $\eqref{decomposition}$, substituting $t$ by $t-s$, $P_{0j}$ by $\abs{n}P_{1j}$ and $\tilde{P}_{1j}$ by $\abs{n}P_{2j}$. And

$$\psi^\eps_R( u_\eps) = \int_0^t \frac{1}{\eps}U^\eps_R(t-s)u_\eps(s) ds.$$

\bigskip
Like the linear case, Remark $\ref{noeps}$, $\psi^\eps_{00}$ and $\psi^\eps_{02}$ do not depend on $\eps$ and we are going to prove the convergence towards $\Psi(u) = \mathcal{F}^{-1}_x\left[\psi^\eps_{00} (u)+\psi^\eps_{02} (u)\right]\mathcal{F}_x$, where $u = \Gamma(h,h)$. To establish such a result we are going to study each term in three different lemmas and then a fourth one will deal with the remainder term. The lemmas will be proven in Appendix $\ref{appendix:hydro}$.

\begin{lemma} \label{psieps0j}
For $\alpha_j \neq 0$ ($j = \pm 1$) we have the following inequality for $\psi^\eps_{0j}$:
$$\exists \tilde{C}_0 > 0, \forall T>0,\quad \norm{\int_0^T\psi^\eps_{0j}(u_\eps)dt}^2_{L^2_xL^2_v} \leq \tilde{C}_0 T^2 \eps^2 E(h_\eps)^2.$$
\end{lemma}

\begin{remark}\label{dtu}
We know that $(h_\eps)_{\eps >0}$  is bounded in $L^\infty_tH^s_xL^2_v$ (see theorems $\ref{perturb}$ and $\ref{decaydv}$).
\end{remark}

\bigskip
This remark gives us the strong convergence to $0$ of the average in time and the strong convergence to $0$ without averaging in time as long as $h_{in}$ belongs to $\mbox{Ker}(L)$ in Lemma $\ref{psieps0j}$.

\begin{lemma}\label{psiepslj}
For $-1\leq j \leq 2$ and for $1\leq l \leq 3$ we have that the three following inequalities hold for $\psi^\eps_{lj}$
\begin{itemize}
\item $\exists \tilde{C}_l > 0, \forall T>0, \quad \norm{\int_0^T\psi^\eps_{lj}(u_\eps)dt}^2_{L^2_xL^2_v} \leq \tilde{C}_l T\eps^2 E(h_\eps)^2, $
\item  $\exists \tilde{C}'_l > 0, \forall T>0, \quad \norm{\psi^\eps_{lj}(u_\eps)}^2_{L^2_{[0,T]}L^2_xL^2_v} \leq \tilde{C}'_l \eps^2 E(h_\eps)^2, $
\item $\forall \abs{\delta} \in [0,1], \exists C^{(l)}_\delta >0, \forall T>0, \quad \norm{\psi^\eps_{lj}(u_\eps)(T)}^2_{L^2_xL^2_v} \leq C^{(l)}_\delta \eps^{2\delta}E(\partial^0_\delta h_\eps)^2.$
\end{itemize}
\end{lemma}

\begin{lemma}\label{psiepsR}
For the remainder term we have the three following inequalities
\begin{itemize}
\item $\exists \tilde{C}_4 > 0, \forall T>0, \quad \norm{\int_0^T\psi^\eps_{R}(u_\eps)dt}^2_{L^2_xL^2_v} \leq \tilde{C}_4 T\eps E(h_\eps)^2, $
\item  $\exists \tilde{C}'_4 > 0, \forall T>0, \quad \norm{\psi^\eps_{R}(u_\eps)}^2_{L^2_{[0,T]}L^2_xL^2_v} \leq \tilde{C}'_4 \eps E(h_\eps)^2, $
\item $\exists \tilde{C}''_4 > 0, \forall T>0, \quad \norm{\psi^\eps_{R}(u_\eps)(T)}^2_{L^2_xL^2_v} \leq \tilde{C}''_4 \eps E(h_\eps)^2. $
\end{itemize}
\end{lemma}

\bigskip
Gathering all Lemmas $\ref{psieps0j}$, $\ref{psiepslj}$ and $\ref{psiepsR}$ gives us the strong convergence of $\Psi^{\eps}(u_\eps)-\Psi(u_\eps)$ towards $0$, thanks to Remark $\ref{dtu}$. It remains to prove that we have indeed the expected convergences of $\Psi(u_\eps)$ towards $\Psi(u)$ as $\eps$ tends to $0$.
\par We start this last step by a quick remark relying on Sobolev embeddings and giving us a strong convergence of $h_\eps$ towards $h$ in $L^\infty_{[0,T]}L^\infty_xL^2_v$, for $T>0$.

\begin{remark}\label{Linfinityx}
We know that $h_\eps \to h$ weakly-* in $L^\infty_tH^s_xL^2_v$, for $s\geq s_0 > d/2$. But we also proved that for all $t>0$ that $(h_\eps)_\eps$ is bounded in $H^s_xL^2_v$. Therefore the sequence $(\norm{h_\eps}_{L^2_v},\eps>0)$ is bounded in $H^s_x$ and therefore converges strongly in $H^{s'}_x$ for all $s'<s$.
\\But, by triangular inequality it comes that
$$\abs{\norm{h_\eps}_{H^{s'}_xL^2_v} - \norm{h}_{H^{s'}_xL^2_v}} \leq \norm{\norm{h_\eps}_{L^2_v} - \norm{h}_{L^2_v}}_{H^{s'}_x}.$$

This means that we also have that $\lim\limits_{\eps \to 0} \norm{h_\eps}_{H^{s'}_xL^2_v} = \norm{h}_{H^{s'}_xL^2_v}$. The space $H^{s'}_xL^2_v$ is a Hilbert space and $h_\eps$ tends weakly to $h$ in it, therefore the last result gives us that in fact $h_\eps$ tends strongly to $h$ in $H^{s'}_xL^2_v$.
\par This result is for all $t>0$ and all $s'\leq s$. Furthermore, $s>d/2$ and so we can choose $s'>d/2$. By Sobolev's embedding we obtain that $h_\eps$ tends strongly to $h$ in $L^\infty_xL^2_v$, for all $t>0$. Reminding that $h_\eps \to h$ weakly-* in $L^\infty_tH^s_xL^2_v$ and we obtain that we have
$$\forall T>0,\quad V_T(\eps) = \sup\limits_{t\in [0,T]}\norm{h_\eps - h}_{L^\infty_xL^2_v} \to 0, \: \mbox{as}\:\: \eps \to 0.$$
\end{remark}

\begin{lemma}\label{psiuepsu}
We have the following rate of convergence:
\begin{itemize}
\item $\exists \tilde{C}_5 > 0, \forall T>0, \: \norm{\int_0^T\Psi(u_\eps)dt - \int_0^T\Psi(u)dt }^2_{L^2_xL^2_v} \leq \tilde{C}_5 T^2V_T(\eps)^2, $
\item  $\exists \tilde{C}'_5 > 0, \forall T>0, \: \norm{\Psi(u_\eps)-\Psi(u_\eps)}^2_{L^2_{[0,T]}L^2_xL^2_v} \leq \tilde{C}'_5 TV_T(\eps)^2, $
\item $\exists \tilde{C}''_5 > 0, \forall T>0, \: \norm{\Psi(u_\eps)-\Psi(u_\eps)}^2_{L^2_xL^2_v}(T) \leq \tilde{C}''_5 V_T(\eps)^2. $
\end{itemize}
\end{lemma}

\bigskip
Thus, those Lemmas, combined with the study of the linear case (Lemmas $\ref{Ueps0j}$, $\ref{Uepslj}$ and $\ref{UR}$) prove the Theorem $\ref{hydrolim}$ with the rate of convergence being the maximum of each rate of convergence. Moreover we have proved
$$h(t,x,v) = V(t,x,v)h_{in}(x,v) + \Psi(t,x,v)(\Gamma(h,h)).$$


\subsection{Proof of Theorem $\ref{hydrolim}$}

Thanks to Theorem $\ref{hydrolimT}$ we can control the convergence of $h_\eps$ towards $h$ for any finite time $T$. Then, thanks to the uniqueness property of Theorem $2.1$ and the control on the remainder of Theorem $2.3$ in \cite{Gu4}, in the case of a hard potential collision kernel, one has
$$\forall T >0, V_T(\eps) \leq C_V \eps.$$
Finally, thanks to Theorem $\ref{perturb}$, we have the exponential decay for both $h_\eps$ and $h$, leading to
$$\norm{h_\eps-h}_{H^s_xL^2_v} \leq 2\norm{h_{in}}_{\mathcal{H}^s_\eps} e^{-\tau_s T}.$$

\bigskip
We define
$$T_M = -\frac{1}{\tau_s}\mbox{ln}\left(\frac{\eps}{2\norm{h_{in}}_{\mathcal{H}^s_\eps}}\right)$$
to get that
$$\forall T \geq T_M, \quad \norm{h_\eps-h}_{H^s_xL^2_v} \leq \eps.$$
This conclude the proof Theorem $\ref{hydrolim}$, by applying Theorem $\ref{hydrolimT}$ to $T_M$.


\appendix

\section{Validation of the assumptions}\label{appendix:validation}

As said in the introduction, all the hypocoercivity theory assumptions hold for several different kinetic models. One can find the proof of the assumptions (H1), (H2), (H3), (H1') and (H2') in \cite{MN} directly for the linear relaxation (see also \cite{CCG}), the semi-classical relaxation (see also \cite{NS}), the linear Fokker-Planck equation, the Boltzmann equation with hard potential and angular cutoff and the Landau equation with hard and moderately soft potential (both studied in a constructive way in \cite{Mo1} and \cite{BM}, for the spectral gaps, see also \cite{Gu2} and \cite{Gu3} for the Cauchy problems):

\begin{itemize}
\item The Linear Relaxation
$$\partial_tf + v.\nabla_xf = \frac{1}{\eps}\left[\left(\int_\R^df(t,x,v_*)dv_*\right)\mu(v) - f\right],$$
\item The Semi-classical Relaxation
$$\partial_tf + v.\nabla_xf = \frac{1}{\eps}\int_{\R^d}\left[\mu(1-\delta f)f_* - \mu_* (1-\delta f_*)f\right]dv_*,$$
\item The Linear Fokker-Planck Equation
$$\partial_tf + v.\nabla_xf = \frac{1}{\eps}\nabla_v.\left(\nabla_vf+fv\right),$$
\item The Boltzmann Equation with hard potential and angular cutoff
$$\partial_tf + v.\nabla_xf = \frac{1}{\eps}\int_{\R^d\times \mathbb{S}^{d-1}}b(cos \theta)|v - v_*|^{\gamma}\left[f'f'_* - ff_*\right]dv_*d\sigma,$$
\item The Landau Equation with hard and moderately soft potential
$$\partial_tf + v.\nabla_xf = \frac{1}{\eps}\nabla_v.\left(\int_{\R^d}\Phi(v-v_*)|v-v_*|^{\gamma+2}\left[f_*(\nabla f) - f(\nabla f)_*\right]\right).$$
\end{itemize}

Assumption (H4) is clearly satisfied by the first three as in that case we have either $\norm{.}_{\Lambda_v}=\norm{.}_{L^2_v}$ or $\Gamma = 0$ (see \cite{MN}). Moreover, (H5) is obvious in the case of a linear equation. It thus remains to prove properties (H5) for the semi-classical relaxation and (H4) and (H5) for the Boltzmann equation and the Landau equation (since our property (H4) is slightly different from (H4) in \cite{MN}).


\subsection{The semi-classical relaxation}

In the case of the semi-classical relaxation, the linearization is slightly different. Indeed, the unique global equilibrium associated to an initial data $f_0$ is (assuming some initial bounds, see \cite{MN})
$$f_\infty = \frac{\kappa_\infty \mu}{1+\delta \kappa_\infty \mu},$$
where $\kappa_\infty$ depends on $f_0$.

\bigskip
Thus, we are no longer in the case of a global equilibrium being a Maxwellian. However, a good way of linearizing this equation is (see \cite{MN}) considering
$$f = f_\infty + \eps \:\frac{\sqrt{\kappa_\infty \mu}}{1+\delta\kappa_\infty \mu}\:h.$$
Using such a linearization instead of the one used all along this paper yields the same general equation $\eqref{LinBoltz}$ with $L$ and $\Gamma$ satisfying all the requirements (see \cite{MN}). Indeed, one may find that $\mbox{Ker}(L) = \mbox{Span}\left(f_\infty/\sqrt{\mu}\right)$ and then notice that this is not of the form needed in assumption (H3). However, this is bounded by $e^{-|v|^2/4}$ and therefore we are still able to use the toolbox (section $\ref{sec:toolbox}$, thus all the theorems.

\bigskip
Let us look at the bilinear operator to show that it fulfils hypothesis (H5). A straightforward computation gives us the definition of $\Gamma$,
 
$$\Gamma(g,h) = \frac{\delta\sqrt{\kappa_\infty}}{2} \int_{\R^d}\sqrt{\mu_*}\frac{\mu_*-\mu}{1+\eps\kappa_\infty \mu_*}[hg_*+h_*g]dv_*.$$
Then, multiplying by a function $f$, integrating over $\R^d$ and looking at the change of variable $(v,v_*) \to (v_*,v)$ yields

$$\langle\Gamma(g,h),f\rangle_{L^2_v} = \frac{\delta \sqrt{\kappa_\infty}}{4}\int_{\R^d\times\R^d}(\mu_*-\mu)(gh_*+g_*h)\left[f\frac{\sqrt{\mu_*}}{1+\delta\kappa_\infty \mu_*} - f_*\frac{\sqrt{\mu}}{1+\delta\kappa_\infty \mu}\right]dvdv_*.$$
Therefore, taking $f$ in $\mbox{Ker}(L)$ gives us the expected property.


\subsection{Boltzmann operator with angular cutoff and hard potential}

Notice that, compared to \cite{MN}, we defined $\Gamma$ in a way that it is symmetric which gives us, using the fact that $\mu_*\mu = \mu_*'\mu'$,
 $$\Gamma(g,h) = \frac{1}{2}\int_{\R^d\times \mathbb(S)^{d-1}} B(\mu^{1/2})_*[g_*'h'+ g'h_*' - g_*h - gh_*]dv_*d\sigma,$$
 
\subsubsection{Orthogonality to $\mbox{Ker}(L)$: (H5)} 
 
A well-known property (see \cite{Go} for instance) tells us that for all $\phi$ in $L^2_v$ decreasing fast enough at infinity and for all $\psi$ in $L^2_v$ one has
\begin{eqnarray*}
\int_{\R^d}\Gamma(g,h)(v)\psi(v)dv = \frac{1}{8}&&\int_{(\R^d)^2\times \mathbb{S}^{d-1}} B [g_*'h'+ g'h_*' - g_*h - gh_*] \\&&((\mu^{1/2})_*\psi + (\mu^{1/2})\psi_* - (\mu^{1/2})_*'\psi' - (\mu^{1/2})'\psi'_*)dvdv_*d\sigma.
\end{eqnarray*}

\bigskip
As shown in \cite{MN} or \cite{Ce1} we have that $\mbox{Ker}(L) = \mbox{Span}(1,v_1,\dots,v_d,|v|^2)\mu^{1/2} $ and therefore taking $\psi$ to be each of these kernel functions gives us (H5).

\subsubsection{Controlling derivatives: (H4)} \label{subsubsec:(H4) Boltz}

To prove (H4) we can define
\begin{eqnarray*}
\Gamma^+(g,h) &=& \int_{\R^d\times \mathbb(S)^{d-1}} B(\mu^{1/2})_*g_*'h'\:dv_*d\sigma,
\\ \Gamma^-(g,h) &=& -\int_{\R^d\times \mathbb(S)^{d-1}} B(\mu^{1/2})_*g_*h\:dv_*d\sigma.
\end{eqnarray*}

\bigskip
By using the change of variable $u = v-v_*$ we end up with $\theta$ being a function of $u$ and $\sigma$ and $v'=v+f_1(u,\sigma)$ and $v'_*=v+f_2(u,\sigma)$, $f_1$ and $f_2$ being functions. Therefore we can make this change of variable, take $j$ and $l$ such that $|j|+|l|\leq s$ and differentiate our operator $\Gamma^-$.

$$\partial_l^j\Gamma^-(g,h) =-\frac{1}{2} \sum\limits_{\overset{j_0+j_1+j_2 = j}{l_1+l_2=l}} \int_{\R^d\times\mathbb{S}^{d-1}}b(cos\theta)|u|^\gamma\partial_0^{j_0}\left(\mu(v-u)^{1/2}\right)\partial_{l_1}^{j_1}g_*\:\partial_{l_2}^{j_2}h\:dud\sigma .$$

Then we can easily compute that, $C$ being a generic constant,

$$\abs{\partial_0^{j_0}\left(\mu(v-u)^{1/2}\right)} \leq C \mu(v-u)^{1/4}.$$

Moreover, we are in the case where $\gamma > 0$ and therefore we have 

$$|u|^\gamma \mu(v-u)^{1/4} \leq C (1+|v|)^\gamma \mu(v-u)^{1/8}.$$

Combining this and the fact that $|b|\leq C_b$ (angular cutoff considered here), multiplying by a function $f$ and integrating over $\T^d \times \R^d$ yields, using Cauchy-Schwarz two times,

\begin{eqnarray*}
\abs{\langle\partial_l^j\Gamma^-(g,h),f\rangle_{L^2_{x,v}}} &\leq& C \sum\limits_{\overset{j_0+j_1+j_2 = j}{l_1+l_2=l}} \int_{\T^d\times\R^d}(1+|v|)^\gamma\abs{\partial_{l_2}^{j_2}h} \abs{f} \left(\int_{\R^d}\mu_*^{1/8} \abs{\partial_{l_1}^{j_1}g_*}dv_*\right)dvdx
\\ &\leq& \mathcal{G}^s(g,h)\norm{f}_\Lambda,
\end{eqnarray*}
with
$$\mathcal{G}^s(g,h) = C \sum\limits_{|j_1|+|l_1|+|j_2|+|l_2| \leq s}\left[\int_{\T^d}\norm{\partial_{l_2}^{j_2}h}^2_{\Lambda_v} \norm{\partial_{l_1}^{j_1}g}^2_{L^2_{v}}dx\right]^{1/2}.$$

\bigskip
At that point we can use Sobolev embeddings (see \cite{Br}, corollary $IX.13$) stating that if $E\left(s_0/2\right)>d/2$ then we have $H^{s/2}_x \hookrightarrow L^{\infty}_x$.
\par So, if $|j_1|+|l_1|\leq s/2$ we have

\begin{eqnarray}
\norm{\partial_{l_1}^{j_1}g}^2_{L^2_{v}} &\leq& \sup\limits_{x \in \T^d}\norm{\partial_{l_1}^{j_1}g}^2_{L^2_{v}} \leq C_s\norm{\norm{\partial_{l_1}^{j_1}g}^2_{L^2_{v}}}_{H^{s/2}_x} \nonumber
\\ &\leq& C_s\sum\limits_{|p|\leq s/2}\sum\limits_{p_1 +p_2 =p} \int_{\T^d\times\R^d} \partial_{l_1+p_1}^{j_1}g \:\partial_{l_1+p_2}^{j_1}g \: dvdx \label{jnon0}
\\ &\leq& C_s \norm{g}^2_{H^s_{x,v}}, \nonumber
\end{eqnarray}

\noindent by a mere Cauchy-Schwarz inequality.
\par In the other case, $|j_2|+|l_2|\leq s/2$ and by same calculations we show

$$\norm{\partial_{l_2}^{j_2}h}^2_{\Lambda_v} \leq C_s \norm{h}^2_{H^s_{\Lambda}}.$$

Therefore, by just dividing the sum into this two subcases we obtain the result (H4) for $\Gamma^-$, noticing that in the case $j=0$ equation $\eqref{jnon0}$ has no $v$ derivatives and the Cauchy-Schwarz inequality does not create such derivatives so the control is only made by $x$-derivatives.

\bigskip
The second term $\Gamma^+$ is dealt exactly the same way with, at the end (the study of $\mathcal{G}^s$), another change of variable $(v,v_*) \to (v',v'_*)$ which gives the result since $(1+|v'|)^\gamma \leq (1+|v|)^\gamma + (1+|v_*|)^\gamma$ if $\gamma > 0$.


\subsection{Landau operator with hard and moderately soft potential}
The Landau operator is used to describe plasmas and for instance in the case of particles interacting via a Coulomb interaction (see \cite{Vi2} for more details). The particular case of Coulomb interaction alone ($\gamma=-3$) will not be studied here as the Landau linear operator has a spectral gap if and only if $\gamma \geq -2$ (see \cite{Gu2}, for not constructive arguments, \cite{MS} for general constructive case and \cite{BM} for explicit construction in the case of hard potential $\gamma>0$) and so only the case $\gamma \geq -2$ may be applicable in this study.

\bigskip
We can compute straightforwardly the bilinear symmetric operator associated with the Landau equation:

$$\Gamma(g,h) = \frac{1}{2\sqrt{\mu}}\nabla_v \cdot \int_{\R^d}\sqrt{\mu\mu_*}\Phi(v-v_*)\left[g_*\nabla_vh + h_*\nabla_vg - g(\nabla_vh)_* - h(\nabla_vg)_*\right]dv_*,$$
where $\func{\Phi}{\R^d}{\R^d}$ is such that $\Phi(z)$ is the orthogonal projection onto $\mbox{Span}(z)^\bot$ so
$$\Phi(z)_{ij} = \delta_{ij}-\frac{z_iz_j}{|z|^2},$$
and $\gamma$ belongs to $[-2,1]$.

\subsubsection{Orthogonality to $\mbox{Ker}(L)$: (H5)} 

Let consider a function $\psi$ in $C^{\infty}_{x,v}$. A mere integration by part gives us

$$\langle\Gamma(g,h),\psi\rangle_{L^2_{v}} = -\frac{1}{2}\int_{\R^d\times\R^d}\nabla_v\left(\frac{\psi}{\sqrt{\mu}}\right)\cdot \left(\sqrt{\mu\mu_*}\Phi(v-v_*)[G]\right)dv_*dv,$$
where
$$G = g_*\nabla_vh + h_*\nabla_vg - g(\nabla_vh)_* - h(\nabla_vg)_*.$$

\bigskip
Then the change of variable $(v,v_*)\to (v_*,v)$ only changes $\nabla_v(\psi/\sqrt{\mu})$ to $\left[\nabla_v(\psi/\sqrt{\mu})\right]_*$ and $G$ becomes $-G$. Therefore we finally obtain

$$\langle\Gamma(g,h),\psi\rangle_{L^2_{v}} = \frac{1}{4}\int_{\R^d\times\R^d}\sqrt{\mu\mu_*}\Phi(v-v_*)[G] \cdot \left[\left(\nabla_v\left(\frac{\psi}{\sqrt{\mu}}\right)\right)_*-\nabla_v\left(\frac{\psi}{\sqrt{\mu}}\right)\right]dv_*dv.$$

\bigskip
As shown in \cite{MN} or \cite{Ce1} we have that $\mbox{Ker}(L) = \mbox{Span}(1,v_1,\dots,v_d,|v|^2)\mu^{1/2} $. Computing the term inside brackets for each of these functions gives us $0$ or, in the case $|v|^2\sqrt{\mu}$, $2(v_*-v)$.
\\However, by definition, $\Phi(v-v_*)[G]$ belongs to $\mbox{Span}(v-v_*)^\bot$ and therefore $\Phi(v-v_*)[G]\cdot (v_*-v) =0$. So $\Gamma$ indeed satisfies (H5).

\subsubsection{Controlling derivatives: (H4)}

The article \cite{Gu2} gives us directly the expected result in its Theorem $3$, equation $(35)$ with $\theta=0$. The case where there are only $x$-derivatives is also included if one takes $\beta =0$.

\section{Proofs given the results in the toolbox}\label{appendix:toolbox} 

We used the estimates given by the toolbox throughout this article. This appendix is to prove all of them. It is divided in two parts. The first one is dedicated to the proof of the equality between null spaces whereas the second part deals with the time derivatives inequalities.


\subsection{Proof of Proposition $\ref{kernel}$:}
We are about to prove the following proposition.

\begin{prop}\label{proj}
Let $a$ and $b$ be in $\R^*$ and consider the operator $G = aL - bv \cdot\nabla_x$ acting on $H^1_{x,v}$.
\\ If $L$ satisfies $\emph{(H1)}$ and $\emph{(H3)}$ then $$\emph{\mbox{Ker}}(G) = \emph{\mbox{Ker}}(L).$$
\end{prop}

To prove this result we will need a lemma.

\begin{lemma}\label{lemproj} 
Let $\func{f}{\T^d\times\R^d}{\R}$ be continuous on $\T^d\times\R^d$ and differentiable in $x$.
\\If $v\cdot \nabla_x f(x,v) = 0$ for all $(x,v)$ in $\T^d \times \R^d$ then $f$ does not depend on $x$.
\end{lemma}

\begin{proof}[Proof of Lemma $\ref{lemproj}$]
Fix $x$ in $\T^d$ and $v$ $\Q$-free in $\R^d$.
\\For $y$ in $\R^d$ we will denote by $\bar{y}$ its equivalent class in $\T^d$.

\bigskip
Define $\function{g}{\R}{\R}{t}{f(\bar{x + tv},v)}$.
\\ We find easily that $g$ is differentiable on $\R$ and that $g'(t) = v.\nabla_xf(x,v) = 0$ on $\R$. Therefore:
$$\forall t \in \R, f(\bar{x+tv},v) = f(x,v).$$
However, a well-known property about the torus is that the set $\left\{x + nv, n \in \Z \right\}$ is dense in $\T^d$ for all $x$ in $\T^d$ and $v$ $\Q$-free in $\R^d$. This combined with the last result and the continuity of $f$ leads to:
$$\forall y \in \T^d,\quad f(y,v) = f(x,v).$$
To conclude it is enough to see that the set of $\Q$-free vector in $\R^d$ is dense in $\R^d$ and then, by continuity of $f$ in $v$:
$$\forall y \in \T^d, \forall v \in \R^d \:,\quad f(y,v) = f(x,v).$$
\end{proof}
\bigskip

Now we have all the tools to prove the proposition about the kernel of operators.

\bigskip
\begin{proof}[Proof of Proposition $\ref{proj}$]
Since $L$ satisfies (H1) we know that $L$ acts on $L^2_v$ and that its Kernel functions $\phi_i$ only depend on $v$. Thus, we have directly the first inclusion
$$\mbox{Ker}(L)\subset \mbox{Ker}(G).$$

\bigskip
Then, let us consider $h$ in $H^1_{x,v}$ such that $G(h) = 0$.
\\ Because the transport operator $v \cdot\nabla_x$ is skew-symmetric in $L^2_{x,v}$ we have
$$0 =\langle G(h),h\rangle _{L^2_{x,v}} = a\int_{\T^d}\langle L(h),h\rangle _{L^2_{v}}dx.$$ 
However, because $L$ satisfies (H3) we obtain:
$$0 \geq \lambda \int_{\T^d} \norm{h(x,.) - \pi_L(h(x,.))}^2_{\Lambda_v}dx.$$
But $\lambda$ is strictly positive and thus:
$$\forall x \in \T^d \:,\quad h(x,\cdot) = \pi_L(h(x,\cdot)) = \sum_{i=1}^d c_i(x)\phi_i.$$

\bigskip
Finally we have, by assumption, $G(h) = 0$ and because $h(x,\cdot)$ belongs to $\mbox{Ker}(L)$ for all $x$ in $\T^d$ we end up with
$$\forall (x,v) \in \T^d\times\R^d \:,\quad v\cdot\nabla_x h(x,v) = 0.$$
By applying the lemma above we then obtain that $h$ does not depend on $x$. But $(\phi_i)_{1\leq i \leq d}$ is an orthonormal family, basis of $\mbox{Ker}(L)$, and therefore we find that for all $i$, $c_i$ does not depend on $x$.
\\ So,we have proved that:
$$\forall (x,v) \in \T^d\times\R^d \:,\: h(x,v) =\sum_{i=1}^d c_i\phi_i(v).$$
Therefore, $h$ belongs to $\mbox{Ker}(L)$ and only depends on $x$.
\end{proof}


\subsection{A priori energy estimates}

In this subsection we derive all the inequalities we used. Therefore, we assume that $L$ satisfies (H1'), (H2') and (H3) while $\Gamma$ has the properties (H4) and (H5), and we pick $g$ in $H^s_{x,v}$. We consider $h$ in $H^s_{x,v}\cap \mbox{Ker}(G_\eps)^\bot$ and we assume that $h$ is a solution to $\eqref{LinBoltz}$:
$$\partial_th + \frac{1}{\eps}v.\nabla_xh = \frac{1}{\eps^2}L(h) + \frac{1}{\eps}\Gamma(g,h).$$
In the toolbox, we wrote inequalities on function which were solutions of the linear equation. As the reader may notice, we will deal with the second order operator  just by applying the first part of (H4) and Young's inequality. Such an inequality only provides two positive terms, and thus by just setting $\Gamma$ equal to $0$ in the next inequalities we get the expected bounds in the linear case (not the sharpest ones though). Therefore we will just describe the more general case and the linear one is included in it.


\subsubsection{Time evolution of pure $x$-derivatives}
The operators $L$ and $\Gamma$ only act on the $v$ variable. Thus, for $0\leq|l|\leq s$, $\partial_l^0$ commutes with $L$ and $v\cdot\nabla_x$. Remind that $v\cdot\nabla_x$ is skew-symmetric in $L^2_{x,v}(\T^d \times \R^d)$ and therefore we can compute 
$$\frac{d}{dt}\norm{\partial_l^0h}_{L^2_{x,v}}^2 = \frac{2}{\eps^2}\langle L(\partial_l^0h),\partial_l^0h\rangle _{L^2_{x,v}} + \frac{2}{\eps}\langle \partial_l^0\Gamma(g,h),\partial_l^0h\rangle _{L^2_{x,v}}.$$
We can then use hypothesis (H3) to obtain
$$\frac{2}{\eps^2}\langle L(\partial_l^0h),\partial_l^0h\rangle _{L^2_{x,v}} \leq -\frac{2\lambda}{\eps^2}\norm{(\partial_l^0h)^\bot}^2_\Lambda.$$
We also use (H3) to get $(\partial_l^0h)^\bot = \partial_l^0h^\bot$.
\par To deal with the second scalar product, we will use hypothesis (H4) and (H5), which is still valid for $\partial_l^0\Gamma$ since $\pi_L$ only acts on the $v$ variable, followed by a Young inequality with some $D_1>0$. This yields

\begin{eqnarray*}
\frac{2}{\eps}\langle\partial_l^0\Gamma(g,h),\partial_l^0h\rangle_{L^2_{x,v}} &=& \frac{2}{\eps}\langle\partial_l^0\Gamma(g,h),\partial_l^0h^\bot\rangle_{L^2_{x,v}}
\\ &\leq& \frac{2}{\eps} \mathcal{G}^s_x(g,h)\norm{\partial^0_lh^\bot}_{\Lambda}
\\ &\leq& \frac{D_1}{\eps}\left(\mathcal{G}^s_x(g,h)\right)^2 + \frac{1}{ D_1\eps}\norm{\partial_l^0h^\bot}^2_\Lambda.
\end{eqnarray*}

\noindent Gathering the last two upper bounds we obtain
$$\frac{d}{dt}\norm{\partial_l^0h}_{L^2_{x,v}}^2 \leq \left[\frac{1}{D_1\eps} -\frac{2\lambda}{\eps^2}\right]\norm{\partial_l^0h^\bot}^2_\Lambda + \frac{D_1}{\eps}\left(\mathcal{G}^s_x(g,h)\right)^2 .$$
Finally, taking $D_1 = \eps/\lambda$ gives us inequalities $\eqref{h}$, $\eqref{dx}$ and $\eqref{d0l}$.


\subsubsection{Time evolution of $\norm{\nabla_vh}^2_{L^2_{x,v}}$}
For that term we get, by applying the equation satisfied by $h$, the following:
$$\frac{d}{dt}\norm{\nabla_vh}_{L^2_{x,v}}^2  = \frac{2}{\eps^2}\langle \nabla_vL(h),\nabla_vh\rangle _{L^2_{x,v}} - \frac{2}{\eps}\langle \nabla_v(v\cdot\nabla_xh),\nabla_vh\rangle _{L^2_{x,v}} +\frac{2}{\eps}\langle\nabla_v\Gamma(g,h),\nabla_vh\rangle_{L^2_{x,v}}.$$
And by writing the second term on the right-hand side of the equality and integrating by part in $x$, we have
$$\langle \nabla_v(v\cdot\nabla_xh),\nabla_vh\rangle _{L^2_{x,v}} = \langle \nabla_xh,\nabla_vh\rangle _{L^2_{x,v}}.$$
Therefore the following holds:
$$\frac{d}{dt}\norm{\nabla_vh}_{L^2_{x,v}}^2  = \frac{2}{\eps^2}\langle \nabla_vL(h),\nabla_vh\rangle _{L^2_{x,v}} - \frac{2}{\eps}\langle \nabla_xh,\nabla_vh\rangle _{L^2_{x,v}}+\frac{2}{\eps}\langle\nabla_v\Gamma(g,h),\nabla_vh\rangle_{L^2_{x,v}}.$$
Then we have by (H1) that $L = K - \Lambda$ and we can estimate each component thanks to (H1) and (H2):

\begin{eqnarray*}
- \langle \nabla_v\Lambda(h),\nabla_vh\rangle _{L^2_{x,v}} &\leq& \nu^\Lambda_4 \norm{h}_{L^2_{x,v}}^2 - \nu^\Lambda_3\norm{\nabla_vh}_{\Lambda}^2,
\\ \langle \nabla_vK(h),\nabla_vh\rangle _{L^2_{x,v}} &\leq& C(\delta)\norm{h}^2_{L^2_{x,v}}  + \delta\norm{\nabla_vh}_{L^2_{x,v}}^2,
\end{eqnarray*}

\noindent where $\delta$ is a strictly positive real that we will choose later.
\par Finally, for a $D > 0$ that we will choose later, we have the following upper bound, by Cauchy-Schwarz inequality:
$$-\frac{2}{\eps}\langle \nabla_xh,\nabla_vh\rangle _{L^2_{x,v}} \leq \frac{D}{\eps}\norm{\nabla_xh}_{L^2_{x,v}}^2 + \frac{\nu_1^\Lambda}{D\nu_0^\Lambda\eps}\norm{\nabla_vh}_\Lambda^2,$$
using the fact that $\norm{.}_{L^2_{x,v}}^2\leq \frac{\nu_1^\Lambda}{\nu_0^\Lambda}\norm{.}^2_\Lambda$.
Finally, another Young inequality gives us a control on the last scalar product, for a $D_2>0$ to be chosen later
$$\frac{2}{\eps}\langle\nabla_v\Gamma(g,h),\nabla_vh\rangle_{L^2_{x,v}} \leq \frac{D_2}{\eps} \left(\mathcal{G}^1_{x,v}(g,h)\right)^2 + \frac{1}{D_2\eps}\norm{\nabla_vh}^2_\Lambda.$$
 We gather here the last three inequalities to obtain our global upper bound:

\begin{eqnarray*}
\frac{d}{dt}\norm{\nabla_vh}_{L^2_{x,v}}^2 &\leq& \frac{1}{\eps^2}\left(2\nu_4^\Lambda + 2 C(\delta)\right)\norm{h}^2_{L^2_{x,v}} + \frac{D}{\eps}\norm{\nabla_xh}^2_{L^2_{x,v}} 
\\ &&+ \left(\frac{2\nu_1^\Lambda\delta}{\nu_0^\Lambda\eps^2}-\frac{2\nu_3^\Lambda}{\eps^2} + \frac{\nu_1^\Lambda}{D\eps\nu_0^\Lambda} + \frac{1}{ D_2\eps}\right)\norm{\nabla_vh}^2_{\Lambda} + \frac{D_2}{\eps} \left(\mathcal{G}^1_{x,v}(g,h)\right)^2.
\end{eqnarray*}

We can go even further since we have $\norm{h}^2_{L^2_{x,v}} = \norm{h^\bot}^2_{L^2_{x,v}} + \norm{\pi_L(h)}^2_{L^2_{x,v}}$.
\\But because $h$ is in $\mbox{Ker}(G_\eps)^\bot$ we can use the toolbox and the equation $\eqref{poincare}$ about the Poincar\'e inequality:
$$\norm{\pi_L(h)}^2_{L^2_{x,v}} \leq  C_p\norm{\nabla_xh}^2_{L^2_{x,v}}.$$
This last inequality yields:

\begin{eqnarray*}
\frac{d}{dt}\norm{\nabla_vh}_{L^2_{x,v}}^2 &\leq& \frac{\nu_1^\Lambda}{\nu_0^\Lambda\eps^2}\left(2\nu_4^\Lambda + 2C(\delta)\right)\norm{h^\bot}^2_{\Lambda} + \left[\frac{C_p}{\eps^2}\left(2\nu_4^\Lambda + 2 C(\delta)\right)+\frac{D}{\eps}\right]\norm{\nabla_xh}^2_{L^2_{x,v}}
\\ && + \left[\frac{2\nu_1^\Lambda\delta}{\nu_0^\Lambda\eps^2}-\frac{2\nu_3^\Lambda}{\eps^2} + \frac{\nu_1^\Lambda}{D\eps\nu_0^\Lambda} + \frac{1}{D_2\eps}\right]\norm{\nabla_vh}^2_{\Lambda} + \frac{D_2}{\eps}\left(\mathcal{G}^1_{x,v}(g,h)\right)^2.
\end{eqnarray*}

Therefore, we can choose $\delta = \nu_0^\Lambda\nu_3^\Lambda/6\nu_1^\Lambda$, $D = 3\nu_1^\Lambda\eps/\nu_0^\Lambda\nu_3^\Lambda$ and $D_2 = 3\eps/\nu_3^\Lambda$ to get the equation $\eqref{dv}$.


\subsubsection{Time evolution of $\langle \nabla_x h,\nabla_v h\rangle _{L^2_{x,v}}$}
In the same way, and integrating by part in $x$ then in $v$ we obtain the following equality:
$$\frac{d}{dt}\langle \nabla_xh,\nabla_vh\rangle _{L^2_{x,v}} = \frac{2}{\eps^2}\langle L(\nabla_xh),\nabla_vh\rangle _{L^2_{x,v}} - \frac{2}{\eps}\langle \nabla_v(v\cdot\nabla_xh),\nabla_xh\rangle _{L^2_{x,v}}+ \frac{2}{\eps} \langle\nabla_x\Gamma(g,h),\nabla_vh\rangle_{L^2_{x,v}}. $$

By writing explicitly $\langle \nabla_v(v\cdot\nabla_xh),\nabla_xh\rangle _{L^2_{x,v}} $ and by integrating by part one can show that the following holds:
$$\langle \nabla_v(v.\nabla_xh),\nabla_xh\rangle _{L^2_{x,v}} = \frac{1}{2}\norm{\nabla_xh}^2_{L^2_{x,v}}.$$

Therefore we have an explicit formula for that term and we can find the time derivative of the scalar product being:

$$\frac{d}{dt}\langle \nabla_xh,\nabla_vh\rangle _{L^2_{x,v}} = \frac{2}{\eps^2}\langle L(\nabla_xh),\nabla_vh\rangle _{L^2_{x,v}} - \frac{1}{\eps}\norm{\nabla_xh}^2_{L^2_{x,v}}+ \frac{2}{\eps} \langle\nabla_x\Gamma(g,h),\nabla_vh\rangle_{L^2_{x,v}}. $$

We can bound above the first term in the right-hand side of the equality thanks to (H1) and then Cauchy-Schwarz in $x$, with a constant $\eta>0$ to be define later.

\begin{eqnarray*}
\frac{2}{\eps^2}\langle L(\nabla_xh),\nabla_vh\rangle _{L^2_{x,v}} &=& \frac{2}{\eps^2}\langle L(\nabla_xh^\bot),\nabla_vh\rangle _{L^2_{x,v}}
\\ &\leq& \frac{C^L}{\eps^2} \int_{\T^d} 2\norm{\nabla_xh^\bot}_{\Lambda_v}\norm{\nabla_vh}_{\Lambda_v}dx
\\ &\leq& \frac{C^L\eta}{\eps^2} \norm{\nabla_xh^\bot}_{\Lambda}^2 + \frac{C^L}{\eta\eps^2}\norm{\nabla_vh}_{\Lambda}^2.
\end{eqnarray*}

Then applying hypothesis (H4) and Young's inequality one more time with a constant $D_3>0$ one may find
$$\frac{2}{\eps}\langle\nabla_x\Gamma(g,h),\nabla_vh\rangle_{L^2_{x,v}} \leq    \frac{D_3}{\eps}\left(\mathcal{G}^1_x(g,h)\right)^2 + \frac{1}{D_3\eps}\norm{\nabla_vh}^2_\Lambda.$$

Hence we end up with the following inequality:
\begin{eqnarray*}
\frac{d}{dt}\langle \nabla_xh,\nabla_vh\rangle _{L^2_{x,v}} \leq &&\frac{C^L\eta}{\eps^2}\norm{\nabla_xh^\bot}_{\Lambda}^2 + \left(\frac{C^L}{\eta\eps^2}+ \frac{1}{ D_3}\right)\norm{\nabla_vh}_{\Lambda}^2 - \frac{1}{\eps}\norm{\nabla_xh}^2_{L^2_{x,v}} 
\\ &&+\frac{D_3}{\eps}\left(\mathcal{G}^1_x(g,h)\right)^2 .
\end{eqnarray*}
Now define $\eta = e/\eps$, $e>0$, and $D_3 = e /C^L$ to obtain equation $\eqref{dx,dv}$.


\subsubsection{Time evolution of $\norm{\partial^j_lh}^2_{L^2_{x,v}}$ for $|j|\geq 1$ and $|j|+|l|=s$}
This term is the only term far from what we already did since we are mixing more than one derivative in $x$ and one derivative in $v$ in general. By simply differentiating in time and integrating by part we find the following equality.

\begin{eqnarray*}
\frac{d}{dt}\norm{\partial^j_lh}^2_{L^2_{x,v}} &=& \frac{2}{\eps^2}\langle \partial^j_lL(h),\partial^j_lh\rangle _{L^2_{x,v}} - \frac{2}{\eps}\langle \partial^j_l(v.\nabla_xh),\partial^j_lh\rangle _{L^2_{x,v}}
\\ && + \frac{2}{\eps}\langle \partial_l^j\Gamma(g,h),\partial_l^jh\rangle_{L^2_{x,v}}
\\                                             &=& \frac{2}{\eps^2}\langle \partial^j_lL(h),\partial^j_lh\rangle _{L^2_{x,v}} - \frac{2}{\eps}\sum\limits_{i,c_i(j)> 0}\langle \partial^j_lh,\partial^{j-\delta_i}_{l+\delta_i}h\rangle _{L^2_{x,v}}
\\ &&+ \frac{2}{\eps}\langle \partial_l^j\Gamma(g,h),\partial_l^jh\rangle_{L^2_{x,v}}.
\end{eqnarray*}

We can then apply Cauchy-Schwarz for the terms inside the sum symbol. For each we can use a $D_{i,l,s}> 0$ but because they play an equivalent role we will take the same $D > 0$, that we will choose later:
$$-\frac{2}{\eps}\langle \partial^j_lh,\partial^{j-\delta_i}_{l+\delta_i}h\rangle _{L^2_{x,v}} \leq \frac{\nu_1^\Lambda}{D\nu_0^\Lambda\eps}\norm{\partial^j_lh}^2_\Lambda +\frac{D}{\eps}\norm{\partial^{j-\delta_i}_{l+\delta_i}h}^2_{L^2_{x,v}}.$$

Then we can use (H1') and (H2'), with a $\delta>0$ we will choose later, to obtain
$$\frac{2}{\eps^2}\langle \partial^j_lL(h),\partial^j_lh\rangle _{L^2_{x,v}}\leq \frac{2}{\eps^2}(C(\delta)+\nu_6^\Lambda)\norm{h}^2_{H^{s-1}_{x,v}} + \frac{2}{\eps^2}\left(\frac{\delta\nu_1^\Lambda}{\nu_0^\Lambda}-\nu_5^\Lambda\right)\norm{\partial^j_lh}^2_\Lambda.$$

Finally, applying (H4) and Young's inequality with a constant $D_2>0$ we obtain

$$\frac{2}{\eps}\langle\partial_l^j\Gamma(g,h),\partial_l^jh\rangle_{L^2_{x,v}} \leq \frac{D_2}{\eps} \left(\mathcal{G}^s_{x,v}(g,h)\right)^2 + \frac{1}{D_2\eps}\norm{\partial_l^jh}^2_\Lambda.$$

\bigskip
Combining these three inequality we find an upper bound for the time evolution. Here we also use the fact that the number of $i$ such that $c_i(j)> 0$ is less or equal to $d$.

\begin{eqnarray*}
\frac{d}{dt}\norm{\partial^j_lh}^2_{L^2_{x,v}} &\leq& \left[\frac{\nu_1^\Lambda  d}{D\eps\nu_0^\Lambda} +\frac{2}{\eps^2} \left(\frac{\delta\nu_1^\Lambda}{\nu_0^\Lambda}-\nu_5^\Lambda\right) + \frac{1}{D_2 \eps}\right]\norm{\partial^j_lh}^2_\Lambda 
\\                                            && + \frac{D}{\eps}\sum\limits_{i,c_i(j)> 0}\norm{\partial^{j-\delta_i}_{l+\delta_i}h}^2_{L^2_{x,v}}+ \frac{2}{\eps^2}(C(\delta) + \nu_6^\Lambda)\norm{h}^2_{H^{s-1}_{x,v}} 
\\ &&+ \frac{D_2}{\eps}\left(\mathcal{G}^s_{x,v}(g,h)\right)^2.
\end{eqnarray*}
Hence, we obtain equations $\eqref{djl}$ and $\eqref{ddeltai}$ by taking $D = 3\nu_1^\Lambda\eps/\nu_0^\Lambda\nu_5^\Lambda$, $ D_2 = 3\eps/\nu_5^\Lambda$ and $\delta = \nu_0^\Lambda\nu_5^\Lambda/6\nu_1^\Lambda$. Also note that in $\eqref{djl}$ we used $\norm{\partial^{j-\delta_i}_{l+\delta_i}h}^2_{L^2_{x,v}} \leq \frac{\nu_1^\Lambda}{\nu_0^\Lambda}\norm{\partial^{j-\delta_i}_{l+\delta_i}h}^2_\Lambda$.


\subsubsection{Time evolution of $\langle \partial^{\delta_i}_{l-\delta_i}h,\partial^0_lh\rangle _{L^2_{x,v}}$}
With no more calculations, we can bound this term in the same way we did for $\frac{d}{dt}\langle \nabla_x h,\nabla_v h\rangle _{L^2_{x,v}}$. Here we get

\begin{eqnarray*}
\frac{d}{dt}\langle  \partial^{\delta_i}_{l-\delta_i}h, \partial^0_lh\rangle _{L^2_{x,v}} &\leq& \frac{C^L\eta}{\eps^2} \norm{ \partial^0_lh^\bot}_{\Lambda}^2 + \left[\frac{C^L}{\eta\eps^2}+ \frac{1}{\eps D_3}\right]\norm{ \partial^{\delta_i}_{l-\delta_i}h}_{\Lambda}^2 -\frac{1}{\eps}\norm{ \partial^0_lh}^2_{L^2_{x,v}} 
\\ && + \frac{D_3}{\eps}\left(\mathcal{G}^s_{x}(g,h)\right)^2.
\end{eqnarray*}

Now define $\eta = e/\eps$, $e>0$, and $D_3 = e/C^L$ to obtain equation $\eqref{deltai,d0l}$.

\bigskip
In the next paragraphs, we are setting $g=h$.

\subsubsection{Time evolution of $\norm{\nabla_vh^\bot}^2_{L^2_{x,v}}$}
By simply differentiating norm and using (H5) to get $\Gamma(h,h)^\bot = \Gamma(h,h)$, we compute
$$\frac{d}{dt}\norm{\nabla_vh^\bot}_{L^2_{x,v}}^2 = 2\langle \nabla_v(G_\eps(h))^\bot,\nabla_vh^\bot\rangle _{L^2_{x,v}} + \frac{2}{\eps}\langle\nabla_v\Gamma(h,h),\nabla_vh^\bot\rangle_{L^2_{x,v}}.$$

By applying (H4) and Young's inequality to the second term on the right-hand side, with a constant $D_2 >0$, and controlling the $L^2_{x,v}$-norm by the $\Lambda$-norm we obtain:
$$\frac{2}{\eps}\langle\nabla_v\Gamma(h,h),\nabla_vh^\bot\rangle_{L^2_{x,v}} \leq \frac{D_2}{\eps} \left(\mathcal{G}^1_{x,v}(h,h)\right)^2 + \frac{1}{\eps D_2}\norm{\nabla_vh^\bot}^2_\Lambda.$$

Then we have to control the first term. Just by writing it and decomposing terms in projection onto $\mbox{Ker}(L)$ and onto its orthogonal we yield:

\begin{eqnarray*}
2\langle \nabla_v(G_\eps(h))^\bot,\nabla_vh^\bot\rangle _{L^2_{x,v}} &=& \frac{2}{\eps^2}\langle\nabla_v L(h),\nabla_vh^\bot\rangle_{L^2_{x,v}} - \frac{2}{\eps}\langle\nabla_v(v\cdot \nabla_xh)^\bot,\nabla_vh^\bot \rangle_{L^2_{x,v}}
\\ &=& \frac{2}{\eps^2}\langle\nabla_v L(h^\bot),\nabla_vh^\bot\rangle_{L^2_{x,v}} - \frac{2}{\eps}\langle\nabla_xh,\nabla_vh^\bot \rangle_{L^2_{x,v}} 
\\ & &- \frac{2}{\eps}\langle v\cdot\nabla_v\nabla_x\pi_L(h),\nabla_vh^\bot \rangle_{L^2_{x,v}} 
\\ && + \frac{2}{\eps}\langle\nabla_v\pi_L(v\cdot\nabla_xh),\nabla_vh^\bot \rangle_{L^2_{x,v}}.
\end{eqnarray*}

\bigskip
Then we can control the first term on the right-hand side thanks to (H1) and (H2), $\delta>0$ to be chosen later:
$$\frac{2}{\eps^2}\langle\nabla_v L(h^\bot),\nabla_vh^\bot\rangle_{L^2_{x,v}}\leq \frac{2(C(\delta)+\nu_4^\Lambda)\nu^1_\Lambda}{\nu_0^\Lambda \eps^2}\norm{h^\bot}^2_\Lambda + \frac{2}{\eps^2}\left(\frac{\nu_1^\Lambda \delta}{\nu_0^\Lambda} - \nu_3^\Lambda\right)\norm{\nabla_vh^\bot}^2_\Lambda.$$

We apply Cauchy-Schwarz inequality to the next term, with $D$ to be chosen later:
$$- \frac{2}{\eps}\langle\nabla_xh,\nabla_vh^\bot \rangle_{L^2_{x,v}} \leq \frac{D}{\eps}\norm{\nabla_xh}^2_{L^2_{x,v}} + \frac{\nu_1^\Lambda}{\nu_0^\Lambda D \eps} \norm{\nabla_vh^\bot}^2_\Lambda.$$

For the third term we are going to apply Cauchy-Schwarz inequality and then use the property (H3). The latter property tells us that the functions in $\mbox{Ker}(L)$ are of the form a polynomial in $v$ times $e^{-|v|^2/4}$. This fact combined with the shape of $\pi_L$, equation $\eqref{piL}$, shows us that we can control, by a mere Cauchy-Schwarz inequality, the third term. Then the property $\eqref{dvcontrolL}$ yields the following upper bound:

\begin{eqnarray*}
- \frac{2}{\eps}\langle v\cdot\nabla_v\nabla_x\pi_L(h),\nabla_vh^\bot \rangle_{L^2_{x,v}} &\leq& \frac{\tilde{D}}{\eps}\norm{v\cdot\nabla_v\pi_L(\nabla_xh)}^2_{L^2_{x,v}} + \frac{1}{\tilde{D} \eps}\norm{\nabla_vh^\bot}^2_{L^2_{x,v}}
\\&\leq& \frac{\tilde{D}C_{\pi 1}}{\eps}\norm{\nabla_xh}^2_{L^2_{x,v}} + \frac{\nu_1^\Lambda}{\nu_0^\Lambda \tilde{D} \eps}\norm{\nabla_vh^\bot}^2_\Lambda.
\end{eqnarray*}

Finally, we first use equation $\eqref{dvcontrolL}$ controling the $v$-derivatives of $\pi_L$ and then see that the norm of $\pi_L(v.f)$ is easily controled by the norm of $f$ (just use (H3) and the definition of $\pi_L$ $\eqref{piL}$ and apply Cauchy-Schwarz inequality) by a factor $C_{\pi 1}$ (increase this constant if necessary in $\eqref{dvcontrolL}$):

\begin{eqnarray*}
\frac{2}{\eps}\langle\nabla_v\pi_L(v.\nabla_xh),\nabla_vh^\bot\rangle_{L^2_{x,v}}&\leq& \frac{D'}{\eps}\norm{\nabla_v\pi_L(v.\nabla_xh)}^2_{L^2_{x,v}} + \frac{1}{\eps D'}\norm{\nabla_vh^\bot}^2_{L^2_{x,v}}
\\ &\leq& \frac{D'C_{\pi 1}}{\eps}\norm{\pi_L(v.\nabla_xh)}^2_{L^2_{x,v}} + \frac{1}{\eps D'}\norm{\nabla_vh^\bot}^2_{L^2_{x,v}}
\\ &\leq& \frac{D'C_{\pi 1}^2}{\eps}\norm{\nabla_xh}^2_{L^2_{x,v}} + \frac{\nu_1^\Lambda}{\nu_0^\Lambda \eps D'}\norm{\nabla_vh^\bot}^2_{L^2_{x,v}}.
\end{eqnarray*}

We then gather all those bounds to get the last upper bound for the time derivative of the $v$-derivative.

\begin{eqnarray*}
\frac{d}{dt}\norm{\nabla_vh^\bot}_{L^2_{x,v}}^2 &\leq& \frac{\nu_1^\Lambda}{\nu_0^\Lambda\eps^2}\left(2\nu_4^\Lambda + 2C(\delta)\right)\norm{h^\bot}^2_{\Lambda} + \left[\frac{D}{\eps}+\frac{D'C_{\pi 1}^2}{\eps}+\frac{\tilde{D}C_{\pi 1}}{\eps}\right]\norm{\nabla_xh}^2_{L^2_{x,v}}
\\ && + \left[\frac{2\nu_1^\Lambda\delta}{\nu_0^\Lambda\eps^2}-\frac{2\nu_3^\Lambda}{\eps^2} + \frac{\nu_1^\Lambda}{\eps\nu_0^\Lambda}\left(\frac{1}{D}+\frac{1}{D'}+\frac{1}{\tilde{D}}\right)+ \frac{1}{\eps D_2}\right]\norm{\nabla_vh^\bot}^2_{\Lambda}
\\ &&+ \frac{D_2}{\eps}\left(\mathcal{G}^1_{x,v}(h,h)\right)^2.
\end{eqnarray*}
 Therefore we obtain $\eqref{dvbot}$ by taking $D=D'=\tilde{D}=9\nu_1^\Lambda\eps/\nu_0^\Lambda\nu_3^\Lambda$, $\delta = \nu_0^\Lambda\nu_3^\Lambda/6\nu_1^\Lambda$ and $D_2 = 3\eps/\nu_3^\Lambda$.


\subsubsection{A new time evolution of $\langle\nabla_xh,\nabla_vh\rangle_{L^2_{x,v}}$}
By integrating by part in $x$ then in $v$ we obtain the following equality on the evolution of the scalar product:
$$\frac{d}{dt}\langle\nabla_xh,\nabla_vh\rangle_{L^2_{x,v}} = 2 \langle\nabla_xG_\eps(h),\nabla_vh\rangle_{L^2_{x,v}} + \frac{2}{\eps} \langle\nabla_v\Gamma(h,h),\nabla_xh\rangle_{L^2_{x,v}}.$$

We will bound above the first term as in the previous case and for the second term involving $\Gamma$ we use (H4) and Young's inequality with a constant $D_3 > 0$:
$$2\langle\nabla_v\Gamma(h,h),\nabla_xh\rangle_{L^2_{x,v}}\leq D_3\left(\mathcal{G}^1_{x,v}(h,h)\right)^2 + \frac{1}{ D_3}\norm{\nabla_xh}^2_{\Lambda}.$$

\noindent We decompose $\nabla_x h$ thanks to $\pi_L$ and we use $\eqref{L2L2lambdafluid}$ to control the fluid part of it,
$$2\langle\nabla_v\Gamma(h,h),\nabla_xh\rangle_{L^2_{x,v}}\leq D_3\left(\mathcal{G}^1_{x,v}(h,h)\right)^2 + \frac{1}{ D_3}\norm{\nabla_xh^\bot}^2_{\Lambda} + \frac{C_\pi}{ D_3}\norm{\nabla_xh}^2_{L^2_{x,v}}.$$

Finally we obtain an upper bound for the time-derivative:

\begin{eqnarray*}
\frac{d}{dt}\langle \nabla_xh,\nabla_vh\rangle _{L^2_{x,v}} &\leq& \left[\frac{C^L\eta}{\eps^2}+ \frac{1}{\eps D_3}\right] \norm{\nabla_xh^\bot}_{\Lambda}^2 + \frac{C^L}{\eta\eps^2}\norm{\nabla_vh}_{\Lambda}^2 +\left[\frac{C_\pi}{\eps D_3}-\frac{1}{\eps}\right]\norm{\nabla_xh}^2_{L^2_{x,v}} 
\\ && + \frac{D_3}{\eps}\left(\mathcal{G}^1_{x,v}(h,h)\right)^2.
\end{eqnarray*}

But now, we can use the properties $\eqref{dvcontrolL}$ and $\eqref{L2L2lambdafluid}$ of the projection $\pi_L$ to go further.

\begin{eqnarray*}
\norm{\nabla_vh}^2_\Lambda &\leq& 2\norm{\nabla_vh^\bot}^2_\Lambda + 2\norm{\nabla_v\pi_L(h)}^2_\Lambda
\\ &\leq& 2\norm{\nabla_vh^\bot}^2_\Lambda + 2 C_{\pi 1}C_\pi \norm{\pi_L(h)}^2_{L^2_{x,v}}
\\ &\leq& 2\norm{\nabla_vh^\bot}^2_\Lambda + 2C_{\pi 1}C_\pi C_p\norm{\nabla_xh}^2_{L^2_{x,v}},
\end{eqnarray*}
where we used Poincare inequality $\eqref{poincare}$ because $h$ is in $\mbox{Ker}(G_\eps)^\bot$.
\par Hence we have a final upper bound for the time derivative:

\begin{eqnarray*}
\frac{d}{dt}\langle \nabla_xh,\nabla_vh\rangle _{L^2_{x,v}} &\leq& \left[\frac{C^L\eta}{\eps^2}+\frac{1}{\eps D_3}\right] \norm{\nabla_xh^\bot}_{\Lambda}^2
\\ && + \frac{2C^L}{\eta\eps^2}\norm{\nabla_vh^\bot}_{\Lambda}^2 +\left[\frac{2C^LC_{\pi 1}C_\pi C_p}{\eps^2 \eta}+\frac{C_\pi}{\eps D_3}-\frac{1}{\eps}\right]\norm{\nabla_xh}^2_{L^2_{x,v}} 
\\ && + \frac{D_3}{\eps}\left(\mathcal{G}^1_{x,v}(h,h)\right)^2.
\end{eqnarray*}

Thus, setting $\eta = 8eC^LC_{\pi 1}C_\pi C_p /\eps$ with $e\geq 1$ and $D_3 = 4C_\pi$ we obtain equation $\eqref{dx,dvbot}$.


\subsubsection{Time evolution of $\norm{\partial_l^jh^\bot}^2_{L^2_{x,v}}$, $j\geq 1$ and $|j|+|l|=s$}
We have the following time evolution:
$$\frac{d}{dt}\norm{\partial_l^jh^\bot}^2_{L^2_{x,v}} = 2 \langle \partial_l^j(G_\eps(h))^\bot,\partial_l^jh^\bot\rangle_{L^2_{x,v}} + \frac{2}{\eps}\langle \partial_l^j\Gamma(h,h),\partial_l^jh^\bot\rangle_{L^2_{x,v}}.$$

As above, we apply (H4) for the last term on the right hand side, with a constant $D_2 > 0$,
$$ 2\langle\partial_l^j\Gamma(h,h),\partial_l^jh^\bot\rangle_{L^2_{x,v}} \leq D_2 \left(\mathcal{G}^s_{x,v}(h,h)\right)^2 + \frac{1}{ D_2}\norm{\partial_l^jh^\bot}^2_\Lambda.$$

Then we evaluate the first term on the right-hand side.

\begin{eqnarray*}
2 \langle \partial_l^j(G_\eps(h))^\bot,\partial_l^jh^\bot\rangle_{L^2_{x,v}} &=& \frac{2}{\eps^2}\langle\partial^j_lL(h),\partial_l^jh^\bot\rangle_{L^2_{x,v}} - \frac{2}{\eps}\langle\partial^j_l(v.\nabla_xh)^\bot,\partial_l^jh^\bot\rangle_{L^2_{x,v}}
\\ &=& \frac{2}{\eps^2}\langle\partial^j_lL(h^\bot),\partial_l^jh^\bot\rangle_{L^2_{x,v}} - \frac{2}{\eps}\langle v\cdot\partial^j_l\pi_L(\nabla_xh),\partial_l^jh^\bot\rangle_{L^2_{x,v}}
\\ && - \frac{2}{\eps}\sum\limits_{i,c_i(j)> 0}\langle\partial^{j-\delta_i}_{l+\delta_i}h,\partial^j_lh^\bot\rangle_{L^2_{x,v}} 
\\ && + \frac{2}{\eps}\langle\partial^j_l\pi_L(v\cdot\nabla_xh),\partial_l^jh^\bot\rangle_{L^2_{x,v}}.
\end{eqnarray*}

Then we shall bound each of these four terms on the right-hand side.
\\ We can first use the properties (H1') and (H2') of $L$ to get, for some $\delta$ to be chosen later,
$$\frac{2}{\eps^2}\langle\partial^j_lL(h^\bot),\partial_l^jh^\bot\rangle_{L^2_{x,v}} \leq \frac{2}{\eps^2} \left(C(\delta)+\nu_6^\Lambda\right)\norm{h^\bot}^2_{H^{s-1}_{x,v}} + \frac{2}{\eps^2}\left(\frac{\nu_1^\Lambda\delta}{\nu_0^\Lambda}-\nu_5^\Lambda\right)\norm{\partial_l^jh^\bot}^2_\Lambda.$$

For the three remaining terms we will apply Cauchy-Schwarz inequality and use the properties of $\pi_L$ concerning $v$-derivatives and multiplications by a polynomial in $v$.
\\First

\begin{eqnarray*}
- \frac{2}{\eps}\langle v\cdot\partial^j_l\pi_L(\nabla_xh),\partial_l^jh^\bot\rangle_{L^2_{x,v}} &\leq& \frac{D}{\eps}\norm{v\cdot\partial^j_l\pi_L(\nabla_xh)}^2_{L^2_{x,v}} + \frac{1}{D\eps}\norm{\partial_l^jh^\bot}^2_{L^2_{x,v}}
\\ &\leq& \frac{DC_{\pi s}}{\eps}\norm{\partial^0_l(\nabla_xh)}^2_{L^2_{x,v}} + \frac{\nu_1^\Lambda}{\nu_0^\Lambda D\eps}\norm{\partial_l^jh^\bot}^2_\Lambda
\\ &\leq& \left\{\begin{array}{c}\displaystyle{\frac{DC_{\pi s}}{\eps}\sum\limits_{|l'| = s} \norm{\partial^0_{l'}h}^2_{L^2_{x,v}} + \frac{\nu_1^\Lambda}{\nu_0^\Lambda D\eps}\norm{\partial_l^jh^\bot}^2_\Lambda}, \:\: \mbox{if}\: |j|=1 \vspace{2mm}
\\\vspace{2mm}\displaystyle{\frac{DC_{\pi s}}{\eps}\sum\limits_{|l'| \leq s-1} \norm{\partial^0_{l'}h}^2_{L^2_{x,v}} + \frac{\nu_1^\Lambda}{\nu_0^\Lambda D\eps}\norm{\partial_l^jh^\bot}^2_\Lambda}, \:\: \mbox{if}\: |j|>1, \end{array}\right.
\end{eqnarray*}

where we used that $|l| = |s| - |j|$. Then
$$- \frac{2}{\eps}\langle\partial^{j-\delta_i}_{l+\delta_i}h,\partial^j_lh^\bot\rangle_{L^2_{x,v}} \leq \frac{D'}{ \eps}\norm{\partial^{j-\delta_i}_{l+\delta_i}h}^2_{L^2_{x,v}}+ \frac{\nu_1^\Lambda}{\nu_0^\Lambda D'\eps}\norm{\partial^j_lh^\bot}^2_\Lambda$$

In the case where $|j|>1$ we can also use that $\norm{\partial^{j-\delta_i}_{l+\delta_i}h}^2_{L^2_{x,v}}$ can be decomposed thanks to $\pi_L$ and its orthogonal projector. Then the fluid part is controlled by the $x$-derivatives only.
\\And finally

\begin{eqnarray*}
\frac{2}{\eps}\langle\partial^j_l\pi_L(v\cdot\nabla_xh),\partial_l^jh^\bot\rangle_{L^2_{x,v}} &\leq& \frac{\tilde{D}}{\eps}\norm{\partial^j_l\pi_L(v\cdot\nabla_xh)}^2_{L^2_{x,v}} + \frac{1}{\tilde{D}\eps}\norm{\partial_l^jh^\bot}^2_{L^2_{x,v}}
\\&\leq& \frac{\tilde{D}C_{\pi s}}{\eps}\norm{\partial^0_l\nabla_xh}^2_{L^2_{x,v}} + \frac{\nu_1^\Lambda}{\tilde{D}\nu_0^\Lambda\eps}\norm{\partial_l^jh^\bot}^2_\Lambda
\\ &\leq& \left\{\begin{array}{c}\displaystyle{\frac{\tilde{D}C_{\pi s}}{\eps}\sum\limits_{|l'| = s} \norm{\partial^0_{l'}h}^2_{L^2_{x,v}} + \frac{\nu_1^\Lambda}{\nu_0^\Lambda \tilde{D}\eps}\norm{\partial_l^jh^\bot}^2_\Lambda}, \:\: \mbox{if}\: |j|=1 \vspace{2mm}
\\\vspace{2mm}\displaystyle{\frac{\tilde{D}C_{\pi s}}{\eps}\sum\limits_{|l'| \leq s-1} \norm{\partial^0_{l'}h}^2_{L^2_{x,v}} + \frac{\nu_1^\Lambda}{\nu_0^\Lambda \tilde{D}\eps}\norm{\partial_l^jh^\bot}^2_\Lambda}, \:\: \mbox{if}\: |j|>1, \end{array}\right.
\end{eqnarray*}

We are now able to combine all those estimates to get an upper bound of the time-derivative we are looking at. We can also give to different bounds, depending on the size $|j|$. We also used that the number of $i$ such that $c_i(j)>0$ is less than $d$.

\bigskip
In the case $\abs{j}>1$,
\begin{eqnarray*}
 \frac{d}{dt}\norm{\partial_l^jh^\bot}^2_{L^2_{x,v}} &\leq&\left[\frac{2}{\eps^2}\left(\frac{\nu_1^\Lambda\delta}{\nu_0^\Lambda} - \nu_5^\Lambda\right) + \frac{\nu_1^\Lambda}{\nu_0^\Lambda\eps}\left(\frac{1}{D}+\frac{d}{D'}+ \frac{1}{\tilde{D}}\right)+ \frac{1}{D_2}\right]\norm{\partial^j_lh^\bot}^2_\Lambda
\\ && + \frac{D' \nu_1^\Lambda}{2\nu_0^\Lambda\eps}\sum\limits_{i,c_i(j)> 0}\norm{\partial^{j-\delta_i}_{l+\delta_i}h^\bot}^2_\Lambda
\\ && + \left[\frac{D C_{\pi s}}{2\eps}+ \frac{D'C_{\pi s}}{\eps}+ \frac{\tilde{D}C_{\pi s}}{\eps}\right]\sum\limits_{|l'| \leq s-1} \norm{\partial^0_{l'}h}^2_{L^2_{x,v}}
\\ && + \frac{2(C(\delta)+\nu_6^\Lambda)}{\eps^2} \norm{h^\bot}^2_{H^{s-1}_{x,v}}
\\ && + \frac{D_2}{\eps}\left(\mathcal{G}^s_{x,v}(h,h)\right)^2.
\end{eqnarray*}

And in the case $|j|=1$,

\begin{eqnarray*}
\frac{d}{dt}\norm{\partial_{l-\delta_i}^{\delta_i}h^\bot}^2_{L^2_{x,v}} &\leq&\left[\frac{2}{\eps^2}\left(\frac{\nu_1^\Lambda\delta}{\nu_0^\Lambda} - \nu_5^\Lambda\right) + \frac{\nu_1^\Lambda}{\nu_0^\Lambda\eps}\left(\frac{1}{D}+\frac{1}{D'}+ \frac{1}{\tilde{D}}\right)+\frac{1}{D_2}\right]\norm{\partial^{\delta_i}_{l-\delta_i}h^\bot}^2_\Lambda
\\ && + \left[\frac{DC_{\pi s}}{\eps}+\frac{D'}{\eps}+\frac{\tilde{D}C_{\pi s}}{\eps}\right]\sum\limits_{|l'|=s}\norm{\partial^0_{l'}h}^2_{L^2_{x,v}}
\\ && + \frac{2(C(\delta)+\nu_6^\Lambda)}{\eps^2}\norm{h^\bot}^2_{H^{s-1}_{x,v}}
\\ && + \frac{D_2}{\eps}\left(\mathcal{G}^s_{x,v}(h,h)\right)^2.
\end{eqnarray*}

By taking $D=\tilde{D} = 9\nu_1^\Lambda \eps/\nu_0^\Lambda\nu_5^\Lambda$, $D_2 =3\eps/\nu_5^\Lambda$, $\delta = \nu_0^\Lambda\nu_5^\Lambda/6\nu_1^\Lambda$ and $D'=9\nu_1^\Lambda \eps/\nu_0^\Lambda\nu_5^\Lambda$, if $|j|=1$, or $D'=9\nu_1^\Lambda d \eps/\nu_0^\Lambda\nu_5^\Lambda$, if $|j|>1$, we obtain $\eqref{djlbot}$ and $\eqref{ddeltaibot}$.


\subsubsection{A new time evolution of $\langle \partial^{\delta_i}_{l-\delta_i}h,\partial^0_lh\rangle _{L^2_{x,v}}$}
By integrating by part in $x$ then in $v$ we obtain the following equality on the evolution of the scalar product.
$$\frac{d}{dt}\langle \partial^{\delta_i}_{l-\delta_i}h,\partial^0_lh\rangle _{L^2_{x,v}} = 2 \langle\partial^{\delta_i}_{l-\delta_i}G_\eps(h), \partial^0_lh\rangle_{L^2_{x,v}} + \frac{2}{\eps} \langle \partial^{\delta_i}_{l-\delta_i} \Gamma(h,h), \partial^0_lh\rangle_{L^2_{x,v}}.$$

We will bound above the first term as in the previous case and for the second term involving $\Gamma$ we use (H4) and Young's inequality with a constant $D_3 > 0$. Moreover, we decompose $\partial_l^0h$ into its fluid part and its microscopic part and we apply $\eqref{L2L2lambdafluid}$ on the fluid part. This yields
$$2\langle \partial^{\delta_i}_{l-\delta_i}\Gamma(h,h), \partial^0_lh\rangle_{L^2_{x,v}} \leq  D_3\left(\mathcal{G}^s_{x,v}(h,h)\right)^2 + \frac{1}{D_3}\norm{\partial^0_lh^\bot}^2_{\Lambda}+ \frac{C_\pi}{ D_3}\norm{\partial^0_lh}^2_{L^2_{x,v}}.$$

Finally we obtain an upper bound for the time-derivative:

\begin{eqnarray*}
\frac{d}{dt}\langle  \partial^{\delta_i}_{l-\delta_i}h, \partial^0_lh\rangle _{L^2_{x,v}} &\leq& \left[\frac{C^L\eta}{\eps^2}+ \frac{1}{D_3}\right] \norm{ \partial^0_lh^\bot}_{\Lambda}^2 + \frac{C^L}{\eta\eps^2}\norm{ \partial^{\delta_i}_{l-\delta_i}h}_{\Lambda}^2 + \left(\frac{C_\pi}{\eps D_3}-\frac{1}{\eps}\right)\norm{ \partial^0_lh}^2_{L^2_{x,v}} 
\\ &&+ \frac{D_3}{\eps}\left(\mathcal{G}^s_{x,v}(h,h)\right)^2.
\end{eqnarray*}

Now we can use the properties of $\pi_L$ concerning the $v$-derivatives, equation $\eqref{dvcontrolL}$, the equivalence of norm under the projection $\pi_L$, equation $\eqref{L2L2lambdafluid}$, and Poincare inequality get the following upper bound:

\begin{eqnarray*}
\norm{ \partial^{\delta_i}_{l-\delta_i}h}_{\Lambda}^2&\leq& 2\norm{ \partial^{\delta_i}_{l-\delta_i}h^\bot}_{\Lambda}^2 + 2\norm{ \partial^{\delta_i}_{l-\delta_i}\pi_L(h)}_{\Lambda}^2
\\ &\leq& 2\norm{ \partial^{\delta_i}_{l-\delta_i}h^\bot}_{\Lambda}^2 + 2C_{\pi s}C_\pi\norm{ \partial^0_{l-\delta_i}(h)}_{L^2_{x,v}}^2
\\ &\leq& 2\norm{ \partial^{\delta_i}_{l-\delta_i}h^\bot}_{\Lambda}^2 + 2C_{\pi s}C_\pi\sum\limits_{|l'|\leq s-1}\norm{ \partial^0_{l'}h}_{L^2_{x,v}}^2.
\end{eqnarray*}

Therefore,

\begin{eqnarray*}
\frac{d}{dt}\langle  \partial^{\delta_i}_{l-\delta_i}h, \partial^0_lh\rangle _{L^2_{x,v}} &\leq& \left[\frac{C^L\eta}{\eps^2}+\frac{1}{D_3}\right] \norm{ \partial^0_lh^\bot}_{\Lambda}^2 + \frac{2C^L}{\eta\eps^2}\norm{ \partial^{\delta_i}_{l-\delta_i}h^\bot}_{\Lambda}^2 + \left(\frac{C_\pi}{\eps D_3}-\frac{1}{\eps}\right)\norm{ \partial^0_lh}^2_{L^2_{x,v}} 
\\ && + \frac{2C^L C_{\pi s}C_\pi}{\eta\eps^2}\sum\limits_{|l'|\leq s-1}\norm{ \partial^0_{l'}h}_{L^2_{x,v}}^2 + \frac{D_3}{\eps}\left(\mathcal{G}^s_{x,v}(h,h)\right)^2.
\end{eqnarray*}
We finally define $\eta = 8eC^LC_{\pi s}C_\pi d/\eps$, with $e>1$, and $D_3=2C_\pi$ to yield equation $\eqref{deltai,d0lbot}$.

\section{Proof of the hydrodynamical limit lemmas}\label{appendix:hydro}

In this section we are going to prove all the different lemmas used in section $9$.
\par All along the demonstration we will use this inequality:

\begin{equation}\label{technicalineq}
\forall t>0, \:k\in \N^*, \:q \geq 0, \:p>0, \:\: t^qk^{2p}e^{-atk^2} \leq C_p(a) t^{q-p}.
\end{equation}


\subsection{Study of the linear part}

\subsubsection{Proof of Lemma $\ref{Ueps0j}$}
Fix $T$ in $[0,+\infty]$. By integrating we compute

\begin{eqnarray*}
\int_0^T U^\eps_{0j}h_{in} dt &=& \sum\limits_{n \in \Z^d-\{0\}} e^{in.x}\left[\int_0^Te^{\frac{i \alpha_j t \abs{n}}{\eps} - \beta_j t \abs{n}^2} dt\right] P_{0j} \left(\frac{n}{\abs{n}}\right) \hat{h}_{in}(n,v)
\\ &=& \sum\limits_{n \in \Z^d-\{0\}} e^{in.x}\frac{\eps}{i\alpha_j\abs{n} - \eps\beta_j \abs{n}^2}\left[e^{\frac{i \alpha_j T \abs{n}}{\eps} - \beta_j T \abs{n}^2}-1\right]P_{0j}\hat{h}_{in}(n,v).
\end{eqnarray*}

The Fourier transform is an isometry in $L^2_x$ and therefore
$$\norm{\int_0^T U^\eps_{0j}h_{in} dt}^2_{L^2_xL^2_v} \leq \eps^2 \sum\limits_{n \in \Z^d-\{0\}} \frac{2}{\alpha_j^2\abs{n}^2 + \eps^2\beta_j^2 \abs{n}^4} \norm{P_{0j}\left(\frac{n}{\abs{n}}\right)\hat{h}_{in}(n,\cdot)}^2_{L^2_v}.$$

Finally, we know that, like $e_{0j}$, $P_{0j}$ is continuous on the compact $\mathbb{S}^{d-1}$ and so is bounded. But the latter is a linear operator acting on $L^2_v$ and therefore it is bounded by $M_{0j}$ in the operator norm on $L^2_v$. Thus
\begin{eqnarray*}
\norm{\int_0^T U^\eps_{0j}h_{in} dt}^2_{L^2_xL^2_v} &\leq& \eps^2\frac{M_{0j}^2}{\alpha_j^2}\sum\limits_{n \in \Z^d-\{0\}}\norm{\hat{h}_{in}(n,\cdot)}^2_{L^2_v}
\\ &\leq& \eps^2\frac{M_{0j}^2}{\alpha_j^2}\norm{h_{in}(\cdot,\cdot)}^2_{L^2_xL^2_v},
\end{eqnarray*}
which is the expected result.

\bigskip
Now, let us look at the $L^2_x$-norm of this operator, to see how the torus case is different from the case $\R^d$ studied in \cite{BU} and \cite{EP}. 
\\Consider a direction $n_1$ in the Fourier transform space of the torus and define $\phi_{n_1} = \mathcal{F}^{-1}_x\left(e^{in_1}\right)$. We have the following equality
$$\langle U^\eps_{0j}h_{in},\phi_{n_1}\rangle_{L^2_x} = \langle\hat{U}^\eps_{0j}\hat{h}_{in},\hat{\phi}_{n_1}\rangle_{L^2_n} =  e^{\frac{i \alpha_j t \abs{n_1}}{\eps} - \beta_j t \abs{n_1}^2}P_{0j}\left(\frac{n_1}{\abs{n_1}}\right)\hat{h}_{in}(n_1,v). $$
If we do not integrate in time, one can easily see that this expression cannot have a limit as $\eps$ tends to $0$ if $P_{0j}\left(\frac{n_1}{\abs{n_1}}\right)\hat{h}_{in}(n_1,v) \neq 0$, and so we cannot even have a weak convergence. The difference with the whole space case is this possibility to single out one mode in the frequency space in the case of the torus. This leads to the possible existence of periodic function at a given frequency, the norm of which will never decrease. This is impossible in the case of a continuous Fourier space, as in $\R^d$, and well described by the Riemann-Lebesgue lemma.

\bigskip
Therefore we have a convergence without averaging in time if and only if $P_{0j}\left(\frac{n_1}{\abs{n_1}}\right)\hat{h}_{in}(n_1,v) = 0$, for all $j = \pm 1$ and all direction $n_1$. This means that for all $j = \pm 1$ and all $n_1$, $\langle e_{0j}\left(\frac{n_1}{\abs{n_1}}\right),\hat{h}_{in}\rangle_{L^2_v} =0$. By the expression known (see theorem $\ref{fourier}$) of $e_{0\pm 1}$, this is true if and only if  $\nabla_x\cdot u_{in} = 0$ and $\rho_{in} + \theta_{in} =0$.

\subsubsection{Proof of Lemma $\ref{Uepslj}$}
This lemma deals with three different terms and we study them one by one because they behaviour are quite different.

\bigskip
\paragraph{\textbf{The term $U^\eps_{1j}$:}}
We remind that we have
$$\hat{U}^\eps_{1j}\hat{h}_{in} = \chi_{\abs{\eps n}\leq n_0}e^{\frac{i \alpha_j t \abs{n}}{\eps} - \beta_j t \abs{n}^2}\left(e^{\frac{t}{\eps^2}\gamma_j(\abs{\eps n})} -1\right)P_{0j}\left(\frac{n}{\abs{n}}\right)\hat{h}_{in}(n,v).$$

If we take $T>0$, by Parseval identity we get
$$\norm{\int_0^T U^\eps_{1j}h_{in} dt}^2_{L^2_xL^2_v} = \sum\limits_{n\in \Z^d-\{0\}} \chi_{\abs{\eps n\leq n_0}}\abs{\int_0^Te^{\frac{i \alpha_j t \abs{n}}{\eps} - \beta_j t \abs{n}^2}\left(e^{\frac{t}{\eps^2}\gamma_j(\abs{\eps n})} -1\right)dt}^2\norm{P_{0j}\hat{h}_{in}}^2_{L^2_v}.$$

But then we can use the fact that $\abs{e^a-1}\leq \abs{a}e^{\abs{a}}$, the inequalites satisfied by $\gamma_j$ and the computational inequality $\eqref{technicalineq}$ to obtain

\begin{eqnarray*}
\abs{\int_0^Te^{\frac{i \alpha_j t \abs{n}}{\eps} - \beta_j t \abs{n}^2}\left(e^{\frac{t}{\eps^2}\gamma_j(\abs{\eps n})} -1\right)dt} &\leq& C_\gamma \eps \int^T_0 t \abs{n}^3 e^{-\frac{t\beta_j}{2}\abs{n}^2}dt
\\ &\leq& C_\gamma \eps C_{3/2}\left(\frac{\beta_j}{4}\right)\int^T_0\frac{1}{\sqrt{t}}e^{-\frac{t\beta_j}{4}\abs{n}^2}dt
\\ &\leq&  C_\gamma \eps C_{3/2}\left(\frac{\beta_j}{4}\right)\int^{+\infty}_0\frac{1}{\sqrt{t}}e^{-\frac{t\beta_j}{4}}dt,
\end{eqnarray*}
which is independent of $n$ and is written $I \eps$. Therefore we have the expected inequality, by using the continuity of $P_{0j}$,
$$\norm{\int_0^T U^\eps_{1j}h_{in} dt}^2_{L^2_xL^2_v} \leq \eps^2 I^2 M_{0j}^2\norm{h_{in}}^2_{L^2_xL^2_v}.$$

\bigskip
The last two inequalities we want to show comes from Parseval's identity, the properties of $\gamma_j$ and the computational inequality $\eqref{technicalineq}$:

\begin{eqnarray}
\norm{ U^\eps_{1j}h_{in}}^2_{L^2_xL^2_v} &=& \sum\limits_{n \in \Z^d-\{0\}} \chi_{\abs{\eps n}\leq n_0} e^{-2\beta_j t \abs{n}^2}\abs{e^{\frac{t}{\eps^2}\gamma_j{\abs{\eps n}}}-1}^2\norm{P_{0j}\left(\frac{n}{\abs{n}}\right)\hat{h}_{in}}^2_{L^2_v} \nonumber
\\ &\leq& M_{0j}^2C_\gamma^2\eps^2 \sum\limits_{n \in \Z^d-\{0\}} \chi_{\abs{\eps n}\leq n_0} t^2 \abs{n}^6e^{-\beta_j t \abs{n}^2}\norm{\hat{h}_{in}}^2_{L^2_v} \nonumber
\\ &\leq& M_{0j}^2C_\gamma^2\eps^2 C_2\left(\frac{\beta_j}{2}\right)\sum\limits_{n \in \Z^d-\{0\}} \chi_{\abs{\eps n}\leq n_0} \abs{n}^2 e^{-\frac{\beta_j t}{2}\abs{n}^2}\norm{\hat{h}_{in}}^2_{L^2_v}. \label{ineq1j}
\end{eqnarray}

Finally, if we integrate in $t$ between $0$ and $+\infty$ we obtain the expected second inequality of the lemma. If we merely bound $e^{-\frac{\beta_j t}{2}\abs{n}^2}$ by one and use the fact that $\chi_{\abs{\eps n}\leq n_0} \leq 1$ and  $\chi_{\abs{\eps n}\leq n_0} \eps^2 \abs{n}^2 \leq n_0^2$ we obtain the third inequality of the lemma for $\delta =1$ and $\delta = 0$. Then by interpolation we obtain the general case for $0 \leq \delta \leq 1$.

\bigskip
\paragraph{\textbf{The term $U^\eps_{2j}$:}}
Fix $T>0$. By Parseval's identity we have

\begin{eqnarray*}
\norm{\int_0^T U^\eps_{2j}h_{in} dt}^2_{L^2_xL^2_v} &=& \sum\limits_{n\in \Z^d-\{0\}} \chi_{\abs{\eps n}\leq n_0}\abs{\int_0^Te^{\frac{i \alpha_j t \abs{n}}{\eps} - \beta_j t \abs{n}^2 + \frac{t}{\eps^2}\gamma_j(\abs{\eps n})}dt }^2\abs{\eps n}^2 \norm{\tilde{P}_{1j}\hat{h}_{in}}^2_{L^2_v}
\\ &\leq&   \sum\limits_{n\in \Z^d-\{0\}} \frac{4}{\beta_j^2 \abs{n}^4} \abs{\eps n}^2 \norm{\tilde{P}_{1j}\left(\abs{\eps n}, \frac{n}{\abs{n}}\right)\hat{h}_{in}}^2_{L^2_v},
\end{eqnarray*}
where we used the inequalities satisfied by $\gamma$ and integration in time. 
\\ Then, $\tilde{P}_{1j}$ is continuous on the compact $[-n_0,n_0] \times \mathbb{S}^{d-1}$ and so is bounded, as an operator acting on $L^2_v$, by $M_{1j}>0$. Hence, Parseval's identity offers us the first inequality of the lemma.

\bigskip
The last two inequalities are just using Parseval's identity and the continuity of $\tilde{P}_{1j}$. Indeed,

\begin{eqnarray*}
\norm{ U^\eps_{2j}h_{in}}^2_{L^2_xL^2_v} &=& \sum\limits_{n\in \Z^d-\{0\}} \chi_{\abs{\eps n}\leq n_0}\abs{e^{\frac{i \alpha_j t \abs{n}}{\eps} - \beta_j t \abs{n}^2 + \frac{t}{\eps^2}\gamma_j(\abs{\eps n})}}^2\abs{\eps n}^2 \norm{\tilde{P}_{1j}(n)\hat{h}_{in}}^2_{L^2_v}
\\ &\leq& M_{1j}^2 \eps^2 \sum\limits_{n\in \Z^d-\{0\}} \chi_{\abs{\eps n}\leq n_0}\abs{n}^2 e^{-t\beta_j \abs{n}^2} \norm{\hat{h}_{in}}^2_{L^2_v}.
\end{eqnarray*}

We recognize here the same form of inequality  $\eqref{ineq1j}$. Thus, we obtain the last two inequalities of the statement in the same way.

\bigskip
\paragraph{\textbf{The term $U^\eps_{3j}$:}}
We remind the reader that
$$\hat{U}^\eps_{3j} = \left(\chi_{\abs{\eps n}\leq n_0} - 1\right)e^{\frac{i \alpha_j t \abs{n}}{\eps} - \beta_j t \abs{n}^2}P_{0j}\left(\frac{n}{\abs{n}}\right).$$
We have the following inequality
$$\abs{\chi_{\abs{\eps n}\leq n_0} - 1} \leq \frac{\eps n}{n_0}.$$
Therefore, replacing $\tilde{P}_{1j}$ by $\frac{1}{n_0}P_{0j}$ and $\beta_j$ by $2\beta_j$ (since $\frac{t}{\eps^2}\gamma_j(\abs{\eps n}) \leq \frac{t\beta_j}{2}\abs{n}^2$) in the proof made for $U^\eps_{2j}$ we obtain the expected three inequalities for $U^\eps_{3j}h_{in}$, the last one only with $\delta=1$. 
\\ To have the last inequality in $\delta$, it is enough to bound $\abs{\chi_{\abs{\eps n}\leq n_0} - 1}$ by $1$ and then using the continuity of $P_{0j}$ to have the result for $\delta = 0$. Finally, we interpolate to get the general result for all $0\leq \delta \leq 1$.


\subsubsection{Proof of Lemma $\ref{UR}$}

Thanks to Theorem $\ref{fourier}$ we have that
$$\norm{U^\eps_Rh_{in}}^2_{L^2_xL^2_v} = \norm{\hat{U}_R(t/\eps^2,\eps n,v)\hat{h}_{in}}^2_{L^2_nL^2_v} \leq C_R^2 e^{-2\frac{\sigma t}{\eps^2}}\norm{h_{in}}^2_{L^2_xL^2_v}.$$
But then we have, thanks to the technical lemma $\ref{technicalineq}$, that $e^{-2\frac{\sigma t}{\eps^2}} \leq C_{1/2}(2\sigma)\frac{\eps}{\sqrt{t}}$, which gives us the last two inequalities we wanted.
For the first inequality, a mere Cauchy-Schwartz inequality yields
$$\norm{\int_0^T U^\eps_R h_{in} dt}^2_{L^2_xL^2_v} \leq T \int_0^T \norm{U^\eps_R h_{in} }^2_{L^2_xL^2_v} dt,$$
which gives us the first inequality by integrating in $t$.

\bigskip
Now, let us suppose that we have the strong convergence down to $t=0$. At $t=0$ we can write that $e^{tG_\eps} = \mbox{Id}$ and therefore that:
$$\mbox{Id}= \chi_{\abs{\eps n}\leq n_0}\sum\limits_{j=-1}^{2}P_{j}\left(\abs{\eps n}, \frac{n}{\abs{n}}\right)+ \hat{U}_R(0,\eps n,v).$$
We have the strong convergence down to $0$ as $\eps$ tends to $0$. Therefore, taking the latter equality at $\eps = 0$ we have, because $\sum\limits_{j=-1}^{2}P_{0j} = \pi_L$,
$$\hat{U}_R(0,0,v) = \mbox{Id} - \pi_L .$$
Then $\hat{U}_R\hat{h}_{in}$ tends to $0$ as $\eps$ tends to $0$ in $C([0,+\infty),L^2_xL^2_v)$ if and only if $h_{in}$ belongs to $\mbox{Ker}(L)$.
\\ In that case, we can use the proof of Lemma $6.2$ of \cite{BU} in which they noticed that
$$U^\eps_R(t,x,v) = e^{tG_{\eps}}U^\eps_R(0,x,v) = e^{tG_{\eps}} \left[\mathcal{F}^{-1}_x \left(Id - \chi_{\abs{\eps n}\leq n_0}\sum\limits_{j=-1}^2P_j(\eps n)\right)\mathcal{F}_x\right].$$
Thanks to that new form we have that, if $h_{in} = \pi_L(h_{in})$,
$$U^\eps_R(t,x,v)h_{in} = e^{tG_{\eps}} \left[\mathcal{F}^{-1}_x \left((1-\chi_{\abs{\eps n}\leq n_0}) - \abs{\eps n}\chi_{\abs{\eps n}\leq n_0}\sum\limits_{j=-1}^{2}\tilde{P}_{1j}(\eps n)\right)\hat{h}_{in}\right],$$
because $\pi_L = \sum\limits_{j=-1}^2P_{0j}$.
\\ Therefore we can redo the same estimates we worked out in the previous lemmas and use the same interpolation method to get the result stated in Lemma $\ref{UR}$.


\subsection{Study of the bilinear part}

\subsubsection{A simplification without loss of generality}
All the terms we are about to study, apart from the remainder term, are of the following form

$$\psi^\eps_{ij}(u_\eps) = \int_0^t\sum\limits_{n \in \Z^d-\{0\}}g(t,s,k,x)P(n)\hat{u}_\eps(s,k,v)ds,$$
with $P(n)$ being a projector in $L^2_v$, bounded uniformely in $n$.

\bigskip
Looking at the dual definition of the norm of a function in $L^2_{x,v}$, we can consider $f$ in $C^\infty_c\left(\T^d \times \R^d\right)$ such that $\norm{f}_{L^2_{x,v}}=1$ and take the scalar product with $\psi^\eps_{ij}(u_\eps)$. This yields, since $P$ is a projector and thus symmetric,

\begin{eqnarray}
\langle\psi^\eps_{ij}(u_\eps),f\rangle_{L^2_{x,v}} &=& \int_{\T^d}\int_0^t\sum\limits_{n \in \Z^d-\{0\}}g(t,s,k,x)\langle P(n)\hat{u}_\eps,f\rangle_{L^2_v}ds \nonumber
\\  &=& \int_{\T^d}\int_0^t\sum\limits_{n \in \Z^d-\{0\}}g(t,s,k,x)\langle\hat{u}_\eps,P(n)f\rangle_{L^2_v}ds. \label{rewrite}
\end{eqnarray}

We are working in $L^2_xL^2_v$ in order to simplify computations as they are exactly the same in higher Sobolev spaces. Therefore, we can assume that hypothesis $(H4)$ is still valid in $L^2_v$ without loss of generality. This means

\begin{equation} \label{bound}
\langle\hat{u}_\eps,P(n)f\rangle_{L^2_v} \leq \norm{h}_{L^2_xL^2_v}\norm{h}_{\Lambda_v}\norm{P(n)f}_{\Lambda_v}.
\end{equation}

Finally, in terms of Fourier coefficients in $x$, $P(n)$ is a projector in $L^2_v$ and uniformely bounded in $n$ as an operator in $L^2_v$.
\\ Thus, combining $\eqref{bound}$ and the definition of the functional $E$, $\eqref{defE}$, we see that
$$\int_0^T\norm{\hat{f}_\eps}^2_{L^2_xL^2_v}dt$$
is a continuous operator from $C(\R^+,L^2_xL^2_v,E(\cdot))$ to $C(\R^+,L^2_xL^2_v,\norm{\cdot}_{L^2_xL^2_v})$.  Looking at $\eqref{rewrite}$, we can consider without loss of generality that the following holds (even for the remainder term) for all $T>0$:

$$\psi^\eps_{ij}(u_\eps) = \int_0^t\sum\limits_{n \in \Z^d-\{0\}}g(t,s,k,x)\hat{f}_\eps(s,k,v)ds,$$
with 
$$\int_0^T\norm{\hat{f}_\eps}^2_{L^2_xL^2_v}dt \leq M_{ij} E(h_\eps)^2.$$

\subsubsection{Proof of Lemma $\ref{psieps0j}$}
For the first inequality, fix $T>0$ and integrate by part in $t$ to obtain

\begin{eqnarray*}
\int_0^T \psi^\eps_{0j}(u_\eps)dt &=& \sum\limits_{n \in \Z^d-\{0\}}e^{in.x}\int_0^T\left(\int_0^te^{i\frac{\alpha_j (t-s)}{\eps}\abs{n}- (t-s)\beta_j\abs{n}^2}\abs{n}\hat{f}_\eps(s)ds\right)dt
\\ &=& \sum\limits_{n \in \Z^d-\{0\}}e^{in.x} \frac{\eps}{i\alpha_j\abs{n}-\eps\beta_j\abs{n}}\left[\int_0^T\left(e^{i\frac{\alpha_j (T-s)}{\eps}\abs{n}-(T-s)\beta_j\abs{n}^2} - 1\right)\hat{f}_\eps(s)ds\right].
\end{eqnarray*}

Finally we can use Parseval's identity 

\begin{eqnarray*}
\norm{\int_0^T \psi^\eps_{0j}(u_\eps)dt}^2_{L^2_xL^2_v} &\leq& \sum\limits_{n \in \Z^d-\{0\}} \frac{\eps^2}{\eps^2\beta_j^2\abs{n}^2 + \alpha_j^2}T\int_0^T 2 \norm{\hat{f}_\eps(s,n,v)}^2_{L^2_v}ds
\\ &\leq& \frac{2M_{1j}^2}{\alpha_j^2}T \eps^2 E(h_\eps)^2,
\end{eqnarray*}
where we used the subsection above and Parseval's identity again. This is exactly the expected result.


\subsubsection{Proof of Lemma $\ref{psiepslj}$}
We divide this proof in three paragraphes, each of them studying a different term.

\paragraph{\textbf{The term $\psi^\eps_{1j}$:}}
We will just prove the last two inequalities and then merely applying Cauchy-Schwarz inequality will lead to the first one.
\\Fix $t>0$. By a change of variable we can write
$$\psi^\eps_{1j}(u_\eps) = \sum\limits_{n \in \Z^d-\{0\}}e^{ik.x}\chi_{\abs{\eps n}\leq n_0}\int_0^te^{\frac{i\alpha_js}{\eps}\abs{n}-\beta_js\abs{n}^2}\left(e^{\frac{s}{\eps^2}\gamma_j(\abs{\eps n})}-1\right)\abs{n}\hat{f}_\eps(t-s)ds.$$
By the study made in the proof of Lemma $\ref{Uepslj}$ we have that

\begin{eqnarray*}
&&\abs{\int_0^te^{\frac{i\alpha_js}{\eps}\abs{n}-\beta_js\abs{n}^2}\left(e^{\frac{s}{\eps^2}\gamma_j(\abs{\eps n})}-1\right)\abs{n}\hat{f}_\eps(t-s)ds} 
\\ && \hspace{4cm}\leq C_\gamma\abs{n}^4\eps \int_0^tse^{-\frac{\beta_js}{2}\abs{n}^2}\abs{\hat{f}_\eps(t-s)}ds.
\end{eqnarray*}

Then we use the computational inequality $\eqref{technicalineq}$ and a Cauchy-Schwarz to obtain

\begin{eqnarray}
&& \abs{\int_0^te^{\frac{i\alpha_js}{\eps}\abs{n}-\beta_js\abs{n}^2}\left(e^{\frac{s}{\eps^2}\gamma_j(\abs{\eps n})}-1\right)\abs{n}\hat{u}_\eps ds} \nonumber
\\ && \hspace{4cm}\leq \eps C_\gamma  C_1\left(\frac{\beta_j}{4}\right)\abs{n}^2 \int_0^te^{-\frac{\beta_js}{4}\abs{n}^2}\abs{\hat{f}_\eps}ds  \nonumber
\\ && \hspace{4cm} \leq  \eps C_\gamma C_1\left(\frac{\beta_j}{4}\right)\abs{n}^2 \sqrt{\frac{4}{\beta_j\abs{n}^2}}\left[\int_0^te^{-\frac{\beta_j(t-s)}{4}\abs{n}^2}\abs{\hat{f}_\eps}^2ds\right]^{1/2}. \label{ineqpsi1j}
\end{eqnarray}

We can obtain the result by using Parseval's identity, denoting $C$ a constant independent of $\eps$ and $T$, the continuity of $P_{1j}$ and the computational inequality $\eqref{technicalineq}$.

$$\norm{\psi^\eps_{1j}(u_\eps)(t)}^2_{L^2_xL^2_v} \leq C\sum\limits_{n \in \Z^d-\{0\}}\chi_{\abs{\eps n}\leq n_0}\eps^2\abs{n}^2\int_0^te^{-\frac{\beta_j(t-s)s}{4}\abs{n}^2}\norm{\hat{f}_\eps(s)}^2_{L^2_v}ds.$$

If we merely bound $e^{-\frac{\beta_j (t-s)}{2}\abs{n}^2}$ by one and use the fact that $\chi_{\abs{\eps n}\leq n_0} \leq 1$ and  $\chi_{\abs{\eps n}\leq n_0} \eps^2 \abs{n}^2 \leq n_0^2$ we obtain the third inequality of the lemma for $\delta =1$ and $\delta = 0$. Then by interpolation we obtain the general case for $0 \leq \delta \leq 1$.
\par If we integrate in $t$ between $0$ and a fixed $T>0$, a mere integration by part yields the expected control on the $L^2_{t,x,v}$-norm. Finally, from the latter control and a Cauchy-Schwarz inequality we deduce the first inequality.

\bigskip
\paragraph{\textbf{The term $\psi^\eps_{2j}$:}}
As in the case $\psi^\eps_{1j}$, we are going to prove the third inequality only.
\\Fix $T>0$, a change of variable gives us
$$\psi^\eps_{2j}(u_\eps) = \sum\limits_{n \in \Z^d-\{0\}}e^{ik.x}\chi_{\abs{\eps n}\leq n_0}\int_0^Te^{\frac{i\alpha_js}{\eps}\abs{n}-\beta_js\abs{n}^2 +\frac{s}{\eps^2}\gamma_j(\abs{\eps n})}\eps\abs{n}^2\hat{f}_\eps(T-s)ds.$$
We can see that
$$\abs{\int_0^Te^{\frac{i\alpha_js}{\eps}\abs{n}-\beta_js\abs{n}^2 +\frac{s}{\eps^2}\gamma_j(\abs{\eps n})}\eps\abs{n}^2\hat{f}_\eps(T-s)ds} \leq \eps\abs{n}^2\int^T_0 e^{-\frac{\beta_js}{2}\abs{n}^2}\abs{\hat{f}_\eps(T-s)}ds.$$
This bound is of the same form as equation $\eqref{ineqpsi1j}$. Therefore we have the same result.

\bigskip
\paragraph{\textbf{The term $\psi^\eps_{3j}$:}}
As above, we will show the third inequality only.
\\Fix $T>0$, we can write
$$\psi^\eps_{3j}(u_\eps) = \sum\limits_{n \in \Z^d-\{0\}}e^{ik.x}(\chi_{\abs{\eps n}\leq n_0}-1)\int_0^Te^{\frac{i\alpha_js}{\eps}\abs{n}-\beta_js\abs{n}^2}\abs{n} \hat{f}_\eps(T-s,n,v)ds.$$
Looking at the fact that $\abs{\chi_{\abs{\eps n}\leq n_0}-1} \leq \frac{\eps \abs{n}}{n_0}$, we find the same kind of inequality as equation $\eqref{ineqpsi1j}$. Thus, we reach the same result.


\subsubsection{Proof of Lemma $\ref{psiepsR}$}
We remind the reader that
$$\Psi^\eps_R( u_\eps) = \int_0^t \frac{1}{\eps}U^\eps_R(t-s)f_\eps(s) ds,$$
and that, by Theorem $\ref{fourier}$,
$$\norm{U^\eps_Rf_\eps}^2_{L^2_xL^2_v} \leq C_R^2 e^{-2\frac{\sigma t}{\eps^2}}\norm{f_\eps}^2_{L^2_xL^2_v}.$$
Hence, a Cauchy-Schwarz inequality gives us the third inequality for $\norm{\psi^\eps_{R}(u_\eps)(T)}^2_{L^2_xL^2_v}$, and then the two others inequality stated above.


\subsubsection{Proof of Lemma $\ref{psiuepsu}$}
We remind the reader that
$$\Psi(u) = \mathcal{F}^{-1}_x\left[\psi^\eps_{00} (u)+\psi^\eps_{02} (u)\right]\mathcal{F}_x.$$
As above, and because in that case $\alpha_j =0$, we can write $\psi^\eps_{0j} (u_\eps -u)(T)$, for some $T>0$, and apply a Cauchy-Schwarz inequality:

\begin{eqnarray*}
\norm{\psi^\eps_{0j} (u_\eps - u)}^2_{L^2_xL^2_v}(T) &=& \sum\limits_{n \in \Z^d-\{0\}} \abs{n}^2\int_{\R^d}\abs{\int_0^Te^{-s\beta_j\abs{n}^2}P_{1j}\hat{\Gamma}(h_\eps -h,h_\eps + h)ds}^2dv
\\ &\leq& \frac{M_{1j}^2}{\beta_j^2} \underset{t\in [0,T]}{\mbox{Sup}} \norm{\Gamma(h_\eps -h,h_\eps + h)}^2_{L^2_xL^2_v}.
\end{eqnarray*}

But because $\T^d$ is bounded in $\R^d$ and thanks to (H4) and the boundedness of $(h_\eps)_\eps$ and $h$ (both bounded by $M$) in $H^s_xL^2_v$ (Theorem $\ref{perturb}$), we can have the following control:
$$\norm{\Gamma(h_\eps -h,h_\eps + h)}^2_{L^2_xL^2_v} \leq 4M^2 C_\Gamma^2  \mbox{Volume}(\T^d)\norm{h_\eps - h}_{L^\infty_xL^2_v}.$$

Therefore we obtain the last inequality and the first two just come from Cauchy-Schwarz inequality.

\bibliographystyle{acm}
\bibliography{bibliography}


\signmb

\end{document}